\documentclass[notitlepage,11pt,reqno]{amsart}  
\usepackage[foot]{amsaddr}
\usepackage[numbers,sort]{natbib}
\usepackage{amssymb,nicefrac,bm,upgreek,mathtools,verbatim}
\usepackage[final]{hyperref}
\usepackage[mathscr]{eucal}
\usepackage{dsfont} 
\usepackage{graphicx}

\usepackage[margin=10pt,font=small,labelfont=bf]{caption}
\usepackage[margin=1in]{geometry} 
\usepackage{enumitem}
\usepackage{yhmath}
\usepackage{relsize}
\usepackage{float}
\usepackage{caption}
\usepackage{subcaption}
\usepackage{mathrsfs}

\usepackage[ruled,vlined]{algorithm2e}

\allowdisplaybreaks  

\usepackage{color}  
\definecolor{dmagenta}{rgb}{.4,.1,.5}       
\definecolor{dblue}{rgb}{.0,.0,.5}     
\definecolor{mblue}{rgb}{.0,.0,.8}     
\definecolor{ddblue}{rgb}{.0,.0,.4}            
\definecolor{dred}{rgb}{.6,.0,.0}   
\definecolor{dgreen}{rgb}{.0,.5,.0}  
\definecolor{Eeom}{rgb}{.0,.0,.5}

\usepackage[normalem]{ulem}

\newtheorem{lemma}{Lemma}[section]
\newtheorem{theorem}{Theorem}[section]
\newtheorem{proposition}{Proposition}[section]

\theoremstyle{definition}

\newtheorem{assumption}{Assumption}[section]

\theoremstyle{remark}
\newtheorem{remark}{Remark}[section]
\numberwithin{equation}{section}

\hypersetup{
  colorlinks=true,
  citecolor=dblue,
  linkcolor=dblue,
  frenchlinks=false,
  pdfborder={0 0 0},
  naturalnames=false, 
  hypertexnames=false,
  breaklinks}

\usepackage[capitalize,nameinlink]{cleveref}

\crefname{section}{Section}{Sections}
\crefname{subsection}{Section}{Sections}
\crefname{condition}{Condition}{Conditions}
\crefname{hypothesis}{Hypothesis}{Conditions}
\crefname{assumption}{Assumption}{Assumptions} 
\crefname{lemma}{Lemma}{Lemmas}

\Crefname{figure}{Figure}{Figures}

\crefformat{equation}{\textup{#2(#1)#3}}
\crefrangeformat{equation}{\textup{#3(#1)#4--#5(#2)#6}}
\crefmultiformat{equation}{\textup{#2(#1)#3}}{ and \textup{#2(#1)#3}}
{, \textup{#2(#1)#3}}{, and \textup{#2(#1)#3}}
\crefrangemultiformat{equation}{\textup{#3(#1)#4--#5(#2)#6}}%
{ and \textup{#3(#1)#4--#5(#2)#6}}{, \textup{#3(#1)#4--#5(#2)#6}}%
{, and \textup{#3(#1)#4--#5(#2)#6}}

\Crefformat{equation}{#2Equation~\textup{(#1)}#3}
\Crefrangeformat{equation}{Equations~\textup{#3(#1)#4--#5(#2)#6}}
\Crefmultiformat{equation}{Equations~\textup{#2(#1)#3}}{ and \textup{#2(#1)#3}}
{, \textup{#2(#1)#3}}{, and \textup{#2(#1)#3}}
\Crefrangemultiformat{equation}{Equations~\textup{#3(#1)#4--#5(#2)#6}}%
{ and \textup{#3(#1)#4--#5(#2)#6}}{, \textup{#3(#1)#4--#5(#2)#6}}%
{, and \textup{#3(#1)#4--#5(#2)#6}}

\crefdefaultlabelformat{#2\textup{#1}#3}

\newcommand{\cE}{{\mathcal{E}}}

\newcommand{\QQ}{\mathbb{Q}}
\newcommand{\R}{\mathbb{R}}

\newcommand{\beql}[1]{\begin{equation}\label{#1}}

\newcommand{\beq}{\begin{displaymath}}
\newcommand{\eeqno}{\end{displaymath}}
\newcommand{\eeq}{\end{equation}}

\newcommand{\RR}{\mathbb{R}}
\newcommand{\NN}{\mathds{N}}
\newcommand{\ZZ}{\mathds{Z}}

\newcommand{\E}{\mathbb{E}}

\newcommand{\Uadm}{\mathfrak{U}}
\newcommand{\Usm}{\mathfrak{U}_{\mathrm{SM}}}

\newcommand{\Ind}{\mathds{1}}

\newcommand{\bU}{\mathbb{U}}

\newcommand{\norm}[1]{\lVert#1\rVert}

\newcommand{\PP}{\mathbb{P}}
\newcommand{\EE}{\mathbb{E}}

\newcommand{\calP}{\mathcal{P}}

\newcommand{\calX}{\mathcal{X}}

\makeatletter
\DeclareRobustCommand\widecheck[1]{{\mathpalette\@widecheck{#1}}}
\def\@widecheck#1#2{%
    \setbox\z@\hbox{\m@th$#1#2$}%
    \setbox\tw@\hbox{\m@th$#1%
       \widehat{%
          \vrule\@width\z@\@height\ht\z@
          \vrule\@height\z@\@width\wd\z@}$}%
    \dp\tw@-\ht\z@
    \@tempdima\ht\z@ \advance\@tempdima2\ht\tw@ \divide\@tempdima\thr@@
    \setbox\tw@\hbox{%
       \raise\@tempdima\hbox{\scalebox{1}[-1]{\lower\@tempdima\box
\tw@}}}%
    {\ooalign{\box\tw@ \cr \box\z@}}}
\makeatother

\usepackage{accents}
\newlength{\dhatheight}

\newcommand\scalemath[2]{\scalebox{#1}{\mbox{\ensuremath{\displaystyle #2}}}}

\let\oldtocsection=\tocsection

\let\oldtocsubsection=\tocsubsection

\let\oldtocsubsubsection=\tocsubsubsection

\renewcommand{\tocsection}[2]{\hspace{0em}\oldtocsection{#1}{#2}}
\renewcommand{\tocsubsection}[2]{\hspace{1em}\oldtocsubsection{#1}{#2}}
\renewcommand{\tocsubsubsection}[2]{\hspace{2em}\oldtocsubsubsection{#1}{#2}}

				\graphicspath{ {./images/} }
				
				\newcommand{\ttl}{\Large Jacobi-like relative value iteration algorithms for\\[5pt]  ergodic  risk-sensitive control of {M}arkov chains
					}

\begin{document}
					
					\title[RVI algorithms for ergodic risk-sensitive control of Markov chains]{\ttl}

					\author[Sumith \ Reddy]{Sumith Reddy Anugu$^*$}
					\author[Guodong \ Pang]{Guodong Pang$^*$}
					\author[Nicola \ Sassone]{Nicola G. Sassone$^*$}
					\address{$^*$Department of Computational Applied Mathematics and Operations Research,
						George R. Brown School of Engineering and Computing, 
						Rice University,
						Houston, TX 77005}
					\email{sa167, gdpang, ngs6@rice.edu}

					\date{July 7, 2026}
					
					\allowdisplaybreaks

					\begin{abstract} 
					We propose a Jacobi-like relative value iteration (RVI) algorithm and a Gauss-Seidel-like implementation for the ergodic risk-sensitive control (ERSC) problem of a controlled discrete-time Markov chain (DTMC) on
					a finite state space. 
					Under the assumption that the DTMC is irreducible and recurrent under every stationary Markov policy, we prove that the iterates of the proposed RVI algorithms converge at a geometric rate.
					The main challenge stems from the multiplicative structure of the ERSC cost criterion and the associated  Bellman-like operators, which prevents us from adapting the analogous global contraction and bi-Lipschitz continuity properties that underlie  the proof of convergence in the average cost setting.
					We overcome this by establishing local contraction properties for the risk-sensitive Bellman-like operators and a local bi-Lipschitz continuity property for their fixed points, and use these properties to show the iterates converge geometrically. 
					We conclude by implementing our proposed RVI algorithms on two examples: service-effort control for a single-server queue of finite capacity, and maximizing the exit rate from a finite domain (on a  graph). 
 \end{abstract}

					\keywords{Ergodic risk-sensitive control problem, Markov chains, Jacobi-like and Gauss-Seidel-like relative value iteration (RVI) algorithms, convergence analysis. 
					}

			\maketitle
			
		\section{Introduction}
								
		Risk-sensitive control has been an important and active area of research because it allows decision-makers to optimize beyond mean-based criteria and account for higher order fluctuations in the objective cost; see the recent surveys~\citep{Baurle2023_survey,BISWAS2023118}. Due to this, risk-sensitive control has been investigated across a variety of  domains, including mathematical finance~\citep{Bielecki_RS_Dynamic_manage_1999,Bielecki_Cox_2005,Davis_Math_Finance_2019,Fleming_opt_growth_RS_1999,Nagai_RS_Portfolio_2002}, inventory planning~\citep{Bouakiz_inventory_exp_1992}, revenue management~\citep{Barz_RS_revenue_2007,Feng_RS_perishable_2008}, robust control theory~\citep{Dupuis_robust_RS_2000,Noorani_RS_RL_2025,Peterson_min_max_entropy_2000}, and reinforcement learning~\citep{Basu_RS_Q_learning_2008,Borkar_2002_Q_learning,Moharrami_policy_grad_exp_2024,Noorani_RS_RL_2025,WuXuRSMDPUtility2023}.
		Risk-sensitive control problems have been formulated in a variety of ways, including through mean-variance criteria~\citep{Sobel1994meanvar}, tail-based objectives~\citep{Li2022quantile_mdp}, coherent risk measures such as conditional value-at-risk~\citep{Chow2015risk}, and utility-based formulations~\citep{Baurle2023_survey,Jaquette1976utility}. Among the utility-based formulations, the use of an exponential utility function has been attractive because it obeys the principle of time consistency and is mathematically tractable for dynamic programming arguments.	In the infinite time horizon setting, the ergodic risk-sensitive criterion is particularly useful for problems in which one seeks to optimize long-run performance while still accounting for fluctuations in the cumulative cost~\citep{Baurle2023_survey, BISWAS2023118}. 
								
		In this work, we consider the ergodic risk-sensitive control (ERSC) problem  of a controlled discrete-time Markov chain (DTMC) $X = \{X_t\}_{t=0}^{\infty}$ on a finite state space with an associated transition probability $p(i,j,u)=\PP(X_1=j|X_0=i, u)$, where our objective is to minimize the following criterion for a risk-sensitivity parameter $\delta>0$:
					\begin{equation}\label{eq-intro_cost_func} 
						\limsup_{T\rightarrow \infty}\frac{1}{\delta T}\log \EE\bigg[e^{\delta \sum_{t=0}^{T-1}c(X_t, U_t)}|X_0=i\bigg]
					\end{equation}
					over  $U$ which lies in an appropriate class of control policies.
					One notable feature of this problem is that {it} takes into account all higher order moments of the total accumulated cost $ \sum_{t=0}^{T-1}c(X_t, U_t)$,  
					due to the exponential in~\eqref{eq-intro_cost_func}. This is in contrast to the average cost control problem, where the minimization criterion (over an appropriate class of controls) is given by 
					\begin{equation*}
						\limsup_{T\rightarrow \infty}\frac{1}{ T} \EE\bigg[ \sum_{t=0}^{T-1}c(X_t, U_t)|X_0=i\bigg]\,.
					\end{equation*}
					In other words, the average cost problem only concerns  with the first moment (the expectation) of the accumulated cost. Hence, the average cost control problem is also commonly referred to as the risk-neutral problem - this terminology is further aided by the fact that the sequence of optimal values of the ERSC problem  converges to the optimal value of the average cost problem, as $\delta\to 0$.  
				
				  The primary objective of the ERSC problem is to  characterize and then find  optimal control policies.
				  It is well known that in the case of finite state space, under the assumption that the DTMC is irreducible and recurrent for all stationary control policies, there exist optimal policies which are stationary. Furthermore, these optimal stationary Markov policies can be completely characterized by the multiplicative
				  Bellman equation.
				  More precisely,  if $(V^*,\lambda^*)$ is the solution pair of the Bellman equation (where $\lambda^*$ turns out to be the optimal ERSC cost and $V^*$ is referred to as the value function),  then a stationary Markov policy is optimal if and only if it is a minimizer of the Bellman equation given below (see Theorem~\ref{thm-hjb}):
				  \begin{equation}\label{eq-intro_ERSC_Bellman}
				  	V^*(i) = \min_{u }\Big(e^{\delta (c(i,u)-\lambda^*)}\sum_{j}V^*(j)p(i,j,u)\Big),
				  \end{equation}
				  where $i$ varies over the state space and the minimization is over an appropriate control set that possibly depends on $i$. 
				  However, computing the minimizers of the Bellman equation requires the knowledge of its solution pair $(V^*,\lambda^*)$ which is not readily available. Therefore, it is a common practice to turn towards numerical algorithms to compute an optimal  stationary Markov policy. {These numerical algorithms typically fall into two categories: (i)	relative value iteration (RVI), and (ii) policy iteration. In this paper, we focus on RVI algorithms for the ERSC problem, and discuss policy iteration algorithms in Section~\ref{sec-lit}}.
				  
				  For a fixed reference state $n$, the existing ERSC RVI algorithm~\citep{Bielecki_1999,Borkar_2002,Cavazos_2003} proceeds as follows. Given an iterate for the value function $V^k$, the next optimal cost iterate $\lambda^{k+1}$ is computed by evaluating the RHS of \eqref{eq-intro_ERSC_Bellman} (without $\lambda^*$) at state $n$, with $V^{k}$ in place of $V^*$.
				  The next value function iterate $V^{k+1}$ is then computed by  evaluating the RHS of \eqref{eq-intro_ERSC_Bellman} at each state, with $(V^k,\lambda^{k+1})$ in place of $(V^*,\lambda^*)$; see the iteration \eqref{eq-alg-existing}.
				  This RVI algorithm can be seen as a multiplicative analogue of the average cost RVI algorithm first proposed by~\cite{White1963}. Moreover, without the minimum operation in \eqref{eq-intro_ERSC_Bellman}, this algorithm acts as a fixed point iteration for computing the principal eigenpair of a linear system.
				  
				  In this paper, we propose two  {new} RVI algorithms for the ERSC problem (see Algorithms~\ref{alg-rvi-bertsekas-JI} and~\ref{alg-rvi-bertsekas-GS}, respectively).
				  For a fixed reference state $n$, Algorithm~\ref{alg-rvi-bertsekas-JI} proceeds as follows.  Given the iterate pair $(V^k,\lambda^k)$, the value function update is given by $$V^{k+1} \doteq  \widetilde F(V^k,\lambda^k),$$ where $\widetilde F(\cdot,\lambda^k)$ is the risk-sensitive Bellman-like operator corresponding to $\lambda^k$ (see \eqref{eq-F-tilde} for its definition and Remark~\ref{rem-ESSP} for a discussion on its role in the ERSC problem). The optimal cost update is then given by
				  $$ \lambda^{k+1} \doteq  \lambda^k + \frac{1}{\delta}\gamma_k\log V^{k+1}(n),$$
				  where $\delta$ denotes the sensitivity parameter and $\gamma_k$ is a user-specified step-size.
				  We emphasize that the updates for both the value function $V^k$ and optimal cost $\lambda^k$ in Algorithm~\ref{alg-rvi-bertsekas-JI}
				  differ significantly from that of the existing RVI algorithm. In particular, the optimal cost iterate is now updated through a one-dimensional line-search with descent direction $\frac{1}{\delta}\log V^{k+1}(n)$, instead of  being computed by evaluating the multiplicative Bellman operator applied to $V^k$ at state $n$. 
				  
				  Algorithm~\ref{alg-rvi-bertsekas-GS} is structurally similar to Algorithm~\ref{alg-rvi-bertsekas-JI} and  updates the value function according to $V^{k+1} \doteq  \widetilde G(V^k,\lambda^k)$, where $\widetilde G$ is a Gauss-Seidel implementation of operator $\widetilde F$; see \eqref{eq-G-tilde} for its definition. Furthermore, Algorithms~\ref{alg-rvi-bertsekas-JI} and~\ref{alg-rvi-bertsekas-GS} resemble the well-known Jacobi and Gauss-Seidel iterations, respectively, that are
				  used in solving diagonally dominant linear systems of equations. It is for this reason that we refer
				  to them as the Jacobi-like RVI and Gauss-Seidel-like RVI algorithms, respectively.
				We also note that Algorithms~\ref{alg-rvi-bertsekas-JI} and~\ref{alg-rvi-bertsekas-GS} can be viewed as the risk-sensitive counterparts of the average cost RVI algorithms introduced in \citep{Bertsekas_VI_98}.
				
				Our main result can be stated as follows: under the assumption that the DTMC is irreducible and recurrent for all stationary policies (see Assumption~\ref{assump-main}), and for an appropriate range of step-sizes $\gamma_k$, the iterates of Algorithms~\ref{alg-rvi-bertsekas-JI} and~\ref{alg-rvi-bertsekas-GS} converge geometrically to the solution pair of the multiplicative Bellman equation (see Theorem 2). We note that the proof of Theorem~\ref{thm-main-1}  differs substantially from that of the existing RVI algorithm; the reason for this follows from the facts that the update for $\lambda^k$ is now given through a line-search iteration, and also because the operators $\widetilde F, \widetilde G$ are different than the multiplicative Bellman operator on the right hand side of \eqref{eq-intro_ERSC_Bellman}. This also leads to increased difficulty in proving our main result, as it means the iteration for $\lambda^k$ is less straightforward to analyze and we can no longer use the span-semi-norm contractive property of the multiplicative Bellman operator present in the existing algorithm. Moreover, the geometric convergence result of Theorem~\ref{thm-main-1} also differs from that of the existing algorithm; see Remark~\ref{rem-conv-compare}. 
				
					In the risk-neutral setting~\citep{Bertsekas_VI_98}, the proof of geometric convergence of the algorithms' iterates uses two key properties: the global contraction property of the average cost Bellman-like operator (with respect to a certain weighted supremum norm), and the global bi-Lipschitz continuity property of the fixed points of this operator; see equations~(17) and~(20), and the proof of Proposition~1 in \citep{Bertsekas_VI_98}. However, in the risk-sensitive setting, the operators $\widetilde F,\widetilde G$ associated with Algorithms~\ref{alg-rvi-bertsekas-JI} and~\ref{alg-rvi-bertsekas-GS} are multiplicative (instead of additive) due to the risk-sensitive cost criterion. This leads to major obstacles when trying to establish the analogous global estimates for the risk-sensitive operators $\widetilde F$ and $\widetilde G$, as the additive structure which underlies the global contraction and bi-Lipschitz continuity properties of the risk-neutral Bellman-like operator does not extend.
					
					We address these obstacles posed by the multiplicative structure by establishing suitable local estimates for the operators $F$ and $G$, where $F(h,\lambda)=\log \widetilde F(e^h,\lambda)$ and $G(h,\lambda)=\log \widetilde G(e^h,\lambda)$, and also for  the map $\lambda \mapsto h_\lambda$, where $h_\lambda$ denotes the fixed point of $F(\cdot,\lambda)$ (which is the same as the fixed point of $G(\cdot, \lambda)$). In particular, we prove that $F(\cdot,\lambda)$ and $G(\cdot,\lambda)$ satisfy local contraction properties under a chosen weighted supremum norm (see Propositions~\ref{thm-contraction} and~\ref{prop-contraction-GS}), and we also establish a local bi-Lipschitz continuity property of the map $\lambda \mapsto h_\lambda$ (see Proposition~\ref{lem-bounds-h}). We then apply these local properties to establish the desired one-step geometric convergence bounds for the value function and cost iterates (see Section~\ref{sec-proof-main-1}).
					
					The proof of the local contraction property of $F(\cdot,\lambda)$ involves first recasting  it into a variational form that is additive in $h$, but at the cost of introducing an additional maximization problem over a set of probability vectors. This enables us to upper bound $F(h_1,\lambda)-F(h_2,\lambda)$ in terms of the expectation of $h_1-h_2$  with respect to the probability vector which arises as the maximizer from the variational representations. Because the maximizing probability vector depends on the inputs $h_1,h_2$ in addition to the original transition kernel, it is difficult to ascertain whether $F(\cdot, \lambda )$ is a global contraction. We address this by lower bounding this probability vector in terms of $\|h_1\|_\infty$ and $ \| h_2\|_\infty$, and then use a certain weighted supremum norm construction tailored to our setting to show that $F(\cdot, \lambda)$ satisfies the desired local contraction property. 
					We note that the local contraction property of operator $G$ is much more involved compared to that of $F$. This is a consequence of the recursive and multiplicative structure of $G$ making it unclear how to bound its components, which prevents us from bounding the maximizing probability vector arising in its variational representation. We overcome this by establishing uniform upper and lower bounds for each component of  $G(h,\lambda)$ in both $h$ and $\lambda$. We also comment that the choice of the weighted supremum norm is important not only for the contraction property, but also in showing the geometric convergence of the algorithms' iterates, as it ensures that the value function iterates satisfy one-step contraction estimates and thus converge to their corresponding fixed points at a geometric rate. 
					
					The proof of local bi-Lipschitz continuity of $\lambda \mapsto h_\lambda$ is more subtle.  In the risk-neutral setting,  the Lipschitz bounds follow directly from the additive structure of the stochastic representation for the analogous fixed points and the uniform recurrence of the controlled DTMC, as one is able to uniformly bound the expected return times to obtain the desired result. However, in the risk-sensitive setting, the multiplicative structure via the exponential inside the expectation of $F(\cdot,\lambda)$ means that its fixed point $h_\lambda$ may not exist for all $\lambda\in \R$ (see Remark~\ref{rem-Lambda}). Furthermore, $h_\lambda$ can be expressed using a stochastic representation involving the exponential of the cumulative cost up to the first return time to the reference state $n$ (see Proposition~\ref{prop-ESSP}), but the exponential moments in this representation for $h_\lambda$ are not necessarily finite for all $\lambda \in \R$, which prevents us from  bounding $h_\lambda - h_{\lambda'}$ for arbitrary $\lambda, \lambda'$ by casting these stochastic representations in a variational form (see Remark~\ref{rem-Lipschitz-difficult}). As a result, we cannot directly adapt the approach involving the stochastic representations used in the risk-neutral setting to prove the desired Lipschitz bounds. 
					
					Our approach instead involves working with a variational form of fixed-point equation $h_\lambda = F(h_\lambda, \lambda)$.  Seeking to bound $h_\lambda - h_{\lambda'}$ in terms of $\lambda ' - \lambda$, we repeatedly apply the variational representations up to the first return time of the DTMC into the reference state $n$ and then use the maximizing probability vectors from the variational representation to construct a new DTMC with law $ Q$. Using the variational representations and an application of Dynkin's formula, we bound $h_{\lambda} - h_{\lambda'}$ in terms of  $\lambda' - \lambda$ multiplied by the expected return time of the new DTMC into state $n$ under $ Q$. This shifts the difficulty in obtaining explicit estimates of the return times under the DTMC $Q$ in terms of the original transition matrix, which is not known because the new transition probabilities depend on the fixed points $h_\lambda$ and $h_{\lambda'}$ (and these are unknown).
					We note that this does not follow immediately from the definition of $Q$, as this will lead one to an upper bound with an exponential moment which may be infinite. To overcome this difficulty, we prove suitable uniform bounds on the tail probabilities of $Q$ in terms of the original transition matrix and $\| h_\lambda\|_\infty,\| h_{\lambda'}\|_\infty.$ This then enables us to uniformly bound the return times under $Q$, which leads to the desired local bi-Lipschitz continuity result of $\lambda \mapsto h_\lambda.$ We emphasize that the proven recurrence bounds do not immediately follow from standard mixing time arguments for Markov chains~\citep{Levin_mixing_2008}, since the transformed law $Q$ depends on the fixed points $h_\lambda, h_{\lambda'}$, whose values are not readily available.
					
					Using the local contraction properties and local bi-Lipschitz continuity properties of operators $F$ and $G$, we are almost in a position to show the geometric convergence of $(h^k,\lambda^k)$, where $h^k \doteq \log V^k$; see the proof sketch at the end of Section~\ref{sec-main-result}. However, the local nature of these properties leads to two notable challenges: (i) in order to obtain suitable estimates from the local contraction property, we first have to ensure that $h^k$ remain in a bounded set, and (ii) in order to apply the local bi-Lipschitz continuity property, we have to show that $\lambda^k$ remain in a region where the corresponding fixed point $h_{\lambda^k}$ exists. Note that these challenges do not arise in the risk-neutral setting, due to the global nature of the estimates corresponding to the risk-neutral operators. We address these two points in Proposition~\ref{prop-c-alpha0}. 
					
					We also evaluate the performance of our proposed algorithms and compare them with the existing RVI algorithm 
					by implementing them on two models, chosen to capture different transition matrix structures.  The first is a service-effort control problem for a discrete-time single-server queue with finite capacity, which as a birth-death process has a tridiagonal transition matrix. Here, we vary both the state-space size and the sensitivity parameter, and find that the Gauss-Seidel-like algorithm performs particularly well when the sensitivity is large. The second is a controlled DTMC on a finite connected graph, where the objective is to maximize the exit rate from a fixed connected subgraph. We recast this problem as an ERSC problem, which allows us to compare the algorithms across graphs with different levels of connectivity (and thus sparse/dense transition matrices). 
					For the two proposed algorithms, we find that the choice of step-size is relevant for performance: larger step-sizes aid in making sufficient progress, but do not guarantee convergence. At the same time, step-sizes that are too small may converge albeit slowly. Inspired from the risk-neutral setting~\citep{Bertsekas_VI_98}, we use step-sizes that start relatively large and taper only after the iterates lie within a certain tolerance.

					\subsection{Literature review}\label{sec-lit}
				 The relative value iteration (RVI) algorithm for the ergodic control of Markov chains with finite state space and control set was  introduced in  \cite{White1963}.
				Since then, there has been an extensive body of work studying RVI algorithms for ergodic control of Markov chains under various assumptions and settings; see Section 4.5 of \cite{Bertsekas_DP} for a full account of these developments.  One important development was given in \cite{Bertsekas_VI_98}, which introduced the Jacobi and  Gauss-Seidel versions of the RVI algorithm,  inspired by the stochastic shortest path problem (SSP). Under an irreducibility condition on the controlled  DTMC, the author uses the contraction property of the Bellman operator associated with the SSP in conjunction 
				with the structure of the SSP to prove that the  iterates of both versions converge to the value function and the optimal cost at a geometric rate.

				In the context of the ERSC problem,  to our knowledge, the first work discussing the  RVI algorithm appeared in \cite{Bielecki_1999}, where the authors study a multiplicative analogue of the risk-neutral RVI algorithm studied in \cite{White1963} for finite-state controlled DTMCs. In addition to assuming irreducibility of the controlled DTMC, the authors assume that the probability of a self-loop is bounded away from $0$ for all states and controls. This allowed the authors to establish a contraction property of the Bellman operator with respect to a span semi-norm, which is then used to show  that the RVI iterates converge at a geometric rate.  This  version of the RVI algorithm has been subsequently studied by several authors under various setups and conditions in \cite{Cavazos_2003,Borkar_2002}, and  \cite{Arapostathis_2019}. The strict positivity condition  on the probability of a self-loop is removed in \cite{Cavazos_2003}, where the authors study a transformed version of the original process, under the general assumption that a solution to the multiplicative Bellman equation exists. \cite{Borkar_2002} establish  the convergence of this RVI algorithm in the countable setting  under the conditions of   aperiodicity and irreducibility, and the near-monotonicity of the running cost, by employing multiplicative ergodic theorems.  {The policy iteration algorithm, which recursively evaluates policies and improves them, has also been investigated for the ERSC problem, both for finite state space \citep{Fleming_RS_Finite_State_1997} and countable state space \citep{Biswas2021ERSC, Chen_RS_Discrete_2023}. Furthermore, the modified policy iteration algorithm, which alternates between partial policy evaluation and improvement, has been recently studied for the ERSC problem in \cite{Murthy_ERSC_MPI_2024}, where the authors work with a transformed version of the cost and transition matrix to prove that the algorithm's iterates converge geometrically (here, geometric convergence refers to the convergence of the normalized modified policy iterates to the optimal cost). We remark that Jacobi-like and Gauss-Seidel-like RVI algorithms proposed in this work (as well as their convergence) are clearly distinct from the modified policy iteration algorithm}. 
				The convergence of the existing RVI algorithm in the case of compact Polish state space was carried out  in  \cite{Arapostathis_2019}. 	 				We note that the RVI algorithm for the ERSC problem has seen applications across various domains, including portfolio optimization (\cite{Bielecki_RS_portfolio_RVI_1999,Bielecki_Cox_2005}), manufacturing systems (\cite{Coraluppi_1999}), and stochastic differential games (\cite{Arapostathis2013}). We note that the existing RVI algorithm has served as a framework for the risk-sensitive Q-learning algorithms proposed in \cite{Borkar_2002_Q_learning} and \cite{Basu_RS_Q_learning_2008}, and similarly, the policy iteration algorithm has served as a framework for actor-critic and policy gradient algorithms \citep{Borkar_sensitivity_RS_2001,Guin_RS_Actor_2026, Moharrami_policy_grad_exp_2024}.

					\subsection{Organization  of the paper} We  introduce the necessary notation in the next subsection. In Section~\ref{sec-prelim}, we introduce the ERSC problem in the case of finite state space,  propose our new Jacobi-like and Gauss-Seidel-like RVI algorithms,  and state our  main result on  their  convergence. 
					In Section~\ref{sec-auxiliary}, we establish the local contraction property and the local bi-Lipschitz continuity property of the operator $F$.  Section~\ref{sec-proof-main-1} contains the proof of our main result. 
					In Section~\ref{sec-num-imp}, we provide the numerical examples to illustrate the performances of our proposed algorithms. In Section~\ref{sec-conclude}, we conclude our work and discuss future directions. In the Appendix, we provide a notation table for our main result, the precise statement of the entropy variational formula, an additional numerical example, and the proofs for two key supporting results. 
					
    \subsection{Notation}\label{sec-notation} 
 We use $(\Omega, \mathcal{F}, \PP)$ to denote the underlying abstract probability space with $\E$ as the
associated expectation. The set of nonnegative real numbers (integers) is denoted
by $\R_+$ $(\mathbb{Z}_+)$, $\mathbb{N}$ stands for the set of natural numbers, and $\Ind_{ A}(\cdot)$ denotes the indicator function corresponding to set $A$. Let  $\R^k$ denote the $k$-dimensional Euclidean space,  $e\doteq (1,1,\ldots,1)^\top\in \R^k$ and $\|\cdot\|_\infty$ denote the usual supremum norm.   For a Polish space $\mathcal{X}$, let $\mathcal{P}(\mathcal{X})$ denote the space of Borel probability measures on $\mathcal{X}$. For two probability measures $P,Q\in \mathcal{P}(\mathcal{X}),$ we say $Q\ll P$ if $Q$ is absolutely continuous with respect to $P$.

	\section{RVI algorithms and their convergence}\label{sec-prelim}
\subsection{ERSC problem}		
Consider a controlled DTMC ${X} = \{X_t\}_{t = 0}^{\infty}$ which takes values in a finite-state space $S=\{1,\dots,n\}$. 
Let $\bU(i)$ be a compact metric space for every $i\in S$. At each instant, if the current state is $i$, then, for a chosen control  $u\in \bU(i) $, the next transition to state $j$ occurs with probability $p(i,j,u)$. Namely,  $\{p(i,j,u): i,j\in S, \text{ and } u\in \bU(i)\}$ is the associated controlled transition probability of  $X$.  Let $S_\bU\doteq \{(i,u): i\in S \text{ and } u\in \bU(i)\} $ denote the set of all allowed state-action pairs.

A control policy (sometimes referred to simply as a policy) is a sequence of controls for the DTMC $X$ at each time instant, which we denote by $U = \{U_0, U_1,\dots\}$ where $(X_t,U_t) \in S_\bU$ for all $t\in \ZZ_+$. We say a control $U$ is admissible if for every $t\in\ZZ_+$, $U_t$ is measurable with respect to the $\sigma$--algebra generated by $\{X_0,X_1,\ldots,X_t\}$. Let $\Uadm$ denote the set of all admissible control policies.  A policy $U$ is called Markov if  $U_t = v_t(X_t)$, where $v_t (\cdot)$ is such that $\big(X_t,v_t(X_t)\big)\in S_\bU$ for $t\in \ZZ_+$ and it is called stationary Markov if $v_t(\cdot)$ is independent of $t$. With slight abuse of notation, we refer to such a  stationary Markov policy by $v$ and let $\Usm$ denote the space of stationary Markov control polices.				
Given an admissible  control policy $U$ and an initial state $i$, the corresponding cost functional for the ERSC problem, with a risk-sensitive parameter $\delta>0$, is given by
\begin{equation}\label{eq-ERSC}
	\mathcal{E}^\delta_i (U)  \doteq \limsup_{T\to \infty}\frac{1}{\delta T}\log \E_i^{U}\left[e^{\delta \sum_{t=0}^{T-1}c(X_t, U_t)}\right],
\end{equation}
where $c: S_\bU\rightarrow \RR_+$ is the running cost and we write $\E_i^{U}$ to emphasize that the underlying control policy is $U$ and  $X_0=i$. Whenever the control policy $v$ lies in $ \Usm $, we write $\E_i^v$.

The objective of the ERSC problem is to minimize the cost in~\eqref{eq-ERSC} over  all the admissible controls $U$ and initial conditions $1\leq i\leq n$, \emph{i.e.},  to find \begin{equation}\label{eq-ERSC-val}\lambda^{*,\delta}\doteq\min_{1\leq i\leq n} \inf_{U \in \Uadm}\mathcal{E}^\delta_i (U)\end{equation} and also to characterize  optimal stationary Markov control policies $v$ (if they exist), \emph{i.e.}, a stationary Markov policy $v$ satisfying $\mathcal{E}^\delta_i(v) = \lambda^{*,\delta}$,  for $1\leq i\leq n$.
\begin{remark}
	It is well known  that as $\delta\to 0$, $\lambda^{*,\delta}$ converges to $\lambda^{*,0}$ the optimal cost of the long-run average (ergodic) cost control problem, given by 
	\begin{equation*}\lambda^{*,0}\doteq\min_{1\leq i\leq n} \inf_{U \in \Uadm}\mathcal{E}_{i}^0 (U),\quad\text{where}\quad\mathcal{E}^0_i (U)  \doteq \limsup_{T\to \infty}\frac{1}{ T} \E_i^{U}\left[\sum_{t=0}^{T-1}c(X_t, U_t)\right]\,.\end{equation*}
	For this reason, we often refer to this problem as the risk-neutral ergodic control problem.
\end{remark}

We now state the conditions  on the controlled transition kernel $p(i,j,u)$ and the process $X$ under which we later prove the convergence of the RVI algorithms.
\begin{assumption}\label{assump-main} The following conditions hold:
	\begin{enumerate}						\item For any $1\leq i,j\leq n$, the functions $u\mapsto c(i,u)$ and $u\mapsto p(i,j,u)$ are continuous on $\bU(i)$.
		\item The DTMC $X$ is irreducible under all stationary Markov policies.
	\end{enumerate}
\end{assumption}

The following theorem, which is Theorem 3.1 of \cite{Cavazos-Cadena2002}, provides the Bellman optimality criterion which completely characterizes the optimal stationary Markov controls.					
\begin{theorem}\label{thm-hjb}
	Under Assumption~\ref{assump-main}, there exists a  function  $V^{*,\delta}:S \rightarrow \R_+$ that is unique up to a multiplicative constant and satisfies
	\begin{equation}\label{eq-hjb}
		e^{\delta \lambda^{*,\delta}}V^{*,\delta}(i) = \min_{u \in \mathbb{U}(i)} \bigg[e^{\delta c(i,u)}\sum_{j =1}^nV^{*,\delta}(j)p(i,j,u)\bigg],  \quad \text{ for $1\leq i\leq n$}\,.
	\end{equation}
	Moreover, $v\in \Usm$ is optimal if and only if it satisfies 
	\begin{equation}\label{eq-optimality}
		\min_{u \in \mathbb{U}(i)} \bigg[e^{\delta c(i,u)}\sum_{j =1}^nV^{*,\delta}(j)p(i,j,u)\bigg]=  e^{\delta c(i,v(i))}\sum_{j =1}^nV^{*,\delta}(j)p(i,j,v(i)),  \quad \text{ for $1\leq i\leq n$}\,.
	\end{equation}
\end{theorem}
In what follows, we often suppress the dependence of $(V^{*,\delta},\lambda^{*,\delta})$ on $\delta$ and simply write  $(V^*,\lambda^*)$. 				
\subsection {Design of RVI algorithms}\label{sec-sub-design}
It is evident from~\eqref{eq-optimality} that  to characterize an optimal stationary  Markov policy, we require the knowledge of the value function $V^{*}$ which in turn, requires the knowledge of $\lambda^{*}$. Therefore, it is of significant interest to numerically compute the pair $(V^{*},\lambda^{*})$. As discussed in Section~\ref{sec-lit}, it is well known that the pair $(V^{*},\lambda^{*})$ can be computed from a variety of numerical algorithms.  Many of them fall into two categories: relative value iteration (RVI) and policy iteration. As mentioned already, we confine ourselves to investigating RVI type algorithms in this paper. 
Most of the earlier works on RVI algorithms for the ERSC problem study a version from \cite{Borkar_2002}, which we refer to as the existing RVI algorithm (see the update~\eqref{eq-alg-existing}), which generates iterates designed to find a fixed point of the Bellman equation from Theorem~\ref{thm-hjb}.
In Theorem 4.5 of~\cite{Borkar_2002}, it is shown that  as $k\to \infty$,   $(V^k,\lambda^k) \to (V^*,\lambda^*)$, under the assumption that the controlled DTMC is irreducible, aperiodic, and recurrent for every stationary Markov control policy. We refer the reader to Remark~\ref{rem-conv-compare} for a discussion on the rate of convergence of this algorithm.		

\smallskip					

\begin{algorithm}[H]
	\TitleOfAlgo{Existing RVI Algorithm}
	\begin{enumerate}
		\item [(i)] Initialize with $k=0$ and $V^0:S \rightarrow \R^+$.
		\item [(ii)] Update: for $1\leq i\leq n$,
		\begin{equation}\label{eq-alg-existing}
			\begin{aligned}
				\lambda^{k+1} &= \min_{u \in \mathbb{U}(n)} \frac{1}{\delta} \log \bigg[e^{\delta c(n,u) }\sum_{j =1}^nV^k(j)p(n,j,u)\bigg],\\
				V^{k+1}(i) &= \min_{u \in \mathbb{U}(i)} \bigg[e^{\delta (c(i,u) - \lambda^{k+1}) }\sum_{j =1}^nV^k(j)p(i,j,u)\bigg] \,.
			\end{aligned}
		\end{equation}
	\item [(iii)] Set $k=k+1$.
	\item [(iv)] Repeat Steps (ii) and (iii). 
	
\end{enumerate}
\end{algorithm}

{ The design of our RVI algorithms is inspired by the following observation: working with \eqref{eq-hjb}, we move $e^{\delta \lambda^*}$ to the RHS, factor $V^*(n)$ from the RHS of \eqref{eq-hjb}, and set $\overline \lambda^*\doteq \lambda^*- \frac{1}{\delta}\log V^*(n)$ (which is possible since $V^*>0$) to obtain}
\begin{align}\nonumber
V^*(i) &= \min_{u \in \mathbb{U}(i)} \bigg[e^{\delta( c(i,u) - \bar \lambda^*) }\Big( p(i,n,u)+ \sum_{j =1}^{n-1}\frac{V^*(j)}{V^*(n)}p(i,j,u)\Big)\bigg], \quad \text{ for $1\leq i\leq n$}\,.
\end{align} 
Here and in what follows, state $n$ is chosen  as the reference state. We re-arrange the above display with $\bar V^*\doteq \frac{1}{V^*(n)}V^*$ 						to obtain
\begin{equation*} 
\begin{aligned}							V^*&= \widetilde F\big( \bar V^*, \bar \lambda^*\big),\\
	\lambda^* &= \bar \lambda^*+ \frac{1}{\delta}\log V^*(n),
\end{aligned}
\end{equation*}
where $\widetilde F: \RR^n\times \RR\rightarrow \RR^n$ is defined by
\begin{equation}\label{eq-F-tilde} 
\begin{aligned}
	&\widetilde F_i(V,\lambda)\doteq \min_{u\in\bU(i)}  \bigg[e^{(\delta c(i,u)-\lambda)}\bigg(p(i,n,u) + \sum_{j =1}^{n-1}V(j)p(i,j,u)\bigg)\bigg] ,\quad \text{for} \,\, 1 \le i \le n\,.
\end{aligned}
\end{equation} 
From here, we see the clear evidence of a Jacobi-type  iteration procedure with iterates $(V^k,\lambda^k)$  if $(\bar V^*,\bar \lambda^*)$ is replaced by $(V^k,\lambda^k)$ and $(V^*,\lambda^*)$ is replaced by $(V^{k+1},\lambda^{k+1})$, for $k\geq 0$ with $(V^0,\lambda^0)$ chosen as the initial condition. We note two key observations regarding the resulting iteration procedure: (i) the condition for  $\lambda^{k+1}=\lambda^k$, \emph{i.e.}, the fixed point of the iteration procedure, is that $V^{k+1}(n)=1$, and (ii)  $(V^*,\lambda^*)$ is a fixed point {(and is unique, up to the normalization $V^*(n) = 1$)}. We also note that these facts remain true if we replace $\delta^{-1}\log V^{k+1}(n)$ by $\gamma_k \delta^{-1} \log V^{k+1}(n)$, for any family of positive real numbers $\{\gamma_k\}_{k\in \ZZ_+}$. As a result, this discussion leads to the following RVI algorithm:

\begin{algorithm}[H]
\caption{Jacobi-like RVI algorithm}\label{alg-rvi-bertsekas-JI}	
\begin{enumerate}
	\item [(i)] Initialize  $k=0$, $V^0:S \rightarrow \R^+$ and $\lambda^0 \in \RR_+$.
	\item [(ii)] Choose step-size $0<\gamma_k<1$ and update: 
	
	\begin{equation}\label{eq-alg-2-V}
		\begin{aligned}
			V^{k+1} &= \widetilde F(V^k,\lambda^k),\\
			\lambda^{k+1}&= \lambda^k + {\frac{1}{\delta}}\gamma_k \log V^{k+1}(n), 
		\end{aligned}
	\end{equation}
	\item [(iii)] Set $k=k+1$.
	\item [(iv)] Repeat Steps (ii) and (iii).
\end{enumerate}
\end{algorithm}			
From here, we observe that we can replace~\eqref{eq-alg-2-V} with a Gauss-Seidel updating procedure. To express this, define $\widetilde G: \R^n \times \R \rightarrow\R^n$ by \begin{equation}\label{eq-G-tilde}
\scalemath{0.9}{	\begin{aligned}
		\widetilde G_i(V,\lambda)&\doteq  \begin{cases} \underset{{u \in \mathbb{U}(1)}}{\min}\bigg[e^{(\delta c(1,u) - \lambda) }\bigg( p(1,n,u) + \sum_{j =1}^{n-1}V(j)p(1,j,u) \bigg)\bigg], &\text{for $i=1$,}\\
			\underset{u\in \mathbb U(i)}{\min}\bigg[e^{(\delta c(i,u) - \lambda)}\bigg(p(i,n,u) + \sum_{j =1 }^{i-1}\widetilde G_j(V,\lambda)p(i,j,u) + \sum_{j=i}^{n-1}V(j)p(i,j,u) \bigg)\bigg], &\text{for $i\neq 1$}\,.
\end{cases}  \end{aligned}}
\end{equation}
This gives us the following RVI algorithm

\begin{algorithm}[H]	
\caption{Gauss-Seidel-like RVI algorithm}\label{alg-rvi-bertsekas-GS}
\begin{enumerate}
	\item [(i)] Initialize $k=0$,  $V^0:S \rightarrow \R^+$ and $\lambda^0 \in \RR_+$.
	\item [(ii)]  Choose step-size $0<\gamma_k<1$ and update:
	\begin{equation*}
		\begin{aligned}
			V^{k+1} &= \widetilde G(V^k,\lambda^k),\\
			\lambda^{k+1}&= \lambda^k + {\frac{1}{\delta}} \gamma_k \log V^{k+1}(n)\,.
		\end{aligned}
	\end{equation*}
	\item [(iii)] Set $k=k+1$.
	\item [(iv)] Repeat Steps (ii) and (iii).
\end{enumerate}
\end{algorithm}

\begin{remark} The reason behind the inclusion of the step-sizes $\{\gamma_k\}_{k\in \ZZ_+}$ in the designs of Algorithms~\ref{alg-rvi-bertsekas-JI} and~\ref{alg-rvi-bertsekas-GS} is two-fold: (i) we will see in Theorem~\ref{thm-main-1} that the inclusion of an appropriate family of  $\{\gamma_k\}_{k\in \ZZ_+}$ ensures the  geometric convergence of these algorithms, and (ii) while numerically implementing these algorithms, the inclusion of an appropriate choice of $\{\gamma_k\}_{k\in \ZZ_+}$ can aid us in achieving faster convergence; see Section~\ref{sec-num-imp} for more elaborate discussions.
\end{remark} 

To investigate the convergence of Algorithms~\ref{alg-rvi-bertsekas-JI} and~\ref{alg-rvi-bertsekas-GS}, it is evident that the risk-sensitive Bellman-like operators $\widetilde F,\widetilde G: \RR^n\times \RR\rightarrow \RR^n$ defined above
play a fundamental role. 
Observe that for every $\lambda\in \RR$, the fixed {point equations} of $\widetilde F(\cdot,\lambda)$ and $\widetilde G(\cdot,\lambda)$ are the same. The following result provides an alternative characterization of the fixed points, whenever they exist - we only discuss the case of $\widetilde F(\cdot,\lambda)$, as the same discussion applies to $\widetilde G(\cdot,\lambda)$. However, it is not a priori clear if for every $\lambda\in \RR$, there exists a fixed point for $\widetilde F(\cdot,\lambda)$; see Remark~\ref{rem-Lambda}. To that end, define 
\begin{equation}\label{eq-Lambda} \Lambda \doteq \Big\{ \lambda\in \RR: \text{ a fixed point of $\widetilde F(\cdot,\lambda)$  exists and is finite}\Big\}\,.\end{equation} 

\begin{proposition}\label{prop-ESSP}
For $\lambda \in \Lambda$, under Assumption~\ref{assump-main},  there exists a unique function $V_\lambda:S\rightarrow \RR$ such that  $V_\lambda =\widetilde F(V_\lambda,\lambda)$. Moreover, the following hold. 
\begin{enumerate}
	\item [(i)] $V_{\lambda}(n) = 1$ if and only if $\lambda = \lambda^*$, where $\lambda^*$ is as defined in~\eqref{eq-ERSC-val}.
	\item [(ii)] An alternative characterization of $V_\lambda$ holds: for $1\leq i\leq n$,\begin{equation}\label{eq-V-lambda-def}
		V_{\lambda}(i) \doteq \inf_{v\in \Usm}\E_i^v \bigg[e^{\sum_{t=0}^{\tau_n - 1}\big(\delta c(X_t, v(X_t)) - \lambda \big)}\bigg] \,.\end{equation}
	Here, $\tau_n \doteq  \min\{t \geq 1: X_t = n\}$ with $i$-dependence suppressed.
\end{enumerate}

\end{proposition}
\begin{proof}
	Part (i) follows from the fact that $V_{\lambda^*}(n) = 1$ and that the map $\lambda \mapsto V_{ \lambda}$ is injective, and  Part (ii) follows from Theorem 5.1 of \cite{Cavazos-Cadena2002}.
\end{proof}
A few remarks are now in order.
\begin{remark}
Observe that from Theorem~\ref{thm-hjb} and the above proposition, whenever $V_\lambda(n)=1$ (or equivalently, $\lambda=\lambda^*$), we have $V_\lambda=V^*$, where $V^*$ is given by Theorem~\ref{thm-hjb}.
\end{remark}
\begin{remark}\label{rem-ESSP}
From Proposition~\ref{prop-ESSP}, it is clear that $V_\lambda$ can be interpreted as the infimum of  the  exponential cost  \begin{equation*}\E_i^v \Big[e^{\sum_{t=0}^{\tau_n - 1}\big(\delta c(X_t, v(X_t)) - \lambda \big)}\Big], \end{equation*} over all $v\in \Usm$,  incurred while the DTMC $X$ starts at state $i$ and enters state $n$, for the first time. In other words, this means that $v^*\in \Usm $ for which the infimum in~\eqref{eq-V-lambda-def} is attained (which exists by Assumption~\ref{assump-main})  corresponds to the exponentially stochastic shortest path from state $i$ to state $n$, while varying over all $v\in \Usm$. The analogous connection between the risk-neutral stochastic shortest path problem and risk-neutral RVI algorithms was thoroughly investigated in \cite{Bertsekas_VI_98}. This problem was also studied in \cite{Cavazos-Cadena2002}, although the authors referred to it as the ``auxiliary expected-total cost" problem. 
\end{remark}
\begin{remark}\label{rem-Lambda} 	From the alternative characterization in Proposition~\ref{prop-ESSP}, due to the exponential nature of the cost, it is more apparent that $V_\lambda$ may not exist for every $\lambda$, \emph{i.e.}, $V_\lambda(i)$ may be infinite for some $1\leq i \leq n$ and $\lambda\in \RR$. This illustrates  clearly the reason for defining $\Lambda$ in~\eqref{eq-Lambda}.				\end{remark}

We remark that much of the analysis that follows is for operator $F$, as many of these techniques can be implemented for operator $G$ as well, in conjunction with Lemma~\ref{lem-bound-G} and Proposition~\ref{prop-contraction-GS}  from Section~\ref{sec-proof-GS}.

From Theorem~\ref{thm-hjb}, we make the following trivial, but important observation: for that we set $h^*=\log V^*$ and take the logarithm on both sides of~\eqref{eq-hjb} and~\eqref{eq-optimality}. It is clear that~\eqref{eq-hjb} and~\eqref{eq-optimality}, respectively become
\begin{equation}\label{eq-hjb-log}
h^*(i) = \min_{u \in \mathbb{U}(i)} \bigg[\delta c(i,u)+\log \Big(\sum_{j =1}^ne^{h^*(j)}p(i,j,u)\Big)\bigg]  -\delta\lambda^*,\qquad \qquad \qquad
\end{equation}
\begin{equation}\label{eq-optimality-log}
\min_{u \in \mathbb{U}(i)} \bigg[\delta c(i,u)+\log \Big(\sum_{j =1}^ne^{h^*(j)}p(i,j,u)\Big)\bigg] =\delta c(i,v(i))+\log \Big(\sum_{j =1}^ne^{h^*(j)}p(i,j,v(i))\Big),
\end{equation}		
for $1\leq i\leq n$. From the above, it is clear that $v\in \Usm$ satisfies~\eqref{eq-optimality} if and only if it satisfies~\eqref{eq-optimality-log}.  Hence, to characterize the optimal stationary Markov policies associated with our ERSC problem, we can equivalently work with~\eqref{eq-hjb-log} and~\eqref{eq-optimality-log}. For this reason, instead of investigating operators $\widetilde F$ and $\widetilde G$, we investigate the  operators $F,G:\RR^n\rightarrow \RR^n$, also referred to as risk-sensitive Bellman operators, given by 
{\begin{align}\label{eq-hjb-op}
	&   F(h,\lambda) \doteq \log \widetilde F(e^{h}, \lambda ),\\\label{def-G}
	&\begin{aligned}
		G(h,\lambda) \doteq \log \widetilde G(e^h, \lambda )\,.
	\end{aligned}
	\end{align}}
Using the operators $F,G$ and setting $h^k\doteq \log V^k$, Algorithm~\ref{alg-rvi-bertsekas-JI} can be equivalently expressed as  
\begin{equation}\label{eq-F-itr}
	\begin{aligned}
		h^{k+1} &=  F(h^k,\lambda^k),\\
		\lambda^{k+1}&= \lambda^k + {\frac{1}{\delta}} \gamma_k h^{k+1}(n), 
	\end{aligned}
\end{equation}
and Algorithm~\ref{alg-rvi-bertsekas-GS} can be expressed as 
\begin{equation}\label{eq-G-itr}
	\begin{aligned}
		h^{k+1} &= G(h^k,\lambda^k), \\
		\lambda^{k+1}&= \lambda^k +  {\frac{1}{\delta}}\gamma_kh^{k+1}(n)\,.
	\end{aligned}
\end{equation}

\subsection{Main result}\label{sec-main-result}
We are now in a position to state our first main result which concerns the convergences of Algorithms~\ref{alg-rvi-bertsekas-JI} and~\ref{alg-rvi-bertsekas-GS}. Before we proceed to do this, we introduce a weighted supremum norm that turns out to be fundamental in proving these convergences. 
We follow the construction from Pg. 293 of \cite{tseng1990solving}.
First, we consider a finite partition $\{S_k\}_{1\leq k\leq l}$ of the state space $S = \{1,\dots, n\}$ defined as follows:
\begin{equation}\nonumber
	S_1 \doteq \{n\},\quad 	S_k \doteq \Big\{i\in S: i\notin \cup_{r \leq k - 1}S_r \text{ and } \min_{u\in \bU(i)}\max_{j \in S_1 \cup \dots \cup S_{k-1}} p(i,j,u) > 0\Big\}\,.
\end{equation}
The non-emptiness of $S_k$, for $k>1$, is due to the irreducibility of the DTMC $X$ in Assumption~\ref{assump-main}. 
For any state $i$, let $1\leq k(i)\leq l$ be such that $i\in S_{k(i)}.$ We now define the relevant quantities which will be used to construct our weighted supremum norm. First, we define $\eta$ to be the \emph{smallest, strictly positive }transition probability, that is,  
{ 
	\begin{equation}\label{eq-transition_prob_const}
		\eta \doteq \min \big\{\overline p(i,j)\,:\overline p(i,j)>0, \, i,j\in S \big\}.
	\end{equation}
	where $\overline p(i,j) \doteq \inf\{p(i,j,u):u\in \bU\}$.}
{Since the minimum in \eqref{eq-transition_prob_const} is over a finite set, it is clear that $\eta$ is strictly positive if and only if $\overline p(i,j)$ is strictly positive for some $(i,j)$. However, this must clearly hold by the following argument. 
	Suppose for contradiction that there exists a $j\in S$ such that $\overline p(i,j) = 0$ for all $i\ne j$. By continuity of $u\mapsto p(i,j,u)$ and compactness of $\mathbb U(i)$, it is clear that there exists $u_{i,j}$ such that $p(i,j,u_{i,j})=0.$ Now, if we define a stationary policy $\pi^*$ such that $\pi^*(i) = u_{i,j}$ whenever $i\ne j$, then it is straightforward to see that $j$ is unreachable from any $i\ne j$ under $\pi^*,$ which implies that the process $X$ under $\pi^*$ is reducible, and thus contradicts Assumption~\ref{assump-main}. }

Using $\eta$ and the partition $\{S_k\}_{1\leq k\leq l}$ defined above, we define  weights $\{w^m_i\}_{1\leq i\leq n}$ as follows: 
\begin{equation}\label{eq-weight_vector}
	w_i^m \doteq 1 - (\eta e^{-2m})^{2 k(i)},  \quad \text{ for } \,\, i < n \,\, \text{ with } \,\, w_n^m=1,\end{equation}
\begin{equation}\label{eq-contraction_const}
	\beta_m \doteq \frac{1 - (\eta e^{-2m})^{2l - 1}}{1 - (\eta e^{-2m})^{2l}},
\end{equation}
for every $m>0$.	 Since $\eta \in (0,1)$, it follows that $\beta_m \in (0,1)$.  	 The weighted supremum norm is then defined as follows: for $x\in \RR^n$,
\begin{equation}\label{def-norm}
	\|x\|_m\doteq \max_{1\leq i\leq n}\frac{|x_i|}{w_i^m}\,.
\end{equation}
From the definition of $w^m_i$ in~\eqref{eq-weight_vector}, it is clear that $w^m_i>w^m_{1}>w^0_1$ as $i=1\in S_2$. This means that for $x\in \RR^n$, 
\begin{equation}\label{norm-rel}
	\|x\|_\infty\leq \|x\|_m\leq \frac{1}{w^0_1}\|x\|_\infty\,.
\end{equation}
Let 
\begin{equation}\label{eq-c-r}
	\underline c\doteq \min_{(i,u)\in S_{\bU}} c(i,u)\quad\text{and}\quad \overline c\doteq \max_{(i,u)\in S_{\bU}}c(i,u)\,.
\end{equation}
From the definition of $\lambda^*$ in~\eqref{eq-ERSC-val}, it is clear that $\underline c\leq \lambda^*\leq \overline c\,.
$
Hence, we simply choose the initial condition of Algorithms~\ref{alg-rvi-bertsekas-JI} and~\ref{alg-rvi-bertsekas-GS}, $(V^0,\lambda^0)$ to be such that $\underline c\leq \lambda^0\leq \overline c$.  

\begin{theorem}\label{thm-main-1} Suppose  that $\{( V^k,\lambda^k)\}_{k\in\NN}$ is the sequence of iterate pairs of either Algorithm~\ref{alg-rvi-bertsekas-JI}  or Algorithm~\ref{alg-rvi-bertsekas-GS} and  $h^k\doteq \log V^k$. 
	Under Assumption~\ref{assump-main} and the condition that $ \lambda^*\leq \lambda^0\leq \overline c$, there exist constants $m>0$ (depending on $(h^0,\lambda^0)$), {$\overline \gamma >0,$}
	and positive functions  $c_h(\gamma,m)$,  $c_\lambda(\gamma,m)$  and $L(m)$ with the following property: for $k\geq 1$, whenever
	{ 	$\gamma_k \in (0,\overline \gamma]$}
	we have $0< c_h(\gamma_k, m) ,c_\lambda(\gamma_k, m)<1,$ and 
	\begin{equation}\label{eq-itr-contraction-1}
		\begin{aligned}
			\|h^{k+1}-h^*\|_{ m}& \leq c_h(\gamma_k, m) \|h^{k}-h^*\|_{ m} + L(m) |\lambda^k-\lambda^*|, \\
			|\lambda^{k+1}-\lambda^*|&\leq c_\lambda (\gamma_k, m) |\lambda^k-\lambda^*|\,.
		\end{aligned}
	\end{equation}
	In particular, there exists some constant $\varrho=\varrho(m)>1$ (depending on $ m$ such that $\varrho(m)\to 1$ as $m\to\infty$) such that
	$$ \lim_{k\to\infty} \varrho^k \big( \|V^k-V^*\|_{ m} + |\lambda^k-\lambda^*|\big)=0\,.$$
\end{theorem}	
\begin{remark}\label{rem-init} The condition $\lambda^0\geq \lambda^*$ can be easily satisfied - simply choose a $v\in\Usm$, $1\leq i\leq n$ and compute  $\cE^\delta_i(v)$. Then,  by definition of $\lambda^*$, $\lambda^0\doteq \cE^\delta_i(v)\geq \lambda^*$. 
\end{remark} 

{
	\begin{remark}\label{rem-conv-compare}
		We now discuss how the geometric convergence result given for Algorithms~\ref{alg-rvi-bertsekas-JI} and~\ref{alg-rvi-bertsekas-GS} in Theorem~\ref{thm-main-1} compares with the other geometric convergence results in the literature for the existing RVI Algorithm. For the discussion below, we let Assumption~\ref{assump-main} hold, and also assume the control sets $\mathbb U(i)$ are all finite. 
		
		For this paragraph only, let $\{(\widetilde V^k,\tilde \lambda^k)\}_{k\ge 0}$ denote the iterates of the existing RVI algorithm. In \cite{Bielecki_1999}, under the additional condition that $\inf_{u\in \mathbb U(i) }p(i,i,u)>0$ for each $i\in S$, the authors prove that there exist $\ell\in \NN$ and $\varsigma\in (0,1)$ such that  
		\begin{equation}\label{eq-Bielecki-geom-conv}
			\begin{aligned}
				\| \log \widetilde V^{\ell} - \log V^* \|_{sp} &\le 2\varsigma \| \log \widetilde V^0 - \log V^*\|_{sp}, \\
				| \tilde \lambda^{\ell} - \lambda^*| &\le \varsigma \| \log \widetilde V^0 - \log V^*\|_{sp}\,.
			\end{aligned}
		\end{equation}
		Here, $V^*$ is normalized so that $V^*(n) = 1$, $\varsigma$ is equal to the Birkhoff coefficient of a certain matrix (which depends on the transition kernel, sensitivity parameter, and running cost; see Theorem 2.2 in \cite{Bielecki_1999}), and $\| \cdot\|_{sp}$ denotes the span semi-norm. It is straightforward to see how by repeatedly  applying \eqref{eq-Bielecki-geom-conv} and replacing $(\widetilde V^0, \tilde \lambda^0)$ with $ (\widetilde V^k, \tilde \lambda^k)  $ for any $k$, we have  that $(\widetilde V^k,\widetilde \lambda^k)\rightarrow (V^*,\lambda^*)$ geometrically. This geometric convergence result was also given in \cite{Cavazos_2003}, where the authors removed the extra condition at the start of the paragraph by working with a transformed version of the running cost and transition matrix. From here, we see that the first difference between \eqref{eq-itr-contraction-1} and~\eqref{eq-Bielecki-geom-conv} lies in the fact that the former is a one-step contraction with respect to a certain weighted supremum norm, while the latter is a multi-step contraction with respect to the span semi-norm. Furthermore, the contraction constants $c_h,c_\lambda$ in~ \eqref{eq-itr-contraction-1} depend  on the step-sizes $\gamma_k$ and the value $m$ (which also depends on $h^0,\lambda^0$, $h_{\lambda^0}$, and $h^*$), in addition to the ERSC problem parameters, \emph{i.e.,} running cost, sensitivity parameter, and transition matrix.
	\end{remark} 
}

We end this section with the sketch of the proof of Theorem~\ref{thm-main-1} and defer the detailed proof to Section~\ref{sec-proof-main-1} and Appendix~\ref{app-c-alpha2}. Recall that  Algorithms~\ref{alg-rvi-bertsekas-JI} and~\ref{alg-rvi-bertsekas-GS}, in terms of $h^k=\log V^k$, are expressed according to~\eqref{eq-F-itr} and~\eqref{eq-G-itr}, respectively.  
The main properties given below (which are proved in Section~\ref{sec-auxiliary}) form the building blocks of the proof: 

\begin{enumerate}
	\item [(P1)]  Local contraction property of $F(\cdot,\lambda)$  (see Proposition~\ref{thm-contraction}) and $G(\cdot,\lambda)$ (see Proposition~\ref{prop-contraction-GS}) under the weighted supremum norm defined in~\eqref{def-norm}.
	\item [(P2)] Lipschitz continuity  of $F(h,\cdot)$ and $G(h,\cdot)$: for $F(h,\cdot)$, we obtain this continuity globally.  However,  due to the iterative nature in the definition of $G(h,\lambda)$, we only obtain this continuity locally; see Proposition~\ref{prop-contraction-GS}.
	\item [(P3)] Local bi-Lipschitz continuity of the map $\lambda\mapsto h_\lambda(n)$, whenever $F(\cdot,\lambda)$  has a unique fixed point $h_\lambda$: more precisely, we show that both $\lambda \mapsto h_\lambda$ and its inverse  are locally Lipschitz in  Proposition~\ref{lem-bounds-h}. 	\end{enumerate}
In the sketch below, we mainly discuss the case where $\{(V^k,\lambda^k)\}_{k\in \NN}$ is a sequence of iterates of Algorithm~\ref{alg-rvi-bertsekas-JI} as the sketch of the proof for Algorithm~\ref{alg-rvi-bertsekas-GS} follows similar arguments from hereon. 
To illustrate the main idea behind the proof, we first discuss the local one-iteration analysis below, \emph{i.e.}, given that $(h^k,\lambda^k)$ is close to $(h_{\lambda^*},\lambda^*)$ for some $k\in \NN$, does  $(h^{k+l},\lambda^{k+l})$ converge to $(h_{\lambda^*},\lambda^*)$, as $l\to\infty$?  Suppose for an initialization of Algorithm~\ref{alg-rvi-bertsekas-JI} and $k\in \NN$, we have $\lambda^k\simeq\lambda^*$, $\lambda^k\leq \lambda^*$ and $h^k\simeq h_{\lambda^*}$ ($x\simeq y$ means that $x$ is very close to $y$). Then, Property (P2)  gives us $h^{k+1}=F(h^k,\lambda^k)\simeq F(h^k,\lambda^*)$. Since $F(\cdot, \lambda^*)$ is a local contraction  and $h_{\lambda^*} $ is the unique fixed point (see Property (P1)), $\|h^{k+1}-h_{\lambda^*}\|$ is strictly smaller than $\|h^k-h_{\lambda^*}\|$. Turning to $\lambda^{k+1}-\lambda^*$, from the first part of Property (P3), we have $h_{\lambda^k}\simeq h_{\lambda^*}\simeq h^k$ which means that the continuity of $F(\cdot,\lambda^k)$ (in fact, it is a local contraction; Property (P1)) gives us  $ h^{k+1}= F(h^k,\lambda^k)\simeq F(h_{\lambda^k},\lambda^k)\simeq h_{\lambda^k}\,.$
From the second part of Property (P3), monotonicity of $\lambda\mapsto h_\lambda(n)$ and the fact that $h_{\lambda^*}(n)=0$, we have $ h^{k+1}(n)\simeq h_{\lambda^k}(n)\gtrsim \underline C( \lambda^*-\lambda^k)$ (here, $x\gtrsim y$ means that $y-x$ less than a small positive number), for some constant $\underline C>0$. On the other hand, from the first part of Property (P3), we also have $h^{k+1}(n)\lesssim \overline C(\lambda^*-\lambda^k)$, for some constant $\overline C>0$ (here, $x\lesssim y $ means that  $x-y$ is less than a small positive  number).  
Combining this with the recursion for $\lambda^k$ and choosing $\gamma_k$ from an appropriate closed interval, it  will be shown to ensure that $|\lambda^{k+1}-\lambda^*|$     is strictly smaller than $\lambda^*-\lambda^k$. 

Similarly,  we can argue that if  $\lambda^k\simeq\lambda^*$, $\lambda^k\geq \lambda^*$ and $h^k\simeq h_{\lambda^*}$, then  $\|h^{k+1}-h_{\lambda^*}\|$ is strictly smaller than $\|h^k-h_{\lambda^*}\|$ and $|\lambda^{k+1}-\lambda^*|$ is strictly smaller than $\lambda^k-\lambda^*$. Iterating the above argument indefinitely establishes the local convergence. It turns out that the precise argument is robust enough to relax the above assumption of $\lambda^k\simeq\lambda^*$ and $h^k\simeq h_{\lambda^*}$, and to be used in the proof of convergence of Algorithm~\ref{alg-rvi-bertsekas-JI}.  

Due to the `local' nature of the above properties, it is essential to ensure that the constants involved in the above arguments do not deteriorate in the subsequent iterations. This is a very subtle issue owing to the fact that in general, $\RR\setminus\Lambda\neq \emptyset$, which is addressed in   Proposition~\ref{prop-c-alpha0}.

	\section{Properties of the risk-sensitive Bellman operator $F(\cdot,\lambda)$}\label{sec-auxiliary}

In this section, we establish the following key results regarding the risk-sensitive Bellman operator $F(\cdot,\lambda)$: 
\begin{enumerate}
	\item [(i)] For every $\lambda\in {\RR}$, we establish a local contraction property of operator $F(\cdot,\lambda)$. This, in turn, will help us in establishing the existence and uniqueness of a $\RR^n$--valued vector $h_\lambda$  such that  $h_\lambda= F(h_\lambda,\lambda)$.
	\item [(ii)] For every $\lambda\in \Lambda$ for $\Lambda$ in \eqref{eq-Lambda}, we establish the local Lipschitz continuity of the map $\lambda \mapsto h_\lambda$ and also that of its inverse. 
\end{enumerate} 
The existing approaches in proving (local) contraction properties involve showing that the multiplicative Bellman operator is a (local) contraction in the span-norm, and then normalizing the operator output at state $n$ to be $1$; see \cite{Bielecki_1999}. 
However, these methods do not extend to our case because $F(h,\lambda)$ is not fixed. The techniques in \cite{Bertsekas_VI_98} are amenable for adaptation to our case and as mentioned already,  this is the reason for investigating $F$ instead of $\widetilde F$. However, this is not straightforward  due to the logarithm term in the definition of $F(h,\lambda)$ in \eqref{eq-hjb-op}.  This can be handled using the entropy variational formula below. 
		
\subsection{Variational representation of the risk-sensitive Bellman operator $F(\cdot, \lambda )$} \label{sec-sub-ent}
In this section, we apply the entropy variational formula (see Proposition~\ref{prop-ent-var} in Appendix~\ref{app-ent-var}) to the operator $F(\cdot,\lambda)$, which will  help us re-write the logarithm term in the definition of $F$ in a form that is important for our analysis.

\begin{lemma}\label{cor-F-var-rep}
For any $h\in \RR^n$, $\lambda\in \RR$ and $1\leq i\leq n$, we have  
\begin{equation}\label{eq-F-var-rep} F_i(h,\lambda) = \min_{u\in \bU(i)}\bigg[\delta c(i,u) + \sup_{q(i,\cdot,u) \in \mathcal{P}(S)}\bigg(\sum_{j=1}^{n-1}h(j)q(i,j,u)- R\big(q(i,\cdot, u) || \, p(i,\cdot, u)\big) \bigg) \bigg]-\lambda\,.\end{equation}
Here, $R\big(q(i,\cdot, u) || \, p(i,\cdot, u)\big) = \sum_{j =1}^nq (i,j,u)\log \frac{q(i,j,u)}{p(i,j,u)}\,.$
\end{lemma}
\begin{remark}\label{rem-suppress} Observe that  the supremum in~\eqref{eq-F-var-rep} can be equivalently written as 
\begin{equation} \sup_{q \in \mathcal{P}(S)}\Big(\sum_{j=1}^{n-1}h(j)q(j) -R\big(q(\cdot ) || p(i, \cdot, u)\big) \Big)\,.\end{equation}
However, we choose to write this as the supremum over $q(i,\cdot,u)\in \calP(S)$ (in~\eqref{eq-F-var-rep} and in what follows)  to aid the reader in keeping track of the underlying dependency of the optimal probability measure $q\in \calP(S)$ on $(i,u)$. 
\end{remark}
\begin{proof}
	To begin with, for any $h\in \RR^n$ and $\lambda\in {\RR}$, define $\tilde h\in \RR^n$ as $\tilde{h}(j) \doteq  h(j) $ for $1\leq j\leq n-1$ and $0$ for  $j=n$.  Then, the expression defining $F(h,\lambda)$ in~\eqref{eq-hjb-op} becomes
	\begin{equation}\label{eq-F-alt-rep} F_i(h,\lambda )= \min_{u\in\bU(i)}\Big[\delta c(i,u)+ \log \E^u_i[e^{\tilde{h}(X_1)}]\Big]-\lambda, \quad \text{{for $1\leq i\leq n$}}\,.\end{equation}
	Applying Proposition~\ref{prop-ent-var} to $\log \E^u_i\left[e^{\tilde{h}(X_1)}\right]$  gives us 
	\begin{equation*}
		\begin{aligned}
			\log \E^u_i\Big[e^{\tilde{h}(X_1)}\Big]&= \sup_{q(i,\cdot,u)\in \mathcal{P}(\mathcal{S})} \bigg[\sum_{j =1}^n\tilde{h}(j)q(i,j,u)-R\big(q(i,\cdot,u) \,||\,p(i,\cdot, u)\big)  \bigg]\\
			&=\sup_{q(i,\cdot,u)\in \mathcal{P}(\mathcal{S})} \bigg[\sum_{j =1}^{n-1}{h}(j)q(i,j,u) -R\big(q(i,\cdot,u) \,||\,p(i,\cdot, u)\big)  \bigg]\,.
		\end{aligned}
	\end{equation*}
	Plugging this back into~\eqref{eq-F-alt-rep}, we obtain the desired result. 
\end{proof}

\subsection{Local contraction of the operator $F(\cdot,\lambda)$}\label{sec-cont}
\begin{proposition}\label{thm-contraction} For $m>0$ and $\lambda\in \RR$, we have \begin{equation}\label{eq-cont-1}
	\| F(h_1,\lambda) - F(h_2,\lambda)\|_m \leq \beta_m\| h_1 - h_2\|_m,
\end{equation}
whenever $\|h_1\|_\infty$, $\|h_2\|_\infty \leq m$. Here, $\beta_m\in (0,1)$ is as defined in~\eqref{eq-contraction_const}.
\end{proposition}
\begin{remark} Observe that the above theorem implies that $F(\cdot,\lambda)$ is only a local contraction as $\beta_m$ depends on $m$ and $\beta_m\to 1$ as $m\to\infty$ by~\eqref{eq-contraction_const}.
\end{remark}
\begin{proof}
	Fix $m>0$, $\lambda\in {\RR}$ and $h_1,h_2\in \RR^n$ such that $\|h_1\|_\infty,\|h_2\|_\infty\leq m$. From Corollary~\ref{cor-F-var-rep}, for $1\leq i\leq n$, we have
	\begin{equation}\label{eq-contraction-1}
		F_i(h_1,\lambda) = \min_{u\in \bU(i)}\bigg[\delta c(i,u)+ \sup_{q(i,\cdot,u) \in \mathcal{P}(S)}\bigg(\sum_{j=1}^{n-1}h_1(j)q(i,j,u)- R\big(q(i,\cdot,u) \,||\,p(i,\cdot, u)\big) \bigg)\bigg]-\lambda\,,\end{equation}
	\begin{equation}\label{eq-contraction-2}
		F_i(h_2,\lambda) = \min_{u\in \bU(i)}\bigg[\delta c(i,u)+\sup_{q(i,\cdot,u) \in \mathcal{P}(S)}\bigg( \sum_{j=1}^{n-1}h_2(j)q(i,j,u)- R\big(q(i,\cdot,u) \,||\,p(i,\cdot, u)\big) \bigg)\bigg]-\lambda\,.
	\end{equation}
	Next, we choose $v_2\in \Usm$ such that it is a minimizer for~\eqref{eq-contraction-2}. 
	This gives us
	\begin{equation}\label{eq-contraction-3}
		\begin{aligned}
			F_i(h_1,\lambda)& - F_i(h_2,\lambda ) \\
			&\leq  \sup_{q(i,\cdot,v_2(i)) \in \mathcal{P}(S)}\bigg(  \sum_{j=1}^{n-1}h_1(j)q(i,j,v_2(i)) - \sum_{j=1}^nq (i,j,v_2(i))\log \left( \frac{q(i,j,v_2(i))}{p(i,j,v_2(i))} \right)\bigg)\\
			&\qquad  -\sup_{q(i,\cdot,v_2(i)) \in \mathcal{P}(S)}\bigg(\sum_{j=1}^{n-1}h_2(j)q(i,j,v_2(i)) - \sum_{j=1}^nq (i,j,v_2(i))\log \left( \frac{q(i,j,v_2(i))}{p(i,j,v_2(i))} \right) \bigg)\,.
		\end{aligned}
	\end{equation}
	To arrive at the inequality, we use the fact that $v_2\in \Usm$ is sub-optimal for the minimization in~\eqref{eq-contraction-1}. From Proposition~\ref{prop-ent-var}, we know that 
	$$\sup_{q(i,\cdot,v_2(i)) \in \mathcal{P}(S)}\bigg(\sum_{j=1}^{n-1}h_1(j)q(i,j,v_2(i)) - \sum_{j =1}^nq (i,j,v_2(i))\log \left( \frac{q(i,j,v_2(i))}{p(i,j,v_2(i))} \right)  \bigg) $$
	has a unique maximizer $q^*=q^*(i,\cdot, v_2(i))$  given by	
	\begin{equation}\label{eq-contraction-4-0}
		q^*(i,j,v_2(i)) = \frac{e^{\tilde h_1(j)}p(i,j,v_2(i))}{\sum_{k=1}^ne^{\tilde h_1(k)}p(i,k,v_2(i))}\geq e^{-2m}p(i,j,v_2(i))\,.
	\end{equation}
	To get the above inequality, we use $\|h_1\|_\infty \leq m$.  Here,  $\tilde h_1\in \RR^n$ is defined  as $\tilde{h}_1(j) \doteq  h_1(j) $, for $1\leq j\leq n-1$ and $0$, for  $j=n$.	
	Plugging this into~\eqref{eq-contraction-3} and using the fact that $q^*(i,\cdot,v_2(i))$ is sub-optimal for 
	$$ \sup_{q(i,\cdot,v_2(i)) \in \mathcal{P}(S)}\bigg(\sum_{j=1}^{n-1}h_2(j)q(i,j,v_2(i))- \sum_{j =1}^nq (i,j,v_2(i))\log \left( \frac{q(i,j,v_2(i))}{p(i,j,v_2(i))} \right)  \bigg),$$
	we get
	\begin{equation}\label{eq-contraction-4}
		F_i(h_1,\lambda) - F_i(h_2,\lambda) \leq \sum_{j =1}^{n-1}(h_1(j) - h_2(j))q^*(i,j,v_2(i))\,.	\end{equation}
	Following the arguments from the proof of Lemma 3 in \cite{tseng1990solving}, we can conclude that $q^*(i,\cdot,v_2(i))$ satisfies $\sum_{j =1}^{n-1}w^m_jq^*(i,j,v_2(i))\leq \beta_m w^m_i$,  for $1\leq i\leq n$.
	Recall $w^m$ and $\beta_m$ from \eqref{eq-weight_vector} and \eqref{eq-contraction_const}, respectively. 
	Now for the right hand side of~\eqref{eq-contraction-4}, we have 
	\begin{equation}\label{eq-contraction_ineq_1}
		\begin{aligned}
			\sum_{j =1}^{n-1}\frac{(h_1(j) - h_2(j))}{w^m_j}w^m_jq^*(i,j,v_2(i)) 
			\leq \beta_m w^m_i \max_{1\leq j \leq n} \left\{\frac{|h_1(j) - h_2(j)|}{w^m_j}\right\} =  \beta_m w_i^m \norm{h_1 - h_2}_m\,.
		\end{aligned}
	\end{equation}
	Hence we obtain $	F_i(h_1,\lambda) - F_i(h_2,\lambda)  \le  \beta_m w_i^m \norm{h_1 - h_2}_m$. 
	Interchanging the roles of $h_1$ and $h_2$, we obtain  $F_i(h_2,\lambda) - F_i(h_1,\lambda) \leq \beta_m w_i^m \norm{h_1 - h_2}_m.$
	From the above two displays, we have 
	\begin{equation*}
		\frac{|F_i(h_1,\lambda) - F_i(h_2,\lambda)|}{w_i^m} \leq \beta_m  \norm{h_1 -h_2}_m,\quad  \text{ for $1\leq i\leq n$,}
	\end{equation*} 
	which in turn, using the definition of $\|\cdot\|_m$ in~\eqref{def-norm} gives us the desired result. 
\end{proof}
\subsection{ Analysis  of the fixed point of $F(\cdot,\lambda)$}\label{sec-prop-hl}
In this section, we consider the existence of fixed points of $F(\cdot,\lambda)$ and also discuss their properties, in terms of $\lambda$. We remark that the existence of these fixed points does not follow from Proposition~\ref{thm-contraction} as the theorem only gives us the local contraction.  
Recall $\tau_n=  \min\{t \geq 1: X_t = n\}$, and 
define  $\underline \tau\doteq \min_{1\leq i\leq n}\inf_{v\in \Usm}\EE_i^v[\tau_n]$ and $\overline \tau\doteq \max_{1\leq i\leq n}\sup_{v\in \Usm}\EE_i^v[\tau_n]$.
It is well-known that under Assumption~\ref{assump-main}, $\underline \tau\leq \overline \tau<\infty$; see Theorem 4.1 in \cite{Cavazos-Cadena2002}. The following lemma gives us the lower bound of $h_\lambda$, in terms of $\lambda$, whenever $h_\lambda$ is finite.   Recall $\underline c$ from~\eqref{eq-c-r}.
\begin{lemma} \label{lem-hl} For any $\lambda \in \Lambda$ and $1\leq i\leq n$,   $h_\lambda(i)\geq\big(\delta \underline c-\lambda \big)\underline \tau\,.$
\end{lemma}
\begin{proof}
	 From  Proposition~\ref{prop-ESSP}(ii), we know that $ h_\lambda (i)= \inf_{v\in \Usm}\log \E_i^{v} \Big[e^{\sum_{t=0}^{\tau_n - 1}\big(\delta c(X_t, v(X_t)) - \lambda \big)}\Big]\,.$
	Using Jensen's inequality, we have
	\begin{equation*}
		h_\lambda(i)\geq \inf_{v\in \Usm}\E_i^{v} \bigg[\sum_{t=0}^{\tau_n - 1}\big(\delta c(X_t, v(X_t)) - \lambda \big)\bigg]\geq (\delta\underline c-\lambda)\min_{1\leq i\leq n}\inf_{v\in \Usm}\EE_i^{v}[\tau_n]\,. \end{equation*}
	From the definition of $\underline \tau$,  the lemma is proved. 
\end{proof}

Next we investigate the local bi-Lipschitz continuity of the map $\lambda\mapsto h_\lambda$, whenever  $\lambda\in \Lambda$. For  each $m>0$, we let
\begin{equation}\label{eq-lm} \Lambda_m \doteq \big\{\lambda\in \RR: \max_{1\leq i\leq n} h_\lambda (i) \leq m\big\}\subset \Lambda\,.\end{equation}
{ Since the DTMC $X$ is irreducible under all stationary policies by Assumption~\ref{assump-main}, we know that for each $i\in S$ and stationary policy $v$ there exists a finite path of states $i=x_0, x_1, \dots, x_{\ell^v(i) - 1}, x_{\ell^v(i)} = n$ with $\ell^v(i)<\infty$ such that $p(x_j,x_{j+1},v(x_j))>0$ for $j = 0,\dots, \ell^v(j) - 1$. We define
\begin{equation}\label{eq-L-max-transition-path}
	\overline L \doteq \max_{1\leq i\leq n}\sup_{v\in \mathfrak{U_{sm}}}\ell^v(i)\,.
\end{equation}
Let $\alpha^v(i) \doteq \Pi_{j=0}^{\ell^v(i) - 1}p(x_j,x_{j+1},v(x_j))\in (0,1)$, and define
\begin{equation}\label{eq-alpha-min-prob}
	\underline \alpha \doteq \min_{1\leq i\leq n} \inf_{v\in \mathfrak U_{sm}}\alpha^v(i). 
\end{equation}
Under Assumption~\ref{assump-main}, we have that $\overline L <\infty$ and $\underline{\alpha}\in (0,1)$.  
We are now in a position to state our result on the local bi-Lipschitz continuity of the map $\lambda \mapsto h_\lambda.$ 
}

\begin{proposition}\label{lem-bounds-h}
Suppose $\lambda,\lambda'\in \Lambda $ and $\lambda>\lambda'$ and let $m'\doteq \max_{1\leq i\leq n}h_{\lambda'}(i)$ and $ m\doteq \max_{1\leq i\leq n}h_{\lambda}(i)$. Then,  we have
\begin{equation}\label{eq-bounds-h}
	\widetilde  N_*(m,\lambda )(\lambda-\lambda') \leq h_{\lambda'}(i)-h_{\lambda}(i)\leq \widetilde  N^*(m',\lambda ')(\lambda-\lambda'),
\end{equation}
for $1\leq i\leq n$. Here, 
\begin{equation}\label{eq-N}
	\begin{aligned}\widetilde  N^*(m',\lambda')& \doteq \frac{\overline L e^{2\overline L({m'}-(\delta \underline c-\lambda' )\underline \tau)}}{\underline \alpha},\\
		\widetilde N_*(m,\lambda)& \doteq 1+ \min_{1\leq i\leq n} \inf_{v\in \Usm} \EE^v_i\Big[(\tau_n-1) e^{\big(( \delta \underline c-\lambda )\underline \tau-m\big)\tau_n}\Big]\,.
	\end{aligned}
\end{equation}
	\end{proposition}
\begin{proof}
	 We first show that 
	\begin{equation}\label{eq-hl-u}h_{\lambda'}(i)-h_{\lambda}(i)\leq \widetilde  N^*(m,\lambda ')(\lambda-\lambda'), \quad \text{ for \,\, $1\leq i\leq n$}\,. \end{equation}
	From Corollary~\ref{cor-F-var-rep}, we have 
	\begin{equation*}
		h_{\lambda}(i)= \min_{u\in \bU(i)}\bigg[\delta c(i,u)+\sup_{q(i,\cdot,u) \in \mathcal{P}(S)}\bigg( \sum_{j=1}^n\tilde h_{\lambda}(j)q(i,j,u) - \sum_{j =1}^nq (i,j,u)\log \left( \frac{q(i,j,u)}{p(i,j,u)} \right) \bigg)\bigg]{-\lambda} ,
	\end{equation*}
	where $\tilde{h}_{\lambda}(j) \doteq  h_{\lambda}(j) $ for $1\leq j\leq n-1$ and $0$ for  $j=n$. 
	Now choose $v\in \Usm$ such that it is a minimizer of the above equation.  For this choice of $v\in \Usm$, we have 
	\begin{equation}\label{eq-lin-11} \scalemath{0.95}{h_{\lambda'} (i)\leq \delta c(i,v(i))  + \sup_{q(i,\cdot,v(i))\in \calP(S)}\bigg(\sum_{j=1}^n\tilde h_{\lambda'}(j)q(i,j,v(i)) - \sum_{j =1}^nq (i,j,v(i))\log \left( \frac{q(i,j,v(i))}{p(i,j,v(i))} \right) \bigg) -\lambda' \,.} \end{equation}
	Here, $\tilde h_{\lambda'}$ is defined similarly as above. Next, choose $q^*(i,\cdot,v(i))\in \calP(S)$ such that it achieves the supremum in the above display.  Note that $q^*(i,\cdot,v(i))$ is given by 
	{
		$$q^*(i,j,v(i)) = \frac{e^{\tilde h_{\lambda'}(j)}p(i,j,v(i))}{\sum_{j=1}^ne^{\tilde h_{\lambda'}(j)}p(i,j,v(i))},$$
		and so it is clear that $q^*$ satisfies the following upper and lower bounds
		\begin{equation}\label{eq-q-ub}
			e^{-2M'}p(i,j,v(i))\le q^*(i,j,v(i))\leq e^{2M'}p(i,j,v(i)),
		\end{equation}
		where $M' \doteq {m'}-(\delta \underline c-\lambda' )\underline \tau$ and we used  the definition of $m'$ from the hypothesis of the lemma and  the  bound from Lemma~\ref{lem-hl}. 
	} 
	From here, this gives us
	\begin{equation}\label{eq-lin-21} h_{\lambda}(i)\geq \delta c(i,v(i))+ \sum_{j=1}^n\tilde h_{\lambda}(j)q^*(i,j,v(i))-\sum_{j =1}^nq^* (i,j,v(i))\log \left( \frac{q^*(i.j,v(i))}{p(i,j,v(i))} \right){-\lambda}  \,.\end{equation}
	It is easy to deduce that $q^*=\big(q^*(i,j,v(i))\big)_{i,j\in S}$ is a Markov transition probability. Let $X^*=\{X^*_t\}_{t=0}^\infty$ denote the associated DTMC on $S$ and $\tau_n^*$ denote the corresponding  first return time to state $n$ starting from $i$ (we have suppressed the dependence on $i$). Finally, to keep the expressions below concise, we write the expectation associated with $q^*$ as $\EE^*$ and denote the  law of $X^*$ by $\QQ^*$. If $X^*_0=i$, then we write $\EE^*_i$ to emphasize this.

	For $T>0$, we apply Dynkin's formula to $h_{\lambda}(X_t^*)$ and $h_{\lambda'}(X_t^*)$ (up to the stopping time $(\tau^*_n-1)\wedge T$), where $h_\lambda(\cdot)$ and $h_{\lambda'}(\cdot)$ satisfy~\eqref{eq-lin-11} and~\eqref{eq-lin-21}, respectively. Fixing $i\neq n$, from Dynkin's formula, we obtain
	\begin{equation}\label{eq-line-31}
		h_{\lambda'}(i) = \tilde h_{\lambda'}(i) =  \E_i^*\big[ \tilde h_{\lambda'}(X^*_{\tau_n^* \wedge T })\big] + \E_i^*\Bigg[\sum_{t=0}^{(\tau_n^* - 1) \wedge T}-\mathcal{A}^* \tilde h_{\lambda'}(X^*_t)\Bigg],
	\end{equation}
	where $\mathcal A^*$ is the one step generator for the DTMC $X^*$, \emph{i.e.}, $\mathcal A^* f(i)\doteq  \E_i^*[f(X^*_1)] - f(i)$ for any $f:S\rightarrow \RR$. Note that we can replace $h_{\lambda'}(x)$ with $\tilde h_{\lambda'}(x)$ inside the sum of the generators. The same procedure holds if we take $i=n$. From the definition of $\mathcal A^* \tilde h_{\lambda'}(x) $ and the inequality in \eqref{eq-lin-21}, we have that 
	\begin{equation}\label{eq-line-41}
		\mathcal A^* \tilde h_{\lambda'}(x) \geq -\delta c(x, v(x)) + \lambda' + R\big(q^*(x,\cdot, v(x))|| \, p(x, \cdot, v(x)) \, \big).
	\end{equation}
	Combining \eqref{eq-line-31} and \eqref{eq-line-41} gives us
	\begin{equation*}
		h_{\lambda'}(i)\leq  \EE_i^*\Bigg[\sum_{t=0}^{(\tau^*_n-1)\wedge T} \bigg(\delta c(X^*_t,v(X^*_t))- R\big(q^*(X_t^*, \cdot, v(X_t^*)) \| \, p(X_t^*, \cdot, v(X^*_t))\big) - \lambda'\bigg)+ \tilde h_{\lambda'} (X^*_{\tau^*_n\wedge T})\Bigg].
	\end{equation*}
	An analogous approach can be used to establish 
	\begin{equation*}
		h_{\lambda}(i)\geq  \EE_i^*\Bigg[\sum_{t=0}^{(\tau^*_n-1)\wedge T} \big(\delta c(X^*_t,v(X^*_t)) - R\big(q^*(X_t^*, \cdot, v(X^*_t)) \| \, p(X_t^*, \cdot, v(X^*_t))\big) -\lambda \big)+ \tilde h_{\lambda} (X^*_{\tau^*_n\wedge T})\Bigg]\,.
	\end{equation*}
	From the above, we clearly have
	\begin{equation}\label{eq-lb-ret-1}
		h_{\lambda'}(i)-h_{\lambda}(i)\leq\EE_i^*\Big[ \tau^*_n\wedge T\Big](\lambda-\lambda')+ \EE_i^*\Big[\tilde h_{\lambda'}(X^*_{\tau^*_n\wedge T})-\tilde h_{\lambda}(X^*_{\tau^*_n\wedge T})\Big]\,.
	\end{equation}
	From the monotone convergence theorem, we have
	$ \EE_i^*[\tau^*_n\wedge T]\uparrow \EE_i^*[\tau^*_n], \text{ as $T\uparrow \infty$}\,.$
	Owing to this fact, we take $T\uparrow \infty$ in~\eqref{eq-lb-ret-1}. This gives us
	\begin{equation*}
		h_{\lambda'}(i)-h_{\lambda}(i){\leq}\EE_i^*\big[ \tau^*_n\big](\lambda-\lambda')+ \limsup_{T\to\infty}\EE_i^*\Big[|\tilde h_{\lambda'}(X^*_{\tau^*_n\wedge T})|+|\tilde h_{\lambda}(X^*_{\tau^*_n\wedge T})|\Big]\,.
	\end{equation*} 
	The term with $\limsup$ above goes to zero as $T\uparrow \infty$ as $\tilde h_\lambda(n)$ and $\tilde h_{\lambda'}(n)$ are zero. To summarize, until now, we have shown that  
	\begin{equation}\label{eq-lb-tau-0}
		h_{\lambda'}(i)-h_{\lambda}(i)\leq\EE_i^*\big[ \tau^*_n\big](\lambda-\lambda')\,.
	\end{equation} 
	{ Below we obtain a finite bound for $\EE_i^*[\tau_n^*]$ in terms of the original transition kernel and $M'= {m'}-(\delta \underline c-\lambda' )\underline \tau$.
		Recall the definition of the constants  $\overline L$ and $\underline \alpha$ in \eqref{eq-L-max-transition-path} and~\eqref{eq-alpha-min-prob}, respectively. From the lower bound on the transition kernel $q^*$ in \eqref{eq-q-ub}, we have that
		\begin{equation}\label{eq-tau-prob-bound-init}
			\mathbb Q_i^*\big(\tau^*_{n} \le \overline L\big) \ge \Pi_{j=0}^{\ell^v(i)-1}q^*(x_j,x_{j+1},v(x_j)) \ge e^{-2M'\overline L}\Pi_{j=0}^{\ell^v(i) - 1}p(x_j,x_{j+1},v(x_j)) \ge e^{-2M'\overline L} \underline \alpha\,.
		\end{equation}
		Under a similar argument as above and inducting on $k\ge 1$,
		we have that  
		\begin{equation}\label{eq-tau-prob-bound}
			\mathbb Q_i^*\big(\tau_n^* > k\overline L\big) \le \big(1 - e^{-2M'\overline L}\underline \alpha\big)^k\,.
		\end{equation}
		We prove this below. Note that the base case holds because of \eqref{eq-tau-prob-bound-init}. We now verify the induction step. Suppose that $\eqref{eq-tau-prob-bound}$ holds for some $k\ge 1$. We observe that
		\begin{align*}
			\mathbb Q_i^*\big(\tau_n^* > (k+1)\overline L\big) &= 	\mathbb Q_i^*\big(\tau_n^* > (k+1)\overline L \, \, \big| \tau^*_n > k\overline L\big)\, \mathbb Q_i^*(\tau^*_n > k\overline L)\\
			&=  \E_i^*\big[\mathbb Q^*\big(\tau^*_n > \overline L\, \big| X^*_0 = X^*_{k \overline L}\big) \big] \, \mathbb Q^*_i(\tau^*_n > k\overline L)\\
			&\le \big(1 - e^{-2M'\overline L}\underline \alpha\big) \cdot \big(1 - e^{-2M'\overline L}\underline \alpha\big)^k\\
			&= \big(1 - e^{-2M'\overline L}\underline \alpha\big)^{k+1}\,.
		\end{align*}
		In the above, to get the first equality we conditioned on the event $\tau^*_n > k \overline L$. To get the second equality, we used the strong Markov property and the fact that the process is stationary (where with slight abuse of notation we used $X_0^* = X^*_{k\overline L}$ to denote a new Markov process starting from  $X^*_{k\overline L}$ ). Lastly, to get the third line, we used \eqref{eq-tau-prob-bound-init} and the inductive hypothesis. This proves \eqref{eq-tau-prob-bound}.
		
		Using \eqref{eq-tau-prob-bound}, we bound the mean of $\tau^*_n$ as follows
		\begin{equation}\label{eq-tau-mean-bound}
			\begin{aligned}
				\E_i^*\big[\tau_n^*\big] &= \sum_{k\ge 0}\mathbb Q^*\big(\tau_n^* > k\,: X_0^* = i\big)\\
				&\le \overline L \sum_{k\ge 0}\mathbb Q^*\big(\tau_n^* > k\overline L\,: X_0^* = i\big)\\
				&\le \overline L \sum_{k\ge 0}\big(1 - e^{-2M'\overline L}\underline \alpha\big)^k\\
				&= \frac{\overline L e^{2M'\overline L}}{\underline \alpha}\,.
			\end{aligned}
		\end{equation}
		
		Combining~\eqref{eq-lb-tau-0} and~\eqref{eq-tau-mean-bound} completes the proof of~\eqref{eq-hl-u}.
	}
		
		We now move on to show that \begin{equation}\label{eq-hl-l}h_{\lambda'}(i)-h_{\lambda}(i)\geq \widetilde  N_*(m,\lambda ')(\lambda-\lambda'), \quad {\text{ for $1\leq i\leq n$}}\,. \end{equation} 
		The proof follows along the same lines as the proof of~\eqref{eq-hl-u}. Hence we omit the details and only provide the expression analogous to~\eqref{eq-lb-tau-0}. To do this, let us define a transition probability that is analogous to~\eqref{eq-q-ub}:
		\begin{equation*}
			q_*(i,j,v'(i)) = \frac{e^{\tilde h_\lambda(j)}p(i,j,v'(i))}{\sum_{j=1}^ne^{\tilde h_\lambda(j)}p(i,j,v'(i))}\,.
		\end{equation*} 
		In the above, we choose $v'\in \Usm$ such that 
		$$ h_{\lambda'} (i)=\delta c(i,v'(i)) -\lambda'+  \sup_{q(i,\cdot,v'(i))\in \calP(S)}\bigg(\sum_{j=1}^n\tilde h_{\lambda'}(j)q(i,j,v'(i))- \sum_{j =1}^nq (i,j,v'(i))\log \left( \frac{q(i,j,v'(i))}{p(i,j,v'(i))} \right) \bigg)\,.$$
		
		Let $X_*=\{X_{*,t}\}_{t=0}^\infty$ denote the associated DTMC on $S$ and $\tau_{*,n}$ denote the corresponding  first return time to state $n$ starting from $i$ (we have suppressed the dependence of $i$). Finally, to keep the expressions below concise, we write the expectation associated with $q_*$ as $\EE_*$ and denote the  law of $X_*$ by $\QQ_*$. If $X_{*,0}=i$, then we write $\EE_{*,i}$ to emphasize this.
		
		Using the above construction, the expression analogous to~\eqref{eq-lb-tau-0} is given by
		\begin{equation}\label{eq-lb-tau}
			h_{\lambda'}(i)-h_{\lambda}(i)\geq\EE_{*,i}\big[ \tau_{*,n}\big](\lambda-\lambda')\,.
		\end{equation} 
		
		{Below we obtain a simpler lower bound for $\EE_{*,i}\big[ \tau_{*,n}\big]$:
			\begin{equation*}
				\EE_{*,i}\big[ \tau_{*,n}\big]=1+ \sum_{N=1}^\infty (N-1)\QQ_*(\tau_{*,n}=N)\geq 1+ \sum_{N=1}^\infty (N-1)e^{N( \delta \underline c-\lambda )\underline \tau-mN} \PP(\tau_n=N).
			\end{equation*}
			The above inequality can be proved as follows. The probability $\QQ_*(\tau_n^*=N)$ involves exactly $N$ transitions. On the other hand, using~\eqref{eq-q-ub}, we know that the probabilities of each of the intermediate transitions (under $\QQ_*$)   are bounded from below by $e^{( \delta \underline c-\lambda )\underline \tau-m}$ times the probability of that transition under $\PP$ (the original measure). In other words, we have $\QQ_*(\tau_{*,n}=N)\geq e^{N( \delta \underline c-\lambda )\underline \tau-mN} \PP(\tau_n=N)\,.$

				Now it is clear that the following holds
				$$ \sum_{N=1}^\infty (N-1)e^{N( \delta \underline c-\lambda )\underline \tau-mN} \PP(\tau_n=N)= \EE^{v'}_i\big[(\tau_n-1) e^{\big(( \delta \underline c-\lambda )\underline \tau-m\big)\tau_n} \big]\,.$$
				This in turn implies that 
				\begin{equation}\label{eq-lb-tau-11}\begin{aligned} \EE_i^*\big[ \tau^*_n\big]&\ge 1+  \EE^{v'}_i\big[(\tau_n-1) e^{\big(( \delta \underline c-\lambda )\underline \tau-m\big)\tau_n}\big]\\
						&\geq 1+ \min_{1\leq i\leq n} \inf_{v'\in \Usm} \EE_i^{v'}\big[(\tau_n-1) e^{\big(( \delta \underline c-\lambda )\underline \tau-m\big)\tau_n}\big]= \widetilde N_*(m,\lambda)\,.\end{aligned}\end{equation}
				This completes the proof of the lemma.} 
\end{proof}
	
	\begin{remark}\label{rem-Lipschitz-difficult} 
		In the risk-neutral setting~\citep{Bertsekas_VI_98}, for any $\lambda\in \R$, the fixed point of the average cost Bellman-like operator (denoted by $h^{avg}_\lambda$) satisfies the following stochastic representation:
		$$ h^{avg}_\lambda(i) = \inf_{v\in \mathfrak U_{sm}} \E_i^v \bigg[\sum_{t=0}^{\tau_n - 1}\big(c(X_t, v(X_t)) - \lambda\big)\bigg], \quad i = 1,\dots, n\,.$$
		We review how this can be used to obtain the Lipschitz bounds in the risk-neutral setting. Recall that under Assumption~\ref{assump-main},  the  DTMC is positive recurrent uniformly over all stationary policies. Hence, letting  $\overline  N$ and $\underline N$ denote the maximum and minimum expected return times to state $n$ over all stationary policies (respectively), for any $\lambda' > \lambda$,  we have that the following inequality holds
		$$ \underline N (\lambda ' - \lambda)\le  h_\lambda^{avg}(i) - h_{\lambda'}^{avg}(i) \le \overline N (\lambda ' - \lambda), \quad \forall i = 1,\dots, n, $$
		resulting in the global bi-Lipschitz continuity property of $\lambda \mapsto h_{\lambda}^{avg}$. 
		
		In the risk-sensitive setting, because the stochastic representation for $h_\lambda$ is not additive due to the presence of the exponential moments of the cumulative cost up to $\tau_n$ (see Proposition~\ref{prop-ESSP}), we observe that this approach does not directly extend. If one attempts to use the variational representations to bound $h_\lambda(i) - h_{\lambda'(i)}$, one ends up with an expected return time under a new law corresponding to the maximizer of the variational representation, which may not be finite. It is for these reasons why we used the approach in the proof of Proposition~\ref{lem-bounds-h}.  
	\end{remark}
	
	Proposition~\ref{lem-bounds-h} gives us the following important proposition. Recall  $\Lambda_m$ from~\eqref{eq-lm}.
	\begin{proposition}\label{coro-connectivity_Lambda} The following hold: 
		\begin{enumerate}
			\item [(i)] For $\lambda,\lambda'\in \Lambda$ such that $\lambda{>} \lambda'$, we have
			$h_\lambda(i){<} h_{\lambda'}(i)$, for every $1\leq i\leq n$.
			\item [(ii)]There exists $\lambda_c\in \RR$ such that the set $\Lambda$ defined in~\eqref{eq-Lambda} can be written as  $\Lambda=(\lambda_{c},\, \infty)$. Moreover, it follows that for $1\leq i\leq n$, $$\lim_{\lambda \uparrow \infty }h_{\lambda}(i) = - \infty,\quad \lim_{\lambda \downarrow \lambda_{c}}h_{\lambda}(i) =  \infty\,.$$ 
			\item [(iii)]Let $m^*\doteq \max_{1\leq i\leq n} h^*(i)$. Then, for $m>m^*$ we have
			$$ \Big(\lambda^*-\frac{m-m^*}{\widetilde N_*(m,\lambda^*)},\,\, \infty\Big)\subset \Lambda_m\,.$$
		\end{enumerate} 
	\end{proposition}
	\begin{proof}
		Part (i) follows immediately from the application of  Proposition~\ref{prop-ESSP}(ii) and the fact that   $h_\lambda=\log V_\lambda$ and $h_{\lambda'}=\log V_{\lambda'}$.
		
		To prove part (ii), we apply Lemma~\ref{lem-hl} and Proposition~\ref{lem-bounds-h} to get
		$$ (\delta c-\lambda') \underline \tau \leq h_{\lambda'}(i)\leq h_\lambda(i)- \widetilde N_*(m,\lambda) (\lambda'-\lambda), \quad  \text{ for $1\leq i\leq n$,}$$
		where $m=\max_{1\leq i\leq n} h_\lambda (i)$ and $\widetilde N^*(m,\lambda)$ is defined in~\eqref{eq-N}. From here, we can conclude that for any $\lambda'>\lambda$, $h_{\lambda'}$ exists. In other words, $\Lambda$ is of the form $(\lambda_c,\infty)$ for some $\lambda_c\in \RR$. The rest of the proof for part (ii) follows trivially from here.
		
		To prove part (iii), we apply Proposition~\ref{lem-bounds-h} with $\lambda^*$ and $\lambda\in \Lambda$ such that $\lambda^*\geq \lambda\geq \lambda^*- \frac{m-m^*}{\widetilde N^*(m,\lambda^*)}$. This gives us $$h_{\lambda}(i)-h^*(i)\leq \widetilde N^*(m,\lambda) (\lambda^*-\lambda)\leq \frac{N^*(m,\lambda)(m-m^*)}{N(m,\lambda^*)}\,.$$ Using the fact that $\widetilde N^*(m,\lambda)\leq \widetilde N^*(m,\lambda^*)$ (which follows from the definition) and   the definition of $m^*$, we obtain that 
		$ h_{\lambda}(i)\leq m-m^*+h^*(i)\leq m\,.$
		Taking the maximum over $1\leq i\leq n$ and from the definition of $\Lambda_m$, we conclude the claim in part (iii). 
	\end{proof}

\section{Proof of Theorem~\ref{thm-main-1}} \label{sec-proof-main-1} 

We split the proof into two cases: $\{(V^k,\lambda^k)\}_{k\in \NN}$ is the sequence of  iterate pairs  of either  Algorithm~\ref{alg-rvi-bertsekas-JI} or   Algorithm~\ref{alg-rvi-bertsekas-GS}.   To begin with, recall that for $\lambda\in \Lambda$, $h_\lambda$ satisfies $h_\lambda=F(h_\lambda,\lambda)$ (or  $h_\lambda=G(h_\lambda,\lambda)$, as $F$ and $G$ inherit identical fixed point equations).  We fix $(h^0,\lambda^0)$ such that $ \lambda^*\leq \lambda^0\leq \overline c$, according to the hypothesis of Theorem~\ref{thm-main-1}; see also Remark~\ref{rem-init}.  Since $h^*$ exists, from Proposition~\ref{coro-connectivity_Lambda}(i), we can also infer that $h_{\lambda^0}$ exists, \emph{i.e.}, the operator $F(\cdot,\lambda^0)$ has  a unique fixed point denoted by $h_{\lambda^0}.$   

\subsection{Proof of~\eqref{eq-itr-contraction-1} in the case of  Algorithm~\ref{alg-rvi-bertsekas-JI}}\label{sec-proof-JI}

Set
\begin{equation}\label{def-m}
	\begin{aligned}
		m= \frac{1}{w^1_0}\Big\{&\big(\|h^0-h^*\|_{\infty}+\|h^*\|_{\infty}+ |\lambda^*-\lambda^0|\big)\vee \max_{1\leq i\leq n}h_{\lambda^0}(i)\Big\},\\
		N_*(m) &\doteq \widetilde N_*(m, \bar c), \quad\quad \quad N^*(m) \doteq \widetilde N^*(m, \bar c),\\
		\alpha_0&\doteq \big(\|h^0-h_{\lambda^0}\|_m\big) \vee \big(N_*(m) (\lambda^0-\lambda^*)\big)\,.
	\end{aligned}
\end{equation}
	We begin by proving the following result concerning the one-iteration behavior of Algorithm~\ref{alg-rvi-bertsekas-JI}.
	\begin{proposition}\label{prop-c-alpha0}
		Suppose for any $k\in \ZZ_+$,  $0<\alpha\leq \alpha_0$ and the following holds: 
		\begin{equation}\label{eq-alpha2}
			\begin{aligned}
				\|h^{k}-h_{\lambda^k}\|_m\leq \alpha  \quad&\text{and}\quad	|\lambda^{k}-\lambda^*|\leq  \frac{\alpha}{N_*(m)},\\
				\|h^{k}\|_\infty\leq m\quad&\text{and}\quad \lambda^{k}\in \Lambda_m\,.	
			\end{aligned}
		\end{equation}
		Then, whenever $|\lambda^{k+1}-\lambda^*|\leq N_*(m)^{-1} \alpha$ and
		\begin{equation}\label{eq-gamma-cond}\gamma_k\leq \Big(\frac{m-m^*}{ N_*(m)}\Big)\Big( \frac{1}{\frac{\alpha_0}{N_*(m)}+\overline c-\underline c+\varkappa(m)}\Big),\end{equation} 
		we have $\|h^{k+1}\|_\infty\leq m$ and $\lambda^{k+1}\in \Lambda_m$. Here, {$\varkappa(m)\doteq \max_{u\in \bU(n)}\{ (1-p(n,n,u))m\}$.} 
		
	\end{proposition}
	\begin{proof}
		 To prove that $\|h^{k+1}\|_\infty\leq m$, using~\eqref{norm-rel}, we get
		\begin{equation*}
			\begin{aligned}
				\|h^{k+1}\|_\infty\leq \|h^{k+1}\|_m&\leq  \|F(h^k,\lambda^*) - F(h^*,\lambda^*) \|_m+ \|h^*\|_m+ |\lambda^k-\lambda^*|\\
				&\leq \beta_m \|h^k-h^*\|_m+\|h^*\|_m + |\lambda^0-\lambda^*|\\
				&\leq \frac{1}{w^0_1}\big( \|h^k-h^*\|_\infty+\|h^*\|_\infty)+ |\lambda^0-\lambda^*|\,.
			\end{aligned}
		\end{equation*} 
		In the above, to get the first line, we use the definition of $h^{k+1}$, the fact that $h^*$ satisfies $h^*=F(h^*,\lambda^*)$ and the triangle inequality; the second and third lines are consequences of Proposition~\ref{thm-contraction}, the fact that $\beta_m < 1$  and $|\lambda^k-\lambda^0|\leq |\lambda^0-\lambda^*|$ (which is clear from the hypothesis),  and~\eqref{norm-rel}. From the definition of $m$, we can infer that $\|h^{k+1}\|_\infty\leq m$.
		
		To prove that $\lambda^{k+1}\in \Lambda_m$, we consider two cases: \\
		{\noindent \bf Case (a): $\lambda^k\leq \lambda^*$.} Since $|\lambda^{k+1}-\lambda^*|\leq N_*(m)^{-1} \alpha$ from the hypothesis, it is clear that $\lambda^{k}\leq \lambda^{k+1}$ and from Proposition~\ref{coro-connectivity_Lambda}(i), we know that $h_{\lambda^{k+1}}$ exists and $ h_{\lambda^{k+1}}(i)\leq h_{\lambda^{k}}(i)\leq m,$ for $1\leq i\leq n$. This proves that $\lambda^{k+1}\in \Lambda_m$.
		
		{\noindent \bf Case (b): $\lambda^k>\lambda^*$.}    We further split the proof into two sub-cases: (b1) $h^{k+1}(n)> 0$ and (b2) $h^{k+1}(n)\leq 0$. Under Case (b1), 
		the proof is similar to that of proof in Case (a) as $\lambda^{k}< \lambda^{k+1}$. 
		
		For  Case (b2), we obtain a lower bound on $h^{k+1}(n)$  as follows: let $u^*\in\bU(n)$ be such that 
		\begin{equation*}
			{ \begin{aligned}
					h^{k+1}(n)&= c(n,u^*)  +  \log \Big(p(n,n,u^*) + \sum_{j =1}^{n-1} e^{h^k(j)}p(n,j,u^*)\Big){-\lambda^k}\\
					&\geq \underline c+ \sum_{j =1}^{n-1} {h^k(j)}p(n,j,u^*){-\lambda^k}\\
					&\geq \underline c - (1-p(n,n,u^*))m{-\lambda^k}\\
					&\geq \underline c-\lambda^k-\lambda^*+\lambda^* -\varkappa(m)\\
					&\geq -\frac{\alpha}{N_*(m)}-\overline c+\underline c-\varkappa(m)\,.
			\end{aligned}}
		\end{equation*}
		In the above, to get the first inequality, we use Jensen's inequality; to get the second inequality, we use the fact that $\|h^k\|_\infty\leq m$; { to get the third inequality, we add and subtract $\lambda^*$ together with  $\varkappa(m)\ge (1-p(n,n,u^*))m$}, and to get the fourth inequality, we use the second inequality in~\eqref{eq-alpha2} and the fact that $\underline c\leq \lambda^*\leq \overline c$. We then obtain \begin{equation*}
			\lambda^{k+1}-\lambda^*= \lambda^k-\lambda^* +\gamma_kh^{k+1}(n)\geq \gamma_k\Big(-\frac{\alpha}{N_*(m)}-\overline c+\underline c-\varkappa(m)\Big)\,.
		\end{equation*}
		From Proposition~\ref{coro-connectivity_Lambda}(iii), the fact that $\alpha\leq \alpha_0$ and the hypothesis of the lemma, $\lambda^{k+1}\in \Lambda_m$.
		This completes the proof. 
	\end{proof}
	
	\begin{proposition}\label{prop-c-alpha}
		Suppose for any $k\in \ZZ_+$ and  $0<\alpha\leq \alpha_0$,~\eqref{eq-alpha2} holds.
		Then, we have
		\begin{equation}\label{eq-c-alpha2}
			\|h^{k+1}-h_{\lambda^{k+1}}\|_m\leq \widetilde c_h \alpha, \quad\text{and}\quad	|\lambda^{k+1}-\lambda^*|\leq \widetilde  c_\lambda  \frac{\alpha}{N_*(m)},
		\end{equation} for some constants $\widetilde c_h=\widetilde c_h(\gamma_k,m)$, $\widetilde c_\lambda=\widetilde  c_\lambda(\gamma_k,m)$, such that  $\widetilde  c_h,\widetilde c_\lambda\in (0,1)$, whenever $\gamma_k\in (0, \widetilde  \gamma]$, for some $\widetilde\gamma= \widetilde \gamma(m)>0 $.			
	\end{proposition}
	The proof of Proposition~\ref{prop-c-alpha} follows very closely the arguments used in the risk-neutral setting from Proposition 2.1 in \cite{Bertsekas_VI_98} with the only difference being that  the constants involved are dependent on $m$.  We however, provide the  proof for the sake of completeness in Appendix~\ref{app-c-alpha2}.

	We now show that Propositions~\ref{prop-c-alpha0} and~\ref{prop-c-alpha} together imply the first and second inequalities in~\eqref{eq-itr-contraction-1}: from Proposition~\ref{prop-c-alpha0}, we know that $\|h^k\|_\infty\leq m$, $\lambda^k\in \Lambda_m$, for every $k\in\ZZ_+$.
	Hence, we obtain 
	\begin{equation*}
		\begin{aligned}
			\|h^{k+1}-h^*\|_m&\leq \|h^{k+1}-h_{\lambda^{k+1}}\|_m+ \|h_{\lambda^{k+1}}-h^*\|_m\\
			&\leq  \widetilde c_h(\gamma_k,m) \|h^k-h^*\|_m+ \widetilde  c_h(\gamma_k,m) \|h_{\lambda^k}-h^*\|_m  +  \|h_{\lambda^{k+1}}-h^*\|_m\\
			&\leq \widetilde  c_h(\gamma_k,m) \|h^k-h^*\|_m+ { \frac{\big(\widetilde c_h(\gamma_k,m)+\widetilde c_\lambda(\gamma_k,m)\big)N^*(m)}{w^0_1}|{\lambda^k}-\lambda^*|}\,. 
		\end{aligned}
	\end{equation*}
	In the above, to get  the second inequality we use~\eqref{eq-c-alpha2} and the triangle inequality; to get the third inequality, we apply Proposition~\ref{lem-bounds-h} to the pairs $(h_{\lambda^k},h^*)$ and $(h_{\lambda^{k+1}},h^*)$ and use~\eqref{eq-c-alpha2}.
	
	Therefore, the first and second inequalities of~\eqref{eq-itr-contraction-1} are implied by setting $\bar \gamma=\min\{\widehat \gamma,\widetilde \gamma\}$, $c_h(\gamma,m)=\widetilde c_h(\gamma,m)$, $c_\lambda(\gamma,m)=\widetilde c_\lambda(\gamma,m)$ and  $$L(m)=\frac{\big(\widetilde c_h(\gamma_k,m)+\widetilde c_\lambda(\gamma_k,m)\big)N^*(m)}{w^0_1}\,.$$

	\subsection{Proof of~\eqref{eq-itr-contraction-1} in the case of Algorithm~\ref{alg-rvi-bertsekas-GS}} \label{sec-proof-GS}
	In this section, we suppose that $\{(V^k,\lambda^k)\}_{k\in\NN}$ is the sequence of iterate pairs of Algorithm~\ref{alg-rvi-bertsekas-GS}. Again, we provide the proof in terms of $h^k=\log V^k$.  Recall operator $G$ defined in~\eqref{def-G}.
	Also, recall that in terms of operator $G$ and $(h^k,\lambda^k)$, Algorithm~\ref{alg-rvi-bertsekas-GS} can be equivalently expressed as in~\eqref{eq-G-itr}.
	
	We first show that $G(\cdot, \lambda)$ satisfies a local contraction property for any fixed $\lambda \in \Lambda$, by using the local contraction property of the operator $F(\cdot,\lambda)$  represented according to~\eqref{eq-F-alt-rep} and an argument which parallels the argument in the proof of Proposition~\ref{thm-contraction}. Before we proceed, we give a lemma, which provides a bound on $G_i(h,\lambda)-\lambda$ in terms $\|h\|_\infty$, for $1\leq i\leq n$. 
	Since $\eta$, defined in~\eqref{eq-transition_prob_const}, is strictly positive,  we have $\eta n\leq  \sum_{j=1}^n p(i,j,u)=1$. This means that  \begin{equation}\label{eq-bareta}
		0<\bar \eta\doteq1-(n-1)\eta <1\end{equation} and recall $\bar c$ from~\eqref{eq-c-r}.
	\begin{lemma}\label{lem-bound-G}For  $1\leq i\leq n$, $\lambda\in \RR$ and $m> 0$, let  $h\in \RR^n$ be such that $\|h\|_\infty\leq m$. Then, we have the following:
		\begin{equation}\label{eq-G-b} -  \frac{1}{1-\bar \eta}\leq \frac{G_i( h,\lambda)}{\bar c+\|h\|_\infty+|\lambda|}\leq  \frac{1-\chi(m+|\lambda|)^{n-1}}{1-\chi(m+|\lambda|)}\,.\end{equation}
		Here, for $l\geq 0$,
		$$\chi(l)\doteq  \max_{(i,u)\in S_\bU}\frac{(1-p(i,n,u))e^{\bar c +l}}{p(i,n,u)+(1-p(i,n,u)) e^{\bar c +l }}\,.$$
	\end{lemma} 
	\begin{remark}Observe that $\chi(l)<1$ for every $l>0$ and $\lim_{l\to\infty}\chi(l)= 1$.
	\end{remark}
	\begin{proof}
		 Fix $\lambda\in \RR$ and $h\in \RR^n$ satisfying $\|h\|_\infty\leq m$. Now for $1\leq i\leq n$, choose $v\in \Usm$ such that $v(i)$ achieves the minimum in the definition of $G_i(h,\lambda)$. For $1\leq i\leq n$, we apply Proposition~\ref{prop-ent-var} for the following choices:  $\calX=S$, $P=p(i,\cdot,v(i))$ and $f(\cdot)=\bar G^i(\cdot)$  with $\bar G^i: \RR^n\rightarrow \RR$ defined as 
		$$ \bar G^i(j) = \begin{cases}
			G_j(h, \, \lambda), &j \leq i-1,\\
			h(j),\, &i \leq j \leq n-1,\\
			0, \, &j = n,
		\end{cases}$$
		for $1<i\leq n$ and $\bar G^1(\cdot)=h(\cdot)$.
		This gives us
		\begin{equation*}
			\scalemath{0.95}{\begin{aligned}
					G_1(h,\lambda) &=  c(1,v(1)) + \sup_{q(1,\cdot,v(1)) \in \mathcal{P}(S)}\bigg(  \sum_{j=1}^{n-1}h(j)q(1,j,v(1)) - \sum_{j =1}^n q (1,j,v(1))\log \left( \frac{q(1,j,v(1))}{p(1,j,v(1))} \right) \bigg) -\lambda,\\	
					G_i(h,\lambda) &= c(i,v(i)) +\sup_{q(i,\cdot,v(i)) \in \mathcal{P}(S)}\bigg( \sum_{j=1 }^{i-1} G_j(h,\lambda)q(i,j,v(i))  +\sum_{j=i}^{n-1}h(j)q(i,j,v(i))\\
					&\quad\quad  - \sum_{j =1}^nq (i,j,v(i))\log \left( \frac{q(i,j,v(i))}{p(i,j,v(i))} \right) \bigg)-\lambda,
			\end{aligned}}
		\end{equation*}
		for $2\leq i\leq n$. It is clear from choosing $q(i,j,v(i))= p(i,j,v(i))$ that we get
		\begin{equation}\label{eq-G1-l}
			\begin{aligned}
				G_1(h,\lambda) &\geq c(1,v(1))  + \sum_{j=1}^{n-1}h(j)p(1,j,v(1))-|\lambda|, \\
				G_i(h,\lambda) &\geq   c(i,v(i))   + \sum_{j=1 }^{i-1} G_j(h,\lambda)p(i,j,v(i))  +\sum_{j=i}^{n-1}h(j)p(i,j,v(i))-|\lambda|,
			\end{aligned}
		\end{equation}
		for $2\leq i\leq n$. From the first inequality in~\eqref{eq-G1-l}, it is clear that
		\begin{equation}\label{eq-G1-lb}
			G_1(h,\lambda)\geq -\bar c-\|h\|_\infty-|\lambda|\,.
		\end{equation}
		Now let us compute the lower bound on $G_2(h,\lambda)$. From the second inequality in~\eqref{eq-G1-l} and~\eqref{eq-G1-lb}, we get
		\begin{equation*}
			\begin{aligned}
				G_2(h,\lambda)&\geq c(2,v(2))   + G_1(h,\lambda)p(2,1,v(2))  +\sum_{j=2}^{n-1}h(j)p(2,j,v(2))-|\lambda|\\
				&\geq  -(\bar c+\|h\|_\infty+|\lambda|) - (\bar c+\|h\|_\infty+|\lambda|) p(2,1,v(2))\\
				&\geq  -(\bar c+\|h\|_\infty+|\lambda|) - (\bar c+\|h\|_\infty+|\lambda|)\bar \eta\,.
			\end{aligned}
		\end{equation*}
		Recall $\bar \eta$ from~\eqref{eq-bareta}. In the above, to arrive at the second line, we use the definition of $\bar c$, $\|h\|_\infty$ and the fact that $\sum_{j=2}^{n-1}p(2,j,v(2))<1$.  Similarly, for any $i\geq 2$, we obtain 
		\begin{equation}\label{eq-Gi-lb}G_i(h,\lambda)\geq - (\bar c+\|h\|_\infty+|\lambda|) \big(1+\bar \eta+\bar \eta^2+\ldots+\bar \eta^{i-1}\big)\geq -\frac{\bar c +m+|\lambda|}{1-\bar \eta}\,.\end{equation}
		To get the second inequality, we use the fact that $\bar \eta<1$;  see~\eqref{eq-bareta}. Therefore,~\eqref{eq-G1-lb} and~\eqref{eq-Gi-lb} together give us the desired lower bound in~\eqref{eq-G-b}.

		We next move on to the upper bound in~\eqref{eq-G-b}.  It is easy to see  that 
		\begin{equation}\label{eq-G1bound} G_1(h,\lambda)\leq \bar c +\|h\|_\infty +|\lambda|\,. \end{equation}
		Now we consider the case $i=  2$. Again, with the same $v=v(\cdot)$ as above,  we have
		\begin{equation}\label{eq-log-1} 
			\begin{aligned}\nonumber
				G_2(h,\lambda) &=  c(2,v(2)) + \log \Big(p(2,n,v(2)) +  e^{G_1(h,\lambda)}p(2,1,v(2))  +\sum_{j=2}^{n-1}e^{h(j)}p(2,j,v(2))\Big) -\lambda\\\nonumber
				&\leq \bar c +|\lambda| + \log  \Big(p(2,n,v(2)) +  e^{G_1(h,\lambda)}p(2,1,v(2))  +\sum_{j=2}^{n-1}e^{h(j)}p(2,j,v(2))\Big) \\\nonumber
				&\leq \bar c +|\lambda| + \log  \Big(p(2,n,v(2)) +  e^{G_1(h,\lambda)\vee \|h\|_\infty}  \sum_{j=1}^{n-1}p(2,j,v(2))\Big)\\
				&\leq \bar c +|\lambda| + \log  \Big(p(2,n,v(2)) +  e^{G_1(h,\lambda)\vee \|h\|_\infty}\big( 1-p(2,n,v(2))\big)\Big)\,.
			\end{aligned} 
		\end{equation}
		In the above, the first inequality is obtained from the definition of $\bar c$; the third inequality is obtained from the fact that $\sum_{j=1}^n p(2,j,v(2))=1$. Now for $\rho>0$, consider the following function $f(x)\doteq \log \big (\rho +(1-\rho)e^x\big)$. It is clear that for $x\in \RR$,
		$$ f'(x)= \frac{(1-\rho)e^x}{\rho+(1-\rho) e^x}\leq 1\,.$$
		From here, using the mean value theorem and  the fact that $f(0)=0$, we have $f(x)\leq f'(x) x$, for $x\geq 0$. Using this,~\eqref{eq-log-1} and the definition of $\chi(\cdot ) $ from the hypothesis of the lemma, we have
		\begin{equation*}
			\begin{aligned}
				G_2(h,\lambda)&\leq \bar c +|\lambda| + \chi(m+|\lambda|) \big(\bar c +\|h\|_\infty +|\lambda| \big)\vee \|h\|_\infty\\
				&\leq  \bar c +\chi(m+|\lambda|) \|h\|_\infty + |\lambda| + \chi(m+|\lambda|) \big(\bar c +\|h\|_\infty +|\lambda| \big) \\
				&\leq \bar c +\|h\|_\infty+ |\lambda| + \chi(m+|\lambda|) \big(\bar c +\|h\|_\infty +|\lambda| \big)\,.
			\end{aligned}
		\end{equation*}
		In the above, to get the first inequality, we use the fact that $f(x)\leq f'(x)x$, for the function $f$ defined above, the bound on $G_1$ from~\eqref{eq-G1bound} and the definition of $\chi(\cdot)$; to get the second inequality, we use the fact that $a\vee b\leq a+ b$, for non-negative real numbers $a$ and $b$; to get the final inequality, we use the fact that $\chi(m+|\lambda |)\leq 1$.
		
		Following this argument similarly for $2<i\leq n$,  we obtain the desired upper bound in~\eqref{eq-G-b}. This completes the proof of the lemma.
	\end{proof}
	Just like in the case of Algorithm~\ref{alg-rvi-bertsekas-JI}, local Lipschitz continuity of $G(\cdot,\cdot)$ is crucial for the proof of~\eqref{eq-itr-contraction-1}. However, due to the iterative structure of the definition of $G(h,\lambda)$, the proof is more involved than in the case of $F(h,\lambda)$. For example, local Lipschitz continuity of $G(h,\cdot)$ is also not at all immediate, unlike the local Lipschitz continuity of $F(h,\cdot)$, which directly follows from its definition.

	\begin{proposition}\label{prop-contraction-GS}
		Fix $\lambda_1, \lambda_2 \in \Lambda$ and $l>0$. Let $h_1,h_2\in \RR^n$ be such that $\|h_1\|_\infty,\|h_2\|_\infty\leq l$ and 
		$$m\doteq \big(\bar c + l + |\lambda_1|\vee|\lambda_2|\big)\max\Big\{\frac{1}{1-\bar \eta}, \frac{1-\chi(l+|\lambda_1|\vee |\lambda_2|)^{n-1}}{1-\chi(l+|\lambda_1| \vee |\lambda_2|)} \Big\},$$ 
		with $\chi(\cdot)$ as defined in Lemma~\ref{lem-bound-G}. Then, we have  
		\begin{equation}\label{eq-contraction-GS}
			\norm{G(h_1,\lambda_1) - G(h_2,\lambda_2)}_{m} \leq \beta_{ m} \norm{h_1 - h_2}_{m} + \Delta^{m}|\lambda_1 - \lambda_2|,
		\end{equation}
		where $\Delta^{m} \doteq  \max_{1\leq i\leq n}\Delta^{ m}_i$, and $\{\Delta^{ m}_i\}_{1\leq i\leq n}$ is defined recursively as follows:
		\begin{equation}
			\Delta^{ m}_1 = \frac{1}{w_1^{ m}}, \quad \Delta^{m}_i = \frac{1 + \max_{1\leq j \leq i-1}\Delta^{m}_j}{w^{m}_i}, \quad 2\leq i\leq n\,.
		\end{equation}
	\end{proposition}
	The proof of this result is given in Appendix~\ref{app-proof-contraction-GS}.

	\subsubsection{ Completing the proof of~\eqref{eq-itr-contraction-1} in the case of Algorithm~\ref{alg-rvi-bertsekas-GS}} 
	As mentioned earlier, the proof of Theorem~\ref{thm-main-1} for Algorithm~\ref{alg-rvi-bertsekas-GS} follows from the same arguments as those used in the proof of Theorem~\ref{thm-main-1} for Algorithm~\ref{alg-rvi-bertsekas-JI}. Hence, we only discuss the differences and omit the proof, with the main difference being the explicit value of the constant 
	$\widetilde c_h(\gamma,m)$ (appearing in Lemma~\ref{lemma-geom_conv_h_k}). Therefore, to avoid repetition of the arguments  for Algorithm~\ref{alg-rvi-bertsekas-JI}, we simply provide the explicit value of the analogous constant in this case. Also, to avoid introducing new notation we simply use the existing notation. Set
	\begin{equation}\label{def-m-2}
		m=  \frac{1}{w^1_0}\Big\{\big(\|h^0-h^*\|_{\infty}+\|h^*\|_{\infty}+ \widetilde \Delta |\lambda^*-\lambda^0|\big)\vee \max_{1\leq i\leq n}h_{\lambda^0}(i)\Big\},
	\end{equation}
	\begin{equation*}\\\nonumber
		\widetilde c_h(\gamma,m)=\beta_m+ \gamma \Big(\Delta^m+\frac{\beta_m N^*(m)}{w^0_1}\Big) \Big( \beta_m+ \frac{N^*(m)}{N_*(m)}\Big)\,.
	\end{equation*}
	Here, $\widetilde \Delta\doteq \sup_{m>0} \Delta^m$ which can be shown to be finite from the definition of $\{w^m_i\}_{1\leq i\leq n}$ in~\eqref{eq-weight_vector}.
	
	\subsection{Completing the proof of Theorem~\ref{thm-main-1}} From Sections~\ref{sec-proof-JI} and~\ref{sec-proof-GS}, the first part of Theorem~\ref{thm-main-1} follows.  To see how the second part of the theorem follows from the first part, we observe that the first part implies the following: 
	$$ \|h^{k}-h^*\|_{ m} \leq C_0 e^{-C_1 k},$$
	for some constants $C_0=C_0( m),C_1=C_1( m)>0$ with $m$ to be chosen accordingly,  \emph{i.e.}, $m$ is given by~\eqref{def-m} for Algorithm~\ref{alg-rvi-bertsekas-JI} and by~\eqref{def-m-2} for Algorithm~\ref{alg-rvi-bertsekas-GS}. Therefore, in terms of $V^k$ and $V^*$, this gives us 
	\begin{equation*}
		\|\mathscr{V}^k\|_m\leq C_0 e^{-C_1 k} \quad \text{ with }\,\, \mathscr{V}^k(i)\doteq\log \frac{V^k(i)}{V^*(i)}, \quad \text{ for\,\,  $1\leq i \leq n$\,.}
	\end{equation*}
	From the definition of $\|\cdot\|_{ m}$ in~\eqref{def-norm}, supposing that for a large $k$, $V^k(i)\geq V^*(i)$, then $ \frac{V^k(i)}{V^*(i)}\leq e^{w^m_iC_0 e^{-C_1 k}}$ which in turn means that $V^k(i) -V^*(i)\leq V^*(i)(e^{w^m_iC_0 e^{-C_1 k}}-1)\leq C_2 w^m_iC_0 V^*(i)e^{-C_1 k}$, for some $C_2=C_2(m)>0$. Similarly, supposing that $V^k(i)\leq V^*(i)$ gives $V^*(i)-V^k(i)\leq C_3 w^m_iC_0 V^*(i)e^{-C_1 k}$, for some $C_3=C_3(m)>0$. To summarize, we have  argued that $\|V^k-V^*\|_m \leq \max\{C_2,C_3\}C_0e^{-C_1 k}$. From here, the second part of the theorem follows.

\section{Numerical Implementations}\label{sec-num-imp}
In this section, we numerically illustrate the performances of Algorithms~\ref{alg-rvi-bertsekas-JI}--\ref{alg-rvi-bertsekas-GS} in two  examples. 
From Theorem~\ref{thm-main-1}, we know that an approximate choice of $\gamma_k$, for every $k$ (in particular, not arbitrary), is sufficient to ensure the convergence of Algorithms~\ref{alg-rvi-bertsekas-JI} and~\ref{alg-rvi-bertsekas-GS}. Our numerical experiments suggest that a judicious choice of  $\gamma_k$, for every $k$, is also necessary for these algorithms to perform well. For instance, 
if the $\gamma_k$'s are too large, Theorem~\ref{thm-main-1} does not guarantee the convergence of Algorithms~\ref{alg-rvi-bertsekas-JI} and~\ref{alg-rvi-bertsekas-GS}. 
However, if the  $\gamma_k$'s are too small, then the iterates generated by the algorithm will update quite slowly, which may lead to a  slower convergence. 
The contingency  on the size of the step-size parameters is analogous to the average cost algorithms on Pg. 746 of~\cite{Bertsekas_VI_98}, and following the author's approach there, we  make a  choice to use a $\gamma_k$ which starts at a value closer to $1$ and slowly vanishes based on the path of the algorithm. The different forms of $\gamma_k$ as well as the corresponding constants for each example are chosen by experimentation. This ensures that the step-sizes do not decrease too quickly, and only decrease after each algorithm has made enough progress in the sense that the signs of $h^k(n)$ oscillate about $0$. { For both examples, we used either $\gamma_k \propto k^{-1}$ or $\gamma_k~\propto \xi^k$ with $\xi\in (0,1)$, depending on the problem and which step-size decay rate facilitated faster convergence. For the example in Section~\ref{sec-exit}, our tests indicate that $\gamma_k \propto k^{-1}$ and starting with a small value for $\gamma_1$ is more effective. For the example in Section~\ref{sec-queue}, our tests indicate that using $\gamma_k~\propto \xi^k$ with $\xi \approx 0.9$ is more effective.}

\subsection{Example: Service-effort control of a single-server queue }\label{sec-queue}
We consider a service rate control problem for a discrete-time single-server queue with finite capacity. We let $Q_t$ denote the number of customers in the system at time $t$.
Let  $A=\{A_t\}_{t=0}^\infty$ be the arrival process, with $A_t$ being the number of customers who enter the system at time $t$. We assume that $A_t$ is a sequence of i.i.d. random variables taking values on a finite set - we make precise the choice of this finite set in both cases later.  Every customer that enters the system will either wait in a queue or get served by a single server, if available, in the first-come first-served discipline.  At time $t,$ a customer in service is successfully served with a probability $s$, which will be appropriately chosen and hence, will be treated as the control.   Suppose that $s$ lies in a compact subset $\mathbb{U}$ of $ (0,1)$ and denote by $\Uadm$  the set of admissible service-rate control policies that are $\mathbb U$--valued.  The  Bernoulli process associated with the customers serviced is denoted by $\{S_t\}_{t=0}^\infty$.  The objective is to find  a  control policy which achieves the following:
$$ \inf_{ i} \inf_{\{S_t\}\in \Uadm} \limsup_{T\to\infty} \frac{1}{\delta T} \log \EE\big[e^{\delta \sum_{t=0}^{T-1} c(Q_t,S_t)}|Q_0=i\big],$$
where $c(i,s)\doteq  r(i) + g(s)$ with $r(i)$ being the congestion cost when $i$ customers are present in the system and $g(s)$ is the energy cost when service effort is $s$, and $\delta >0$ is the sensitivity parameter. We assume that both $r$ and $g$ are increasing  and  $g$ is convex.

Suppose that the system has a maximum capacity of $n$, so that any new arrivals when there are $n $ customers in the system are rejected. Suppose that $A_t$ is a sequence of i.i.d. Bernoulli random variables taking values $\{0,1\}$, with $\mathbb P(A_t = 1) = \alpha \in (0,1)$. The dynamics of $Q_t$ can be expressed as follows: 
$Q_{t+1}= \big[Q_t+A_{t}\Ind_{Q_t < n} - S_{t}\big]^+ \,.$
For a fixed $s\in \mathbb U$, $Q_t$ is a controlled DTMC with the transition probability $p(i,j,s)$ (depending the choice of $s\in \mathbb U$) given by the following: for $1<i <n$, 
\begin{equation*}
	\begin{aligned} p(i,j,s) &= \begin{cases}
			\alpha (1-s), \, &   j = i+1,\\
			(1-\alpha)(1-s) + \alpha \,s, \, & j = i,\\
			(1-\alpha)s, \, &j = i - 1,
		\end{cases}\\
		p(0,1,s) = \alpha,\quad p(0,0,s) &= 1- \alpha,\quad  p(n,n-1,s) = s\quad \text{and}\quad p(n,n,s) = 1 - s\,.
	\end{aligned}
\end{equation*}

Figure~\ref{fig-plot_single_server_finite} illustrates the convergence of the RVI algorithms. 
The first trend is how the Jacobi iterates seem to frequently oscillate in a damped manner. This trend also occurs for the Gauss-Seidel iterates, but within the first $30$ iterations. We next perform a  sensitivity analysis to understand the effect that the size of the state space $n$ and the sensitivity parameter $\delta$  have on the convergence of the discussed RVI algorithms {(we use $\gamma_k = 0.95^{\hat k}$, as our testing has shown that starting with a larger threshold which stays above $k^{-1}$ for $k<100$ seems to be more effective)}. 
The results are shown in Table~\ref{table:2}. 
We can infer the following for this example: (i)
among our three algorithms, the higher the risk sensitivity, the better the performance of the Gauss-Seidel iteration algorithm,  and (ii) as one would expect, the larger the state space, the slower the  convergence of  all three algorithms; this effect is more pronounced when the risk sensitivity is lower.

\begin{figure}[t]
	\centering
	\includegraphics[width=9cm, height=6cm]{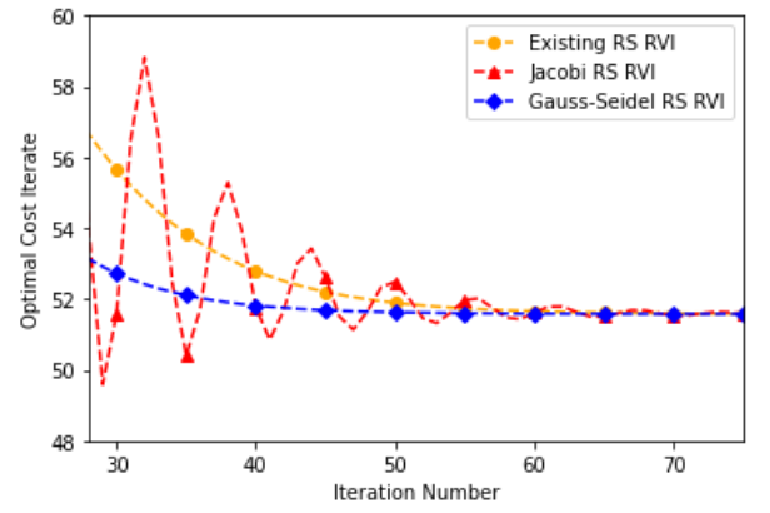}
	\caption{Plot of $\lambda^k$ \emph{vs} the iteration number $k$ for the single server queue with a maximum capacity $n=20$, $\alpha = 0.4,$ $\mathbb U = \{0.1,0.25, 0.4,0.5,0.75,0.9\}$, $c(i,s) = 5(i-1)^+ + 0.25s^2,$ $\gamma_k = 0.95^{\hat k}$ (with updating thresholds $\theta=0.75,0.85$ for Algorithms~\ref{alg-rvi-bertsekas-JI} and~\ref{alg-rvi-bertsekas-GS}, respectively), and $\delta = 1\times 10^{-2}$. The existing RVI algorithm and Algorithms~\ref{alg-rvi-bertsekas-JI} and~\ref{alg-rvi-bertsekas-GS} fall within the tolerance of $\varepsilon = 1\times 10^{-4}$, after $120, 119,$ and $105$ iterations, respectively.  }\label{fig-plot_single_server_finite}
\end{figure}

\begin{table}
	\centering
	\begin{tabular}{||c| c |c c c||} 
		\hline
		$n$ & $\delta$& Existing ERSC RVI  & \textbf{Jacobi} ERSC RVI & \textbf{Gauss-Seidel} ERSC RVI \\ [0.5ex] 
		\hline\hline
		
		$20$ & $5\times 10^{-2}$ & $48$ & $66$ & $45$ \\
		$40$ &$5\times 10^{-2}$ & $68$ & $86$ & $62$ \\
		$60$ & $5\times 10^{-2}$ & $88$ & $104$ & $82$ \\
		\hline 
		$20$ & $1\times 10^{-2}$ & $120$ & $119$ & $105$ \\
		$40$ &$1\times 10^{-2}$ & $144$ & $144$ & $130$ \\
		$60$ & $1\times 10^{-2}$ & $166$ & $164$ & $150$ \\
		\hline
		$20$ & $1\times 10^{-3}$ & $71$ & $83$ & $177$ \\
		$40$ &$1\times 10^{-3}$ & $147$ & $147$ & $405$ \\
		$60$ & $1\times 10^{-3}$ & $276$ & $277$ & $738$ \\
		\hline
	\end{tabular}
	\captionof{table}{Number of iterations required for the difference of successive cost iterates produced by each algorithm to be smaller than a tolerance of $10^{-4}$. }
	\label{table:2}
\end{table} 
\raggedbottom

\subsection{Example: Maximizing the exit-rate from a finite domain}\label{sec-exit}
Consider a controlled DTMC on a finite set that is irreducible for all stationary policies.  Our goal is to find  a control policy that maximizes the rate of exit of the DTMC from a fixed  subset, if the DTMC starts inside that subset. For the sake of  concreteness, we fix the finite set to be a weighted connected graph with $n$ nodes (we denote the set of nodes by $S=\{1,2,\ldots,n\}$) and   the control set  $\bU$ to be a finite set.  The controlled DTMC $X= \{X_t\}_{t=0}^\infty$ is now defined through the transition probability $\{p(i,j,u): i,j\in S, u\in \bU\}$  which is in turn assigned through the weight function $w:S\times S\times \bU\rightarrow [0,\infty)$ (with $w(i,j,\cdot) = 0$ only if  the nodes $i$ and $j$ are not connected)  as follows:
\begin{equation} \label{eq-ex_2_p_probs}
	p(i,j,u)=\frac{w(i,j,u)}{\sum_{k\in S} w(i,k,u)}\,.
\end{equation} 
The graph connectivity implies that the above $p(i,j,u)$  is a valid transition probability and that the DTMC $X$ is irreducible under any stationary Markov control policy.
Let us fix a connected sub-graph $S_0 $ of $S$ and for a control $U=\{U_t\}_{t=0}^\infty$, denote by $ \hat \tau(U)$, the first exit time from the set $S_0$, \emph{i.e.}, $\hat \tau(U) = \inf\{t\geq 0: X_t\notin S_0\}$. As mentioned above, we are interested in the following problem:
\begin{equation}\label{eq-minim-exit} \lambda^*=\max_{i\in S_0} \sup_{U\in \Uadm}  \liminf_{T\to\infty}- \frac{1}{T}\log \PP(\hat \tau(U)>T|X_0=i)\,.\end{equation}
Here, $\Uadm $ is the class of admissible controls that are $\bU$--valued. Note the minus sign in the above display - this is included because the exit-rate, by convention, is non-negative and the quantity $ \frac{1}{T}\log \PP(\hat \tau(U)>T)\leq 0$  as the exit time $\hat \tau(U)$ is almost surely finite.  

We are now in a position to re-cast this problem as an ERSC problem involving the DTMC $X$ and an appropriate running cost. To do this, we first define a controlled DTMC that is restricted to $S_0$. This can be done as follows: define $\overline p_0(i,u)\doteq  \sum_{j \in S_0}p(i,j,u) $ 
and for $i,j\in S_0$, $q(i,j,u)\doteq \frac{p(i,j,u)}{\overline p_0(i,u)}$.
The above expression is well-defined due to the connectedness of $S_0$, \emph{i.e.},  $\overline p_0(i,\cdot)>0$ for all $i\in S$. Hence, $q(i,j,u)$ is also  a transition probability. Let the associated  controlled DTMC be denoted by $Y=\{Y_t\}_{t=0}^\infty$. From the definition of $q(\cdot,\cdot,\cdot)$, it is clear that $Y$  is restricted to $S_0$ and is irreducible.  From here,  for any $U\in \Uadm$, we can verify that  
$$ \PP(\hat \tau(U)>T|X_0=i) = \EE_i^U \Big[e^{\sum_{t=0}^{T-1} \log \overline p_0(Y_t,U_t)}\Big]\,.$$
This, in addition to interchanging the $\limsup$ of a negative quantity with $-\liminf$, reduces~\eqref{eq-minim-exit}  to 
\begin{equation*}
	\lambda^*=-\min_{i\in S_0} \inf_{U\in \Uadm}  \limsup_{T\to\infty} \frac{1}{T}\log \EE_i^U \Big[e^{\sum_{t=0}^{T-1} c(Y_t,U_t)}\Big],
\end{equation*}
which resembles the ERSC problem defined in~\eqref{eq-ERSC-val} with the running cost $c(i,u)= \log \overline p_0(i,u)$.  

For the above problem to be completely defined, it only remains to make a particular choice of the weight function. In our numerical implementation, we consider  two particular choices of weight functions which correspond to the case of a complete graph and the case of a `sparse' graph (that is neither complete nor a tree). In both the cases,  we fix (i) the number of nodes $n=20$ that are randomly chosen inside a disc of radius $R=10$, (ii) the nodes are connected (depending on the case), and (iii) we set $\mathbb U = \{1,1.3,1.6,1.9\}$ and  $S_0 = \{1,\dots,m\}$ with $m=5$. The weight function $w:S\times S\times \bU\rightarrow [0,\infty)$ as $w(i,j,u) = \frac{d(i,j)}{u}$  if nodes  $i,j$ are connected and $w(i,j,u) =0$ otherwise.
	Here, $d(i,j)$ is the Euclidean distance between nodes $i$ and $j$. Regarding the control $u$ as the speed with which the controller travels across the edge connecting $i$ and $j$,  the weight function $w(i,j,u)$ represents 
	the  time taken to travel across the edge connecting nodes $i$ and $j$.

	\begin{figure}[h]
		\centering 
		\begin{subfigure}[t]{0.49\textwidth}
			\includegraphics[width=\textwidth, height=6cm]{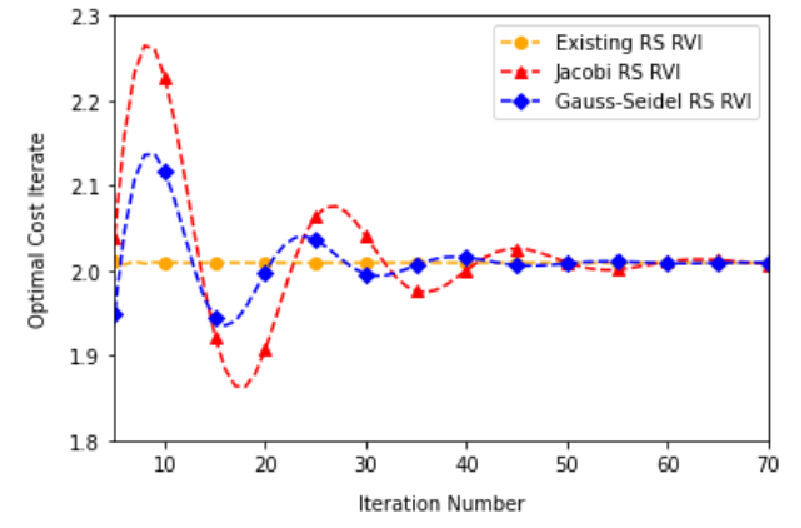}
			
			\caption{}\label{fig-ex-graph_plot_complete}
		\end{subfigure}
		\begin{subfigure}[t]{0.49\textwidth}
			\includegraphics[width=\textwidth, height=6cm]{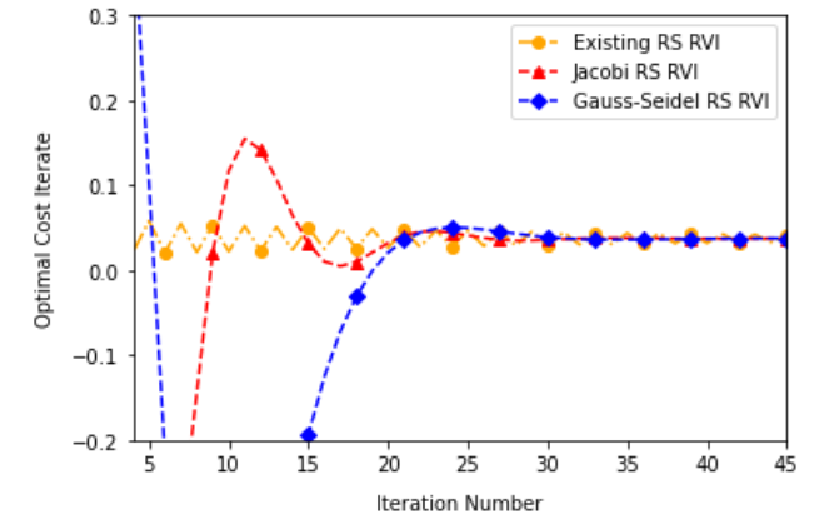}
			
			\caption{}\label{fig-ex-graph_plot_sparse}
		\end{subfigure}
		\caption{Plot of $-\lambda^k$ \emph{vs} the iteration number $k$ for graphs with two different connectivities and parameters $n=20$, $R=10$, $\mathbb U = \{1,1.3,1.6,1.9\}$, and $m=5$. On the left, the graph is complete and the iterates of the existing RVI algorithm and Algorithms~\ref{alg-rvi-bertsekas-JI} and~\ref{alg-rvi-bertsekas-GS} fall within the  tolerance of $\varepsilon = 1\times 10^{-4}$, after $18$, $104$ and  $88$ iterations, respectively. On the right, the graph is neither complete nor a tree and the iterates of the existing algorithm and Algorithms~\ref{alg-rvi-bertsekas-JI} and~\ref{alg-rvi-bertsekas-GS} fall within the  tolerance of $\varepsilon = 1\times 10^{-4}$, after $319$, $66$ and  $68$ iterations, respectively.}
	\end{figure}

	In the implementation of Algorithms~\ref{alg-rvi-bertsekas-JI} and~\ref{alg-rvi-bertsekas-GS}, we set step-sizes $\gamma_k \propto \hat k^{-1}$, where $\hat k$ counts the number of iterations that the iterates $h^k(n)$ sequentially change signs and are larger than a fixed threshold of $\theta = 0.5$ { (which we selected based on a simple grid search procedure)}. 
	Figure~\ref{fig-ex-graph_plot_complete} corresponds to the case where the underlying graph is complete with $\gamma_k= 0.1\hat k^{-1}$ and $\gamma_k= 0.08\hat k^{-1}$ for Algorithms~\ref{alg-rvi-bertsekas-JI} and~\ref{alg-rvi-bertsekas-GS}, respectively. Figure~\ref{fig-ex-graph_plot_sparse} corresponds to the case where the graph is neither complete nor a tree with $\gamma_k=0.25\hat k^{-1}$ and $\gamma_k=0.15\hat k^{-1}$ for Algorithms~\ref{alg-rvi-bertsekas-JI} and~\ref{alg-rvi-bertsekas-GS}, respectively. We observe that for a connected graph, the less the number of edges (in the graph), the better the performance of our Jacobi-like  and Gauss-Seidel-like algorithms, in comparison to the performance of  the existing algorithm.  
	
	{
		We also briefly comment on the algorithmic performance after performing the aperiodicity transformation from~\citep{Cavazos_2003, Murthy_ERSC_MPI_2024} on the cost and transition matrix. This can be done for any ERSC problem with an irreducible transition matrix in order to ensure the controlled process is also aperiodic (without changing the optimal policy), which is usually required for convergence when implementing the existing RVI and modified policy iteration algorithms. We note that for our proposed RVI algorithms, this aperiodicity transformation is not required for convergence, and refer the reader to Appendix~\ref{app-add-num} to see an application of this transformation in the context of this exit rate problem.
	}

	\section{Concluding Remarks}\label{sec-conclude}

	We introduced a Jacobi-like RVI algorithm and a Gauss-Seidel-like implementation  for the ergodic risk-sensitive control problem of a finite-state controlled Markov chain. These algorithms consist of a coupled iteration for the optimal cost and value function, where the  cost is updated through a one-dimensional line-search involving a user-specified step-size and the value function is updated according to a risk-sensitive Bellman-like operator.  Under the assumption that the controlled process is irreducible and recurrent under every stationary policy, we proved that the proposed RVI algorithms converge geometrically. The proof required establishing local contraction properties for both risk-sensitive Bellman-like operators and a local bi-Lipschitz continuity property for the fixed points of these operators. We also evaluated the performance of the proposed algorithms on two examples. 
		
		Several directions remain open for future work. One immediate task is to consider RVI algorithms for ergodic risk-sensitive control problems with countably infinite state spaces, in particular, investigating the convergence (as well as the convergence rate) for RVI algorithms through finite state truncations.  It would also be interesting to explore RVI algorithms for partially observable risk-sensitive control problems, where the original state process is hidden. Another interesting direction is to consider  the learning counterparts for the proposed algorithms for settings where the underlying transition kernel is unknown. The Jacobi-like RVI algorithm and its Gauss-Seidel-like implementation suggest several possible Q-learning and other related reinforcement learning algorithms.

				 \section*{Acknowledgement}
This work is funded by  the NSF Grant DMS 2216765. 

\bibliographystyle{abbrv}
					\bibliography{ERSC_RVI}

\begin{thebibliography}{10}

\bibitem{Arapostathis_2019}
A.~Arapostathis and V.~S. Borkar.
\newblock On the relative value iteration with a risk-sensitive criterion.
\newblock {\em Banach Center Publications}, 122:9--24, 2020.

\bibitem{Arapostathis2013}
A.~Arapostathis, V.~S. Borkar, and K.~S. Kumar.
\newblock Relative value iteration for stochastic differential games.
\newblock In V.~K{\v{r}}ivan and G.~Zaccour, editors, {\em Advances in Dynamic
  Games: Theory, Applications, and Numerical Methods}, Annals of the
  International Society of Dynamic Games 13, pages 3--27. Springer
  International Publishing, 2013.

\bibitem{Barz_RS_revenue_2007}
C.~Barz and K.~Waldmann.
\newblock Risk-sensitive capacity control in revenue management.
\newblock {\em Mathematical Methods of Operations Research}, 65:565--579, 2007.

\bibitem{Basu_RS_Q_learning_2008}
A.~Basu, T.~Bhattacharyya, and V.~S. Borkar.
\newblock A learning algorithm for risk-sensitive cost.
\newblock {\em Mathematics of Operations Research}, 33(4):880--898, 2008.

\bibitem{Bertsekas_VI_98}
D.~P. Bertsekas.
\newblock A new value iteration method for the average cost dynamic programming
  problem.
\newblock {\em SIAM Journal on Control and Optimization}, 36(2):742--759, 1998.

\bibitem{Bertsekas_DP}
D.~P. Bertsekas.
\newblock {\em Dynamic programming and optimal control}.
\newblock Athena Scientific, 2005.

\bibitem{Bielecki_1999}
T.~Bielecki, D.~Hernandez-Hernandez, and S.~Pliska.
\newblock Value iteration for controlled {M}arkov chains with risk sensitive
  cost criterion.
\newblock {\em Proceedings of the 38th IEEE Conference on Decision and
  Control}, 1:126--130, 1999.

\bibitem{Bielecki_RS_portfolio_RVI_1999}
T.~Bielecki, D.~Hernández-Hernández, and S.~Pliska.
\newblock Risk sensitive control of finite state {M}arkov chains in discrete
  time, with applications to portfolio management.
\newblock {\em Mathematical Methods of Operations Research}, 50:167--188, 1999.

\bibitem{Bielecki_RS_Dynamic_manage_1999}
T.~Bielecki and S.~Pliska.
\newblock Risk-sensitive dynamic asset management.
\newblock {\em Applied Mathematics and Optimization}, 39:337--360, 1999.

\bibitem{Bielecki_Cox_2005}
T.~Bielecki, S.~Pliska, and S.-J. Sheu.
\newblock Risk sensitive portfolio management with {C}ox--{I}ngersoll--{R}oss
  interest rates: the {H}{J}{B} equation.
\newblock {\em SIAM Journal on Control and Optimization}, 44(5):1811--1843,
  2005.

\bibitem{BISWAS2023118}
A.~Biswas and V.~S. Borkar.
\newblock Ergodic risk-sensitive control—a survey.
\newblock {\em Annual Reviews in Control}, 55:118--141, 2023.

\bibitem{Biswas2021ERSC}
A.~Biswas and S.~Pradhan.
\newblock Ergodic risk-sensitive control of {M}arkov processes on countable
  state space revisited.
\newblock {\em ESAIM: Control, Optimisation and Calculus of Variations}, 28,
  2021.

\bibitem{Borkar_sensitivity_RS_2001}
V.~Borkar.
\newblock A sensitivity formula for risk-sensitive cost and the actor–critic
  algorithm.
\newblock {\em Systems \& Control Letters}, 44(5):339--346, 2001.

\bibitem{Borkar_2002_Q_learning}
V.~S. Borkar.
\newblock Q-learning for risk-sensitive control.
\newblock {\em Mathematics of Operations Research}, 27(2):294--311, 2002.

\bibitem{Borkar_2002}
V.~S. Borkar and S.~P. Meyn.
\newblock Risk-sensitive optimal control for {M}arkov decision processes with
  monotone cost.
\newblock {\em Mathematics of Operations Research}, 27(1):192--209, 2002.

\bibitem{Bouakiz_inventory_exp_1992}
M.~Bouakiz and M.~J. Sobel.
\newblock Inventory control with an exponential utility criterion.
\newblock {\em Operations Research}, 40(3):603--608, 1992.

\bibitem{Baurle2023_survey}
N.~Bäuerle and A.~Jaśkiewicz.
\newblock Markov decision processes with risk-sensitive criteria: An overview.
\newblock {\em Mathematical Methods of Operations Research}, 99:141--178, 2024.

\bibitem{Cavazos-Cadena2002}
R.~Cavazos-Cadena and E.~Fern{\'a}ndez-Gaucherand.
\newblock Risk-sensitive optimal control in communicating average {M}arkov
  decision chains.
\newblock In M.~Dror, P.~L'Ecuyer, and F.~Szidarovszky, editors, {\em Modeling
  Uncertainty: An Examination of Stochastic Theory, Methods, and Applications},
  International Series in Operations Research \& Management Science 46, pages
  515--553. Springer US, 2002.

\bibitem{Cavazos_2003}
R.~Cavazos-Cadena and R.~Montes-de Oca.
\newblock The value iteration algorithm in risk-sensitive average {M}arkov
  decision chains with finite state space.
\newblock {\em Mathematics of Operations Research}, 28(4):752--776, 2003.

\bibitem{Chen_RS_Discrete_2023}
X.~Chen and Q.~Wei.
\newblock Risk-sensitive average optimality for discrete-time {M}arkov decision
  processes.
\newblock {\em SIAM Journal on Control and Optimization}, 61(1):72--104, 2023.

\bibitem{Chow2015risk}
Y.~Chow, A.~Tamar, S.~Mannor, and M.~Pavone.
\newblock Risk-sensitive and robust decision-making: a cvar optimization
  approach.
\newblock In {\em Advances in Neural Information Processing Systems 28}, pages
  1522--1530, 2015.

\bibitem{Coraluppi_1999}
S.~P. Coraluppi and S.~I. Marcus.
\newblock Risk-sensitive and minimax control of discrete-time, finite-state
  {M}arkov decision processes.
\newblock {\em Automatica}, 35(2):301--309, 1999.

\bibitem{Davis_Math_Finance_2019}
M.~H. Davis.
\newblock {\em Mathematical finance: a very short introduction}.
\newblock Oxford University Press, 2019.

\bibitem{dupuis1997weak}
P.~Dupuis and R.~S. Ellis.
\newblock {\em A weak convergence approach to the theory of large deviations}.
\newblock John Wiley \& Sons, 1997.

\bibitem{Dupuis_robust_RS_2000}
P.~Dupuis, M.~James, and I.~Petersen.
\newblock Robust properties of risk-sensitive control.
\newblock {\em Mathematics of Control, Signals, and Systems}, 13:318--332,
  2000.

\bibitem{Feng_RS_perishable_2008}
Y.~Feng and B.~Xiao.
\newblock A risk-sensitive model for managing perishable products.
\newblock {\em Operations Research}, 56(5):1305--1311, 2008.

\bibitem{Fleming_RS_Finite_State_1997}
W.~H. Fleming and D.~Hern\'{a}ndez-Hern\'{a}ndez.
\newblock Risk-sensitive control of finite state machines on an infinite
  horizon i.
\newblock {\em SIAM Journal on Control and Optimization}, 35(5):1790--1810,
  1997.

\bibitem{Fleming_opt_growth_RS_1999}
W.~H. Fleming and S.~Sheu.
\newblock Optimal long term growth rate of expected utility of wealth.
\newblock {\em The Annals of Applied Probability}, 9(3):871–903, 1999.

\bibitem{Guin_RS_Actor_2026}
S.~Guin, V.~S. Borkar, and S.~Bhatnagar.
\newblock An actor–critic algorithm with function approximation for risk
  sensitive cost {M}arkov decision processes.
\newblock {\em IEEE Transactions on Automatic Control}, 71(1):474–481, 2026.

\bibitem{Jaquette1976utility}
S.~C. Jaquette.
\newblock A utility criterion for markov decision processes.
\newblock {\em Management Science}, 23(1):43--49, 1976.

\bibitem{Levin_mixing_2008}
D.~A. Levin, Y.~Peres, and E.~L. Wilmer.
\newblock {\em {M}arkov Chains and Mixing Times}.
\newblock American Mathematical Society, 2008.

\bibitem{Li2022quantile_mdp}
X.~Li, H.~Zhong, and M.~L. Brandeau.
\newblock Quantile {M}arkov decision processes.
\newblock {\em Operations Research}, 70(3):1428--1447, 2022.

\bibitem{Moharrami_policy_grad_exp_2024}
M.~Moharrami, Y.~Murthy, A.~Roy, and R.~Srikant.
\newblock A policy gradient algorithm for the risk-sensitive exponential cost
  mdp.
\newblock {\em Mathematics of Operations Research}, 50(1):431--458, 2024.

\bibitem{Murthy_ERSC_MPI_2024}
Y.~Murthy, M.~Moharrami, and R.~Srikant.
\newblock On the convergence of modified policy iteration in risk sensitive
  exponential cost {M}arkov decision processes.
\newblock {\em Forthcoming in Operations Research}, 2025.

\bibitem{Nagai_RS_Portfolio_2002}
H.~Nagai and S.~Peng.
\newblock Risk-sensitive dynamic portfolio optimization with partial
  information on infinite time horizon.
\newblock {\em The Annals of Applied Probability}, 12(1):173--195, 2002.

\bibitem{Noorani_RS_RL_2025}
E.~Noorani, C.~Mavridis, and J.~Baras.
\newblock Risk-sensitive reinforcement learning with exponential criteria.
\newblock {\em IEEE Transactions on Cybernetics}, 55(8):3774--3787, 2025.

\bibitem{Peterson_min_max_entropy_2000}
I.~R. Petersen, M.~R. James, and P.~Dupuis.
\newblock Minimax optimal control of stochastic uncertain systems with relative
  entropy constraints.
\newblock {\em IEEE Transactions on Automatic Control}, 45(3):398--412, 2002.

\bibitem{Sobel1994meanvar}
M.~J. Sobel.
\newblock Mean-variance tradeoffs in an undiscounted {MDP}.
\newblock {\em Operations Research}, 42(1):175--183, 1994.

\bibitem{tseng1990solving}
P.~Tseng.
\newblock Solving ${H}$-horizon, stationary {M}arkov decision problems in time
  proportional to $\log ({H})$.
\newblock {\em Operations Research Letters}, 9(5):287--297, 1990.

\bibitem{White1963}
D.~White.
\newblock Dynamic programming, {M}arkov chains, and the method of successive
  approximations.
\newblock {\em Journal of Mathematical Analysis and Applications},
  6(3):373--376, 1963.

\bibitem{WuXuRSMDPUtility2023}
Z.~Wu and R.~Xu.
\newblock Risk-sensitive {M}arkov decision process and learning under general
  utility functions.
\newblock arXiv:2311.13589, 2023.

\end{thebibliography}
					\newpage
\appendix

\section{Summary of key notation and  constants}

\noindent
\setlength{\tabcolsep}{6pt}
\renewcommand{\arraystretch}{1.15}

\begin{tabular}{@{}p{0.22\linewidth}p{0.76\linewidth}@{}}

	$\Lambda$ & Set of parameters $\lambda$ for which the risk-sensitive Bellman-like operator $\widetilde F(\cdot,\lambda)$ has a fixed point. \\[3pt]
	
	$\Lambda_m$ & Subset of $\Lambda$ for which the supremum norm of the fixed point $h_\lambda$ is bounded by $m$. \\[3pt]
	
	$\eta$ & Smallest strictly positive transition probability, after taking the infimum over controls; used in the construction of the weighted supremum norm. \\[3pt]
	
	$\| \cdot \|_m$ & Weighted supremum norm for which the local contraction property of $F(\cdot, \lambda),G(\cdot,\lambda)$ holds and also for which the geometric convergence of $h_k$ in Theorem~\ref{thm-main-1} holds. \\[3pt]
	
	$w_i^m$ & Weights used to define the weighted supremum norm $|\cdot|_m$. \\[3pt]

	$m$ &
	Constant used in the definition of the weighted supremum norm $\|\cdot\|_m$.\\[3pt]
	
	$\overline{\gamma}$ &
	Upper bound on the step sizes $\gamma_k$ ensuring geometric convergence of Algorithms~\ref{alg-rvi-bertsekas-JI} and~\ref{alg-rvi-bertsekas-GS}.\\[3pt]
	
	$c_h(\gamma,m)$ &
	One-step contraction constant for the $h$-iterate; for $\gamma\in(0,\overline{\gamma}]$ one has $0<c_h(\gamma,m)<1$.\\[3pt]
	
	$c_\lambda(\gamma,m)$ &
	One-step contraction constant for the $\lambda$-iterate; for $\gamma\in(0,\overline{\gamma}]$ one has $0<c_\lambda(\gamma,m)<1$.\\[3pt]
	
	$L(m)$ &
	Lipschitz constant used for the fixed point map $\lambda\mapsto h_\lambda$ in the proof of Theorem~\ref{thm-main-1}.\\[3pt]
	
	$\rho(m)$ &
	Geometric convergence parameter ($\rho(m)>1$) appearing in the asymptotic rate statement for the algorithm iterates.\\[3pt]
	
	$\beta_m$ &
	Local contraction constant for the operator $F(h,\lambda)$ with respect to the weighted supremum norm $\|\cdot\|_m$ on the set $\{h:\|h\|_\infty\le m\}$.\\[3pt]
	
	$\widetilde N^*(m',\lambda'),\ \widetilde N_*(m,\lambda)$ &
	Local Lipschitz constants for the map $\lambda\mapsto h_\lambda$ (with $\lambda>\lambda'$, $m=\|h_\lambda\|_\infty$, and $m'=\|h_{\lambda'}\|_\infty$).\\[3pt]
	
	$\underline c,\overline c$ & Minimum and maximum running costs over all states and controls.\\[3pt]
	
	$\tau_n$ & First return time to the reference state $n.$ \\[3pt]
	
	$\underline \tau,\overline \tau$ & Uniform lower and upper bounds on expected return times to state $n$, taken over initial states and stationary Markov policies. \\[3pt]
	
	$\overline L,\underline \alpha$ & Uniform path-length and path-probability constants used in the hitting time estimates for the transformed Markov chains appearing in the proof for local bi-Lipschitz continuity proof of $\lambda \mapsto h_\lambda $ in Proposition~\ref{lem-bounds-h}. \\[3pt]
	
	$\bar\eta,\chi(\cdot)$ & Constants used in the analysis of the Gauss-Seidel-like operator $G$. \\[3pt]
	$\Delta_i^m,\Delta^m$ & Constants used in the contraction for the Gauss-Seidel-like operator $G$. 
\end{tabular}

\section{Entropy variational formula}\label{app-ent-var}
We recall the well-known entropy variational formula as given in Proposition 1.4.2 of \cite{dupuis1997weak}, we have used at several places in our analysis. 
\begin{proposition}\label{prop-ent-var}
	Let $(\mathcal{X},\mathcal{G})$ be a Borel measurable space and $f:\mathcal{X} \rightarrow \RR$ be  a bounded and measurable functional. Then for a probability measure $P$ on  $\mathcal{X}$, 
	\begin{equation}\label{eq-ent-var}
		\log \int_\mathcal{X} e^{f(x)}P(dx) = \sup_{Q\in \mathcal{P}(\mathcal{X})} \bigg[  \int_\mathcal{X}f(x) Q(dx) -R(Q \,||\,P)\bigg],
	\end{equation}
	where $R(Q||P)$ denotes the relative entropy of $Q$ with respect to $P$, \emph{i.e.}, 	
	\begin{equation*}
		R(Q || P) \doteq\begin{cases}
			\int_\mathcal{X} \log \frac{dQ}{dP}(x) Q(dx),\, &\text{ if } Q \ll P, \\
			\infty, \, &\text{ otherwise.}
		\end{cases}
	\end{equation*}
	Furthermore, the supremum in \eqref{eq-ent-var} is uniquely attained at the probability measure $Q^*$ defined by
	\begin{equation}\label{eq-ent-var-max}
		Q^*(dx) \doteq \frac{e^{f(x)}}{\int_\mathcal{X}e^{f(y)}P(dy)}P(dx)\,.
	\end{equation}
\end{proposition}

\section{Additional Numerical Implementations}\label{app-add-num}

We provide a numerical implementation of the exit rate problem from Section~\ref{sec-exit} after performing the previously mentioned aperiodicity transformation of \citep{Cavazos_2003, Murthy_ERSC_MPI_2024}. Upon doing so, we find that the performance of the existing RVI algorithm increases slightly while the performances for Algorithms~\ref{alg-rvi-bertsekas-JI} and~\ref{alg-rvi-bertsekas-GS} decrease marginally (for both the complete and sparse graphs). However, if the aperiodicity coefficient increases, we have observed that the performance of Algorithms~\ref{alg-rvi-bertsekas-JI} and~\ref{alg-rvi-bertsekas-GS} worsens.

\begin{figure}[h]
	\centering 
	\begin{subfigure}[t]{0.49\textwidth}
		\includegraphics[width=\textwidth, height=6cm]{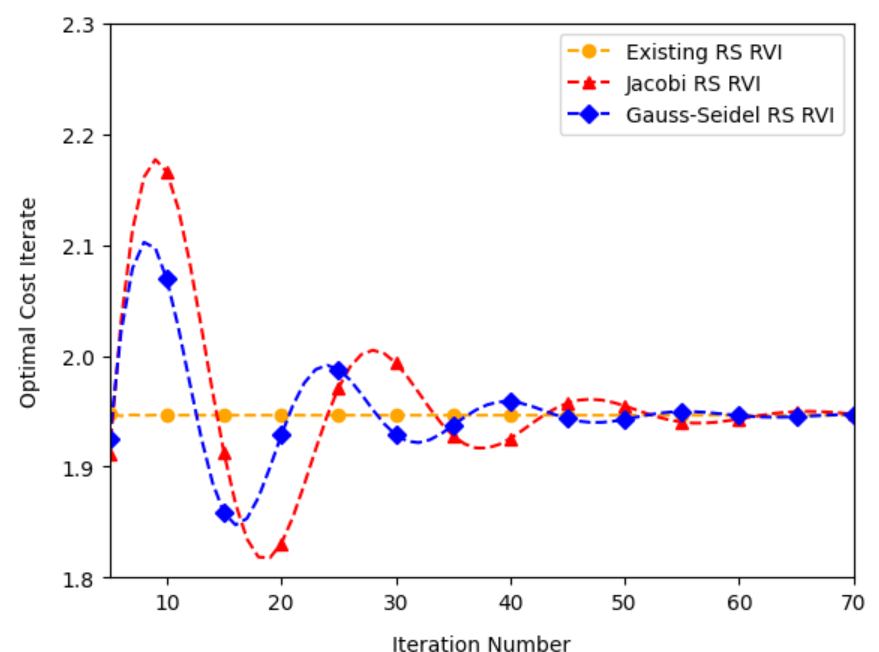}
		
		\caption{}\label{fig-ex-graph_plot_complete-2}
	\end{subfigure}
	\begin{subfigure}[t]{0.49\textwidth}
		\includegraphics[width=\textwidth, height=6cm]{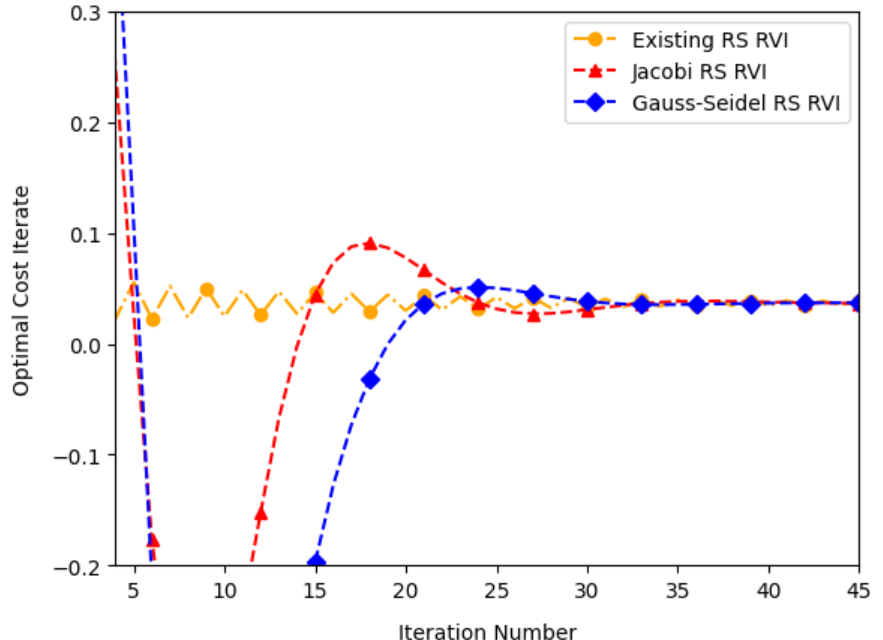}
		
		\caption{}\label{fig-ex-graph_plot_sparse-2}
	\end{subfigure}
	\caption{Plot of $-\lambda^k$ \emph{vs} the iteration number $k$ for graphs with two different connectivities and parameters $n=20$, $R=10$, $\mathbb U = \{1,1.3,1.6,1.9\}$, and $m=5$, where both problems have undergone the aperiodicity transformation with aperiodicity coefficient $\kappa = 0.05$. On the left, the graph is complete and the iterates of the existing RVI algorithm and Algorithms~\ref{alg-rvi-bertsekas-JI} and~\ref{alg-rvi-bertsekas-GS} fall within the  tolerance of $\varepsilon = 1\times 10^{-4}$, after $12$, $71$ and  $99$ iterations, respectively. On the right, the graph is neither complete nor a tree and the iterates of the existing algorithm and Algorithms~\ref{alg-rvi-bertsekas-JI} and~\ref{alg-rvi-bertsekas-GS} fall within the  tolerance of $\varepsilon = 1\times 10^{-4}$, after $198$, $78$ and  $68$ iterations, respectively.}
\end{figure}

\section{Proof of Proposition~\ref{prop-c-alpha}}\label{app-c-alpha2}
Recall that for  $k\geq 1$, the iterates $(h^{k},\lambda^k)$ are given recursively by~\eqref{eq-F-itr}. The proof in this case is divided into several lemmas whose proofs hinge heavily on the use of Proposition~\ref{thm-contraction} and Proposition~\ref{lem-bounds-h}. 
We begin by proving that $|\lambda^{k+1}-\lambda^*|\leq \widetilde c_\lambda N_*^{-1}(m) \alpha$. The proof of this estimate is split into two cases depending on whether $\lambda^k\leq \lambda^*$ or $\lambda^k>\lambda^*$. We only provide  the proof of the estimate of $|\lambda^{k+1}-\lambda^*|$  in detail  in the case where  $\lambda^k\leq \lambda^*$. For the case where $\lambda^k>\lambda^*$, we highlight the key differences in the arguments involved at the end of the proof of Lemma~\ref{lemma-geom_conv_lambda-1}.

{ \noindent\bf Case: $\lambda^k\leq \lambda^*$.} The idea is to use Proposition~\ref{coro-connectivity_Lambda} (i)-(ii) to choose intermediate $\lambda$ values such that $|\lambda-\lambda^*| \propto \alpha$ which is similar to~\eqref{eq-c-alpha2} and then, on a case-by-case basis, show that we can choose a step-size $\gamma_k$ so that $\lambda^{k+1}$ is close enough to the intermediate $\lambda$ values. The approach is summarized as follows:
\begin{enumerate}
	\item[(i)] Using the explicit bounds for the difference between any two scalars in $\Lambda$  from Proposition~\ref{lem-bounds-h} as well as the bijectivity of the function $\lambda \mapsto h_\lambda(n)$ from Proposition~\ref{coro-connectivity_Lambda} (i)-(ii), we  choose $\overline{\lambda}, \underline \lambda\in \Lambda_m$ so that $h_{\overline \lambda}(n) = \alpha\beta_m,\, h_{\underline \lambda}(n) = \alpha.$ We then define $\widehat \lambda$ as the midpoint of $\overline \lambda$ and $\underline \lambda$.
	\item[(ii)] We next analyze $|\lambda^* - \lambda^{k+1}|$ on a case-by-case basis depending on whether our newly defined $\widehat \lambda$ is larger or smaller than $\lambda^k$, where both cases are handled through comparisons with the intermediate values $\widehat \lambda, \underline \lambda, $ and $\overline \lambda.$ In each case, we show there exists a function $c(\gamma, m)$ that is strictly less than 1 whenever $\gamma$ is less than  a certain threshold value. For the step size $\gamma_k$ less than this threshold, we then have  $|\lambda^* - \lambda^{k+1}| \leq c(\gamma_k, m)N_*^{-1}(m) {\alpha}{}$. 	\end{enumerate}

We now give a lemma which provides the existence of the aforementioned $\bar \lambda$ and $\underline \lambda$. 
\begin{lemma}\label{lem-lambda-intermediate} For any $\alpha>0$, there exist $m>0$ large enough  and unique  $\overline \lambda,\underline \lambda \in \Lambda_m$ such that $\lambda^*>\overline \lambda>\underline \lambda$ and 
	\begin{equation*}
		h_{\overline \lambda}(n) = \alpha\beta_m \quad \text {and} \quad h_{\underline \lambda}(n) = \alpha\,.\end{equation*}	
\end{lemma}
This is straightforward consequence of Proposition~\ref{coro-connectivity_Lambda}(ii) and hence, we omit the proof.
In what follows, we refer to  $\overline \lambda$ and $\underline \lambda$ as those defined  in the above lemma. Also, let	$\widehat \lambda \doteq \frac{1}{2}({\overline \lambda + \underline \lambda}).$ We immediately have $\lambda^*>\overline \lambda>\widehat \lambda>\underline \lambda$, from Proposition~\ref{coro-connectivity_Lambda}(i). For a more quantitative comparison between the intermediate $\lambda$ values, we provide the following lemma. 	{}
\begin{lemma}\label{lemma-geom_conv_lambda_identities} The following inequalities hold among $\lambda^*$, $\overline \lambda$, $\underline \lambda$ and $\widehat \lambda$. 
	\begin{align}\nonumber
		\frac{(1 - \beta_m)\alpha}{N^*(m)} &\leq \overline \lambda - \underline \lambda \leq \frac{(1-\beta_m)\alpha}{N_*(m)},\\\nonumber
		\frac{\alpha \beta_m }{N^*(m)} &\leq \lambda^* - \overline \lambda \leq \frac{ \alpha \beta_m}{N_*(m)},\\ 
		\label{eq-tilde_lam-overline_lam}
		\frac{\alpha}{N^*(m)}& \leq \lambda^* - \underline \lambda \leq \frac{\alpha}{N_*(m)},\\\nonumber
		\frac{(1+\beta_m)\alpha}{2N^*(m)} &\leq \lambda^* - \widehat \lambda \leq \frac{(1+\beta_m)\alpha}{2N_*(m)},\\\nonumber
		\frac{(1-\beta_m)\alpha}{2N^*(m)} &\leq \overline \lambda - \widehat \lambda \leq \frac{(1-\beta_m)\alpha}{2N_*(m)}\,.
	\end{align}
	\end{lemma}	
\begin{proof}
	Using Proposition~\ref{lem-bounds-h}, we have 
\begin{equation*}
	\begin{aligned}
		&0 < N_*(m) (\overline \lambda - \underline \lambda) \leq h_{\underline \lambda}(n) - h_{\overline \lambda}(n) \leq N^*(m)(\overline \lambda - \underline \lambda) \\
		&\implies  N_*(m) (\overline \lambda - \underline \lambda) \leq (1-\beta_m)\alpha \leq N^*(m)(\overline \lambda - \underline \lambda)\,.
	\end{aligned}
\end{equation*}
by our choice of $\underline \lambda$ and $\overline \lambda. $ This gives us the first inequality in~\eqref{eq-tilde_lam-overline_lam}.

Similarly, using the fact that $h_{\lambda^*}(n) = 0$, we obtain the second and third inequalities in~\eqref{eq-tilde_lam-overline_lam}.  Lastly, the fourth inequality in~\eqref{eq-tilde_lam-overline_lam} is obtained from the definition of $\widehat \lambda$, and the fifth inequality is a result of  subtraction of the second inequality in~\eqref{eq-tilde_lam-overline_lam} from the fourth.
\end{proof}	

\begin{lemma}\label{lemma-geom_conv_lambda}
	Suppose~\eqref{eq-alpha2} holds for some $\alpha>0$, $\lambda^k\leq \widehat \lambda$ and $\gamma_k\leq \frac{1}{N^*(m)}$. Then,
\begin{equation*}
	|\lambda^{k+1} - \lambda^*| \leq c_1(\gamma_k, m) \frac{\alpha}{N_*(m)}, 
\end{equation*}
where 
$$c_1(\gamma, m) \doteq  \max \left\{ 1 - \gamma \frac{(1-\beta_m)N_*(m)^2}{2N^*(m)},\, \beta_m \gamma N_*(m) \right\}\,.$$
\end{lemma}
\begin{proof}
	Since the map  $\lambda\mapsto h_\lambda(n)$ is monotonically decreasing, Proposition~\ref{lem-bounds-h} and our choice of $\overline \lambda$ give us 
\begin{equation}\label{eq-h_lam_lower_bound}
	h_{\lambda^k}(n) \geq h_{\widehat \lambda}(n) \geq h_{\over{\lambda}}(n) + N_*(m)(\overline \lambda - \widehat \lambda) \geq \alpha \beta_m + \frac{\alpha(1-\beta_m) N_*(m)}{2N^*(m)}\,.
\end{equation}
From the definition of $h^{k+1}$, the norm $\|\cdot\|_m$ defined in~\eqref{def-norm}, Proposition~\ref{thm-contraction} and the first inequality in~\eqref{eq-alpha2}, we obtain 
\begin{equation}\label{eq-h_k+1_h_lambda_k_contract_prop}
	|h^{k+1}(n) - h_{\lambda^k}(n)| \leq \norm{F(h^k,\lambda^k) - F(h_{\lambda^k},\lambda^k)}_m \leq \beta_m \norm{h^k - h_{\lambda^k}}_m \leq \alpha \beta_m\,. 
\end{equation}
Hence, $h^{k+1}(n) \geq h_{\lambda^k}(n) - \alpha \beta_m.$ Combining this with \eqref{eq-h_lam_lower_bound} shows that $h^{k+1}(n) \geq  \frac{\alpha(1-\beta_m) N_*(m)}{2N^*(m)},$	which enables us to bound $\lambda^* - \lambda^{k+1}$ from above as follows:
\begin{equation}\label{eq-iteration_case_i_upper_bound}
	\begin{aligned}
		\lambda^* - \lambda^{k+1} &= \lambda^* - \lambda^k - \gamma_k h^{k+1}(n) \\
		&\leq \lambda^* - \lambda^k - \gamma_k\frac{\alpha(1-\beta_m) N_*(m)}{2N^*(m)} \\
		&\leq  \frac{\alpha}{N_*(m)} - \gamma_k \frac{\alpha (1-\beta_m)N_*(m)}{2N^*(m)}\\ 
		&= \frac{\alpha}{N_*(m)}\bigg(1 - \gamma_k \frac{(1-\beta_m)N_*(m)^2}{2N^*(m)}\bigg)\,.
	\end{aligned}
\end{equation}
Next, we turn to bound $\lambda^* - \lambda^{k+1}$ from below. To begin with, from~\eqref{eq-h_k+1_h_lambda_k_contract_prop}, we obtain $$h^{k+1}(n) \leq h_{\lambda^k}(n) + \alpha \beta_m\,.$$ From the above display, Proposition~\ref{lem-bounds-h} and the fact that $\lambda^k\leq \lambda^*$, we immediately get 
\begin{equation}\label{eq-hkn-ub}
	h^{k+1}(n) \leq N^*(m)(\lambda^* - \lambda^k) + \alpha \beta_m\ .
\end{equation}
Now  observe that
\begin{equation}\label{eq-iteration_case_i_lower_bound}
	\begin{aligned}
		\lambda^* - \lambda^{k+1} &= \lambda^* - \lambda^k - \gamma_k h^{k+1}(n) \\
		&\geq \lambda^* - \lambda^k - \gamma_k \left(N^*(m)(\lambda^* - \lambda^k) + \alpha \beta_m \right)\\
		&= (1-\gamma_k N^*(m))(\lambda^* - \lambda^k) - \alpha \beta_m\gamma_k\\ 
		&\geq -\alpha \beta_m \gamma_k\,.
	\end{aligned}
\end{equation}
In the above, to get the first line we use~\eqref{eq-F-itr}; to get the second line, we use~\eqref{eq-hkn-ub}; to get the fourth  line, we use the fact that $\gamma_k \leq \frac{1}{N^*(m)}$. Therefore, the fourth lines in~\eqref{eq-iteration_case_i_upper_bound} and~\eqref{eq-iteration_case_i_lower_bound} imply that 
$$ |\lambda^* - \lambda^{k+1}|\leq c_1(\gamma_k, m)\frac{\alpha}{N_*(m)}.$$
This completes the proof. \end{proof}
	\begin{lemma}\label{lemma-geom_conv_lambda-1}
	Suppose~\eqref{eq-alpha2} holds for some $\alpha>0$, $ \widehat \lambda<\lambda^k\leq \lambda^*$  and 
\begin{equation}\label{eq-gam-bound} \gamma_k\leq \min\Big\{ \frac{1+\beta_m}{\beta_mN_*(m)+N^*(m)},\,\,  \frac{1-\beta_m}{2\beta_mN^*(m)}\Big\}\,. \end{equation}
Then 
\begin{equation}\label{eq-l-bound-1}
	|\lambda^{k+1} - \lambda^*| \leq c_2( m) \frac{\alpha}{N_*(m)},\quad\text{where}\quad c_2(m) \doteq  \frac{(1+\beta_m)}{2}\,.
\end{equation}

\end{lemma}
	\begin{proof} The proof of this lemma is divided into two cases depending on the sign of $h^{k+1}(n)$, \emph{i.e.},  (a) $h^{k+1}(n) \geq 0$, or (b) $h^{k+1}(n) < 0.$
		
		{\bf \noindent Case (a):} From the definition of $\lambda^{k+1}$, we have $\lambda^{k+1} \geq \lambda^k$. Using the bounds on the difference between $\widehat \lambda $ and $\lambda^*$ in the fourth inequality of~\eqref{eq-tilde_lam-overline_lam} as well as the hypothesis that $ \lambda^k >\widehat \lambda $, we get 
		\begin{equation}\label{eq-lambda_opt_leq_lambda_iter}
			\lambda^* \leq \widehat \lambda + \frac{(1+\beta_m)\alpha}{2N_*(m)} \leq \lambda^k+ \frac{(1+\beta_m)\alpha}{2N_*(m)} \leq \lambda^{k+1} + \frac{(1+\beta_m)\alpha}{2N_*(m)}\,.
		\end{equation}
		We next bound $h^{k+1}(n)$ from above:
		\begin{equation}\label{eq-bounds_lam_iterates}
			\begin{aligned}
				0\leq h^{k+1}(n)& \leq |h^{k+1}(n) - h_{\lambda^k}(n)| + |h_{\lambda^k}(n)|\\
				&\leq\norm{h^{k+1} - h_{\lambda^k}}_m + N^*(m)|\lambda^k - \lambda^*| \\
				&\leq \norm{F(h^k,\lambda^k) - F(h_{\lambda^k},\lambda^k)}_m + N^*(m)|\lambda^k - \lambda^*|\\
				&\leq \beta_m \norm{h^k - h_{\lambda^k}}_m + N^*(m)|\lambda^k - \lambda^*|\\
				&\leq \beta_m\alpha+ \alpha \frac{N^*(m)}{N_*(m)}\,.
			\end{aligned}
		\end{equation}
			This, in conjunction with $\lambda^k \leq \lambda^*,$  implies that
			\begin{equation}\label{eq-lam_k_leq_lam_*}
				\lambda^{k+1} = \lambda^k + \gamma_k h^{k+1}(n) \leq \lambda^* + \gamma_k \bigg(\alpha \beta_m + \alpha \frac{N^*(m)}{N_*(m)}\bigg)\,.
			\end{equation}
			Hence, 
			\begin{equation*}
				\lambda^{k+1} - \lambda^* \leq \gamma_k \alpha\bigg( \beta_m + \frac{N^*(m)}{N_*(m)}\bigg) \leq \frac{(1+\beta_m)\alpha}{N_*(m)}\,.
			\end{equation*}
			To get the second inequality above, we use the fact that $\gamma_k$ satisfies~\eqref{eq-gam-bound}. Additionally, using \eqref{eq-lambda_opt_leq_lambda_iter}, we see that
			\begin{equation*}
				\lambda^* - \lambda^{k+1} \leq \frac{(1+\beta_m)\alpha}{2N_*(m)}\,.
			\end{equation*}
			As a result, we obtain~\eqref{eq-l-bound-1}.	
			This completes the proof for {Case (a)}.
			
			{\noindent \bf Case (b):}	The assumption that $\|h^k - h_{\lambda^k}\|_m \leq \alpha$  and the bound on $|h^{k+1}(n)-h_{\lambda^k}(n)|$ from~\eqref{eq-h_k+1_h_lambda_k_contract_prop} give us 
			\begin{equation}\label{eq-h_lam_k_leq_h_k+1}
				h_{\lambda^k}(n) \leq h^{k+1}(n) + \alpha \beta_m \leq \alpha\beta_m,
			\end{equation}
			where the last inequality holds because $h^{k+1}(n) < 0$ by assumption. However, recall that $\overline \lambda$ was chosen so that $h_{\overline \lambda}(n) = \alpha \beta_m.$ Since $\lambda\mapsto h_\lambda(n)$ is monotonically decreasing, it follows that $\lambda^k \geq \overline \lambda.$ Since $\lambda^k \leq \lambda^*$ and $\lambda\mapsto h_\lambda(n)$ is monotonically decreasing,  we get $0=h_{\lambda^*}(n) \leq h_{\lambda^k}(n)$.  Hence, we get $0 \leq h_{\lambda^k}(n) \leq \alpha \beta_m$. Combining this with  \eqref{eq-h_lam_k_leq_h_k+1}  and re-arranging the resulting inequality, we get 
			\begin{equation}\label{eq-h_k+1_bound}
				-\alpha \beta_m \leq h^{k+1}(n) \leq 0\,.
			\end{equation}
			We use this to construct the desired upper bound for $\lambda^* - \lambda^{k+1}.$ This is sufficient because $\lambda^{k+1} = \lambda^k + \gamma_k h^{k+1}(n) < \lambda^k$ as $h^{k+1}(n) < 0$, and $\lambda^k \leq \lambda^*$ by assumption. From \eqref{eq-h_k+1_bound} and~\eqref{eq-gam-bound}, it follows that 
			$$|\gamma_kh^{k+1}(n)| \leq \frac{(1-\beta_m)\alpha}{2N^*(m)}\,. $$
			Combining this with our bound on the difference of $\overline \lambda$ and $\widehat \lambda$ in the fifth inequality of~\eqref{eq-tilde_lam-overline_lam}, we have
			$$ |\gamma_k h^{k+1}(n) | \leq \overline \lambda -\widehat \lambda \leq \lambda^k - \widehat \lambda\, ,$$ 
			where $\lambda^k \geq \overline \lambda$ was shown in the paragraph following \eqref{eq-h_lam_k_leq_h_k+1}. Following the definition of $\lambda^{k+1}$, this leads to $\lambda^{k+1} = \lambda^k + \gamma_k h^{k+1}(n) \geq \widehat \lambda$. As a result, we can apply the bound on $\lambda^* - \widehat \lambda$ from the fourth inequality in~\eqref{eq-tilde_lam-overline_lam} to obtain	
			\begin{equation*}
				\lambda^* - \lambda^{k+1} \leq \lambda^* - \widehat \lambda \leq \frac{(1+\beta_m)\alpha}{2N^*(m)} \leq c_2( m)\frac{\alpha}{N_*(m)}\,.
			\end{equation*}
			This completes the proof for Case (b) as well as the proof for the lemma. 
\end{proof}

{\noindent \bf Case: $\lambda^k>\lambda^*$.} 
The proof in this case involves similar arguments as those used in the case where $\lambda^k\leq \lambda^*$. However, we only discuss one of the main differences. We restrict ourselves to using the same notation as in the previous case.  In this case, we define our intermediate values $\overline \lambda$ and  $\underline \lambda$ in a similar manner to Lemmas~\ref{lem-lambda-intermediate} and~\ref{lemma-geom_conv_lambda_identities} but with the following modification: $\overline \lambda$ and  $\underline \lambda$ are the unique real numbers such that $\underline \lambda$, $\overline \lambda\in \Lambda_m$ satisfying $h_{\underline \lambda}(n) = - \alpha \beta_m$ and $h_{\overline \lambda}(n) = -\alpha$. In this case, we again define $\widehat \lambda \doteq \frac{1}{2}(\overline \lambda + \underline \lambda).$ From these definitions, it follows that $\lambda^* < \underline \lambda < \widehat \lambda < \overline \lambda$, and these values satisfy the same  inequalities as those given in Lemma~\ref{lemma-geom_conv_lambda_identities} (but with $\lambda^*-\overline \lambda$, $\lambda^*-\underline\lambda$ and $\lambda^*-\widehat \lambda$ replaced by $\overline \lambda- \lambda^*$, $\underline \lambda-\lambda^*$ and $\widehat \lambda-\lambda^*$, respectively).  

From here, the rest of the proof is split into two cases depending on whether $\lambda^k\in (\lambda^*,\widehat \lambda)$ or not. In the case where $\lambda^k\notin (\lambda^*,\widehat \lambda)$, we obtain a result analogous to Lemma~\ref{lemma-geom_conv_lambda} and in the case where  $\lambda^k\in (\lambda^*,\widehat \lambda)$, we obtain a result analogous to Lemma~\ref{lemma-geom_conv_lambda-1}. The  proofs of these results follow closely the main arguments of the   proofs of Lemmas~\ref{lemma-geom_conv_lambda} and~\ref{lemma-geom_conv_lambda-1}, with minor changes. Hence, we omit them.

From Lemmas~\ref{lemma-geom_conv_lambda} and~\ref{lemma-geom_conv_lambda-1}, and their analogs for the case where $\lambda^k>\lambda^*$, whenever 
$$\gamma_k\leq  \min\Big\{\frac{1}{N^*(m)},\,\,  \frac{1+\beta_m}{\beta_mN_*(m)+N^*(m)},\,\,  \frac{1-\beta_m}{2\beta_mN^*(m)}\Big\},$$
we have  $|\lambda^{k+1} - \lambda^*| \leq \widetilde c_\lambda(\gamma_k, m) \frac{\alpha}{N_*(m)},$ where $\widetilde c_\lambda(\gamma,m)\doteq \max\{ c_1(\gamma,m),c_2(m)\}$. Moreover, if 
$$ \gamma\leq  \widehat \gamma\doteq  \min\Big\{\frac{1}{\beta_mN_*(m)}, \,\,\frac{2N^*(m)}{(1-\beta_m) N_*(m)^2}, \,\, \Big(\frac{m-m^*}{ N_*(m)}\Big)\Big( \frac{1}{\frac{\alpha_0}{N_*(m)}+\overline c-\underline c+\varkappa(m)}\Big)\Big\},$$
then $\widetilde c_\lambda(\gamma,m)<1$ and from  Proposition~\ref{prop-c-alpha0}, we know that $\lambda^{k+1}\in \Lambda$ and in fact $\lambda^{k+1}\in \Lambda_m$. In particular, $h_{\lambda^{k+1}}$ exists. In the following lemma, we estimate $\|h^{k+1}-h_{\lambda^k}\|_m$ in terms of $\alpha$.
\begin{lemma}\label{lemma-geom_conv_h_k}
Suppose~\eqref{eq-alpha2} holds for some $\alpha>0$. Then, 
\begin{equation}\label{eq-h-comp}
	\norm{h^{k+1} - h_{\lambda^{k+1}}}_m \leq \widetilde c_h(\gamma_k, m) \alpha,
\end{equation}
where 
$$ \widetilde c_h(\gamma,m)\doteq \beta_m+ \gamma \big(\|e\|_m+\frac{\beta_m N^*(m)}{w^0_1}\big) \Big( \beta_m+ \frac{N^*(m)}{N_*(m)}\Big)\,.$$
Recall $w^0_1$ from~\eqref{eq-weight_vector}, $\beta_m$ from~\eqref{eq-contraction_const} and $N_*(m)$ is as defined in the second line of~\eqref{def-m}.
\end{lemma}
\begin{proof}
From~\eqref{eq-F-itr}, we have 
\begin{equation}\label{eq-bounds_h_value_func}
	\begin{aligned}
		\| h^{k+1} - h_{\lambda^{k+1}}\|_m &\leq \|(\lambda^{k+1} - \lambda^k)e\|_m + \|F(h^k,\lambda^{k+1}) - F(h_{\lambda^{k+1}},\lambda^{k+1}) \|_m\\
		&\leq |\lambda^{k+1} - \lambda^k| \|e\|_m  + \beta_m \|h_{\lambda^k} - h_{\lambda^{k+1}} \|_m+\beta_m \|h^k - h_{\lambda^k}\|_m\,.
	\end{aligned}
\end{equation} 
In the above, we apply Proposition~\ref{thm-contraction} and the triangle inequality. We now estimate the first two  terms. 
Consider $|\lambda^{k+1} - \lambda^k|$. Using the definition of $\lambda^{k+1}$ from~\eqref{eq-F-itr}, 
we see that 
\begin{equation*}
	\begin{aligned}
		|\lambda^{k+1} - \lambda^k| = \gamma_k | h^{k+1}(n)|& \leq \gamma_k\bigg(|h^{k+1}(n) - h_{\lambda^k}(n)| + |h_{\lambda^k}(n)|\bigg) \\
		&\leq \gamma_k\bigg(\norm{F(h^k,\lambda^k) - F(h_{\lambda^k},\lambda^k)}_m + N^*(m)|\lambda^k - \lambda^*|\bigg) \\
		&\leq \gamma_k\bigg(\beta_m \norm{h^k - h_{\lambda^k}}_m + N^*(m)|\lambda^k - \lambda^*|\bigg). 
	\end{aligned}
\end{equation*}
To arrive at the second line, we use the fact that $\|x_n\|\leq\|x\|_m$, for $x\in \RR^n$, Proposition~\ref{lem-bounds-h} for $h_{\lambda^k}(n)$, and $h^*(n)$, and the definition of $h^{k+1}$ and $h_{\lambda^k}$. To get the last line, we apply Proposition~\ref{thm-contraction}.

From~\eqref{norm-rel},  Proposition~\ref{lem-bounds-h} and the fourth inequality in~\eqref{eq-bounds_lam_iterates}, we have
\begin{equation*}
	\begin{aligned}
		\|h_{\lambda^k} - h_{\lambda^{k+1}} \|_m&\leq \frac{N^*(m)}{w^0_1}|\lambda^{k+1}-\lambda^{k}|\leq  \frac{\gamma_kN^*(m)}{w^0_1}\bigg(\beta_m \norm{h^k - h_{\lambda^k}}_m + N^*(m)|\lambda^k - \lambda^*|\bigg)\,.
	\end{aligned}
\end{equation*} 
Combining the inequalities in~\eqref{eq-bounds_h_value_func} and~\eqref{eq-bounds_lam_iterates} with the above inequality, we obtain 
\begin{equation*}
	\begin{aligned}
		\| h^{k+1} - h_{\lambda^{k+1}}\|_m & \leq \beta_m\norm{h^k - h_{\lambda^k}}_m +  \gamma_k \|e\|_m\bigg(\beta_m \norm{h^k - h_{\lambda^k}}_m + N^*(m)|\lambda^k - \lambda^*|\bigg)\\
		&\qquad + \gamma_k\beta_mN^*(m)\bigg(\beta_m \norm{h^k - h_{\lambda^k}}_m + N^*(m)|\lambda^k - \lambda^*|\bigg)\\
		&\leq \beta_m \alpha + \gamma_k \|e\|_m \bigg(\beta_m \alpha  + N^*(m)\frac{\alpha}{N_*(m)}\bigg)+   \frac{\gamma_kN^*(m)}{w^0_1}\bigg(\beta_m \alpha  + N^*(m)\frac{\alpha}{N_*(m)}\bigg)\\
		&= \widetilde c_h(\gamma_k,m) \alpha\,.
	\end{aligned}
\end{equation*}
In the second inequality, we use the first two inequalities in~\eqref{eq-alpha2}. This proves the first part of~\eqref{eq-h-comp}. This completes the proof of the lemma. 
\end{proof}

From  the above analysis, and the  definitions of $\widetilde c_h$ and $\widetilde c_\lambda$, it is easy to see  that whenever
$$ \gamma_k\leq \widetilde \gamma\doteq  \min\Bigg\{\frac{1}{\beta_mN_*(m)}, \frac{2N^*(m)}{(1-\beta_m) N_*(m)^2}, \frac{(1-\beta_m)N_*(m)}{\big(\|e\|_m+\beta_m N^*(m) \big) \big(\beta_mN_*(m)+N^*(m)\big)},\widehat  \gamma\Bigg\},$$
we have $\widetilde c_h(\gamma_k,m)<1$ and $\widetilde c_\lambda(\gamma_k,m)<1$.  This completes the proof of Proposition~\ref{prop-c-alpha}. \hfill $\Box$

\medskip

\section{Proof of Proposition~\ref{prop-contraction-GS}}\label{app-proof-contraction-GS}
\begin{proof}
	We proceed by using the contraction property of $F(\cdot,\lambda)$ (obtained in Proposition~\ref{thm-contraction}) and inductively showing that 
\begin{equation}\label{eq-G_contract}
	\frac{|G_i(h_1,\lambda_1) - G_i(h_2,\lambda_2)|}{w_i^m} \leq \beta_m \norm{h_1 - h_2}_m + \Delta_i^m|\lambda_1 - \lambda_2|
\end{equation}
for $1\leq i\leq n$. From here, we use the definition of $\|\cdot\|_m$ in~\eqref{def-norm}  to obtain \eqref{eq-contraction-GS}. 

\noindent \textbf{Initialization step:} We prove~\eqref{eq-G_contract} for $i=1$.
Observe that $G_1(h_1,\lambda_1) = F_1(h_1,\lambda_1).$ Furthermore, the local contraction property of $F$ (Proposition~\ref{thm-contraction}) implies that 
\begin{equation}\label{eq-int-step-1}\frac{|F_1(h_1,\lambda_1) - F_1(h_2,\lambda_1)|}{w_1^m} \leq \beta_m \norm{h_1 - h_2}_m\,.\end{equation} From the definition of $G_1(\cdot,\cdot)$, we have 
$$ G_1(h_1, \lambda_1) - G_1(h_2, \lambda_1) = F_1(h_1,\lambda_1) - F_1(h_2,\lambda_1),$$
and using~\eqref{eq-int-step-1}, we get
$$\frac{|G_1(h_1,\lambda_1) - G_1(h_2,\lambda_1)|}{w_1^m} \leq \beta_m \|h_1 - h_2\|_m\,. $$	
Using the definitions of $G_1(\cdot,\cdot)$ and $\Delta_1^m$, and the above display, it is now clear that
\begin{equation*}
	\begin{aligned}
		\frac{G_1(h_1,\lambda_1)}{w_1^m}   &\leq \frac{G_1(h_2,\lambda_1)}{w_1^m}+\beta_m  \|h_1 - h_2\|_m\\
		&\leq \frac{G_1(h_2,\lambda_2)}{w_1^m}+\beta_m \|h_1 - h_2\|_m + \Delta_1^m{|\lambda_1-\lambda_2|}\,.
	\end{aligned}
\end{equation*} 
Interchanging the roles of $(h_1,\lambda_1)$ and $(h_2,\lambda_2)$, we obtain
\begin{equation*}
	\frac{G_1(h_2,\lambda_2)}{w_1^m}   \leq  \frac{G_1(h_1,\lambda_1)}{w_1^m}+\beta_m \|h_1 - h_2\|_m + \Delta_1^m |\lambda_1-\lambda_2|\,.
\end{equation*}
Combining  the above two displays proves~\eqref{eq-G_contract} for $i=1$.

\noindent \textbf{Induction step: } Suppose that \eqref{eq-G_contract} holds for $1\leq i\leq  r-1$, where $r<n$. We show that it also holds for $r$. Let $v^*\in \Usm$ be a minimizing control policy for the right hand side of $G_r(h_2, \lambda_1).$ Using an argument analogous to that in the proof of Proposition~\ref{thm-contraction}, it follows that
\begin{equation*}
	\begin{aligned}
		&\frac{G_r(h_1,\lambda_1)-G_r(h_2,\lambda_1)}{w_r^m}\\
		&\quad= \frac{1}{w_r^m}\min_{u \in \bU(r)} \Big[ c(r,u) + \log \Big(p(r,n,u) + \sum_{j=1 }^{r-1} e^{G_j(h_1,\lambda_1)}p(r,j,u)  +\sum_{j=r}^{n-1}e^{h_1(j)}p(r,j,u)\Big) \Big]\\
		&\quad\qquad-\frac{1}{w_r^m}\min_{u \in \bU(r)} \Big[ c(r,u) + \log \Big(p(r,n,u) + \sum_{j=1 }^{r-1} e^{G_j(h_2,\lambda_1)}p(r,j,u)  +\sum_{j=r}^{n-1}e^{h_2(j)}p(r,j,u)\Big) \Big]\,\\
		&\quad \leq  \frac{1}{w_r^m}\log \bigg( p(r,n,v^*(r)) + \sum_{j=1}^{r-1}e^{G_j(h_1,\lambda_1)}p(r,j,v^*(r)) + \sum_{j=r}^{n-1}e^{h_1(j)}p(r,j,v^*(r))\bigg) \\
		&\quad\qquad -\frac{1}{w_r^m}\log \bigg( p(r,n,v^*(r)) + \sum_{j=1}^{r-1}e^{G_j(h_2,\lambda_1)}p(r,j,v^*(r)) + \sum_{j=r}^{n-1}e^{h_2(j)}p(r,j,v^*(r))\bigg).
	\end{aligned}
\end{equation*}
From here, following similar calculations as those leading up to~\eqref{eq-contraction-4}, we have
\begin{equation}\label{eq-GS_contraction_induct_3}
	\begin{aligned}
		&\frac{G_r(h_1,\lambda_1)-G_r(h_2,\lambda_1)}{w_r^m} \\
		& \qquad \leq \frac{1}{w_r^m}\bigg(\sum_{j=1}^{r-1}[G_j(h_1,\lambda_1) - G_j(h_2, \lambda_1)]q^*(r,j,v^*(r)) + \sum_{j=r}^{n-1}[h_1(j) - h_2(j)]q^*(r,j,v^*(r))\bigg),
	\end{aligned}
\end{equation}  	
with $$ q^*(r,j,v^*(r)) = \frac{e^{\tilde h_1(j)}p(r,j,v^*(r))}{\sum_{k=1}^n e^{\tilde h_1(k)}p(r,k,v^*(r))}\,.$$
Here, $\tilde h_{1}: \RR^n\rightarrow   \R^n, $ as follows
$$ \tilde h_{1}(j) = \begin{cases}
	G_j(h_1, \, \lambda_1), &j \leq r-1,\\
	h_1(j),\, &r \leq j \leq n-1,\\
	0, \, &j= n\,.
\end{cases}$$
Similarly, $\tilde h_{2}$ is defined. From the hypothesis of the theorem and Lemma~\ref{lem-bound-G}, we have $\|\widetilde h_1\|_\infty\leq m$. Therefore,  $q^*(r,\cdot, v^*(r)) \geq e^{-2m}p(r,\cdot,v^*(r))$.

Now following similar calculations as those leading up to~\eqref{eq-contraction_ineq_1}, we get
\begin{equation*}
	\begin{aligned}
		\frac{G_r(h_1,\lambda_1)-G_r(h_2,\lambda_1)}{w_r^m} &\leq \beta_m \max_{1 \leq i \leq n}\frac{|\tilde h_1(i) - \tilde h_2(i)|}{w_i^m} \leq \beta_m \| h_1 - h_2\|_m\,.
	\end{aligned}
\end{equation*} 
Interchanging the roles of $h_1$ and $h_2$, we get 
\begin{equation*}
	\frac{G_r(h_2,\lambda_1)-G_r(h_1,\lambda_1)}{w_r^m} \leq \beta_m \| h_1 - h_2\|_m\,.
\end{equation*}
From the above two displays, we have
\begin{equation*}
	\frac{|G_r(h_1,\lambda_1)-G_r(h_2,\lambda_1)|}{w_r^m} \leq \beta_m \| h_1 - h_2\|_m\,.
\end{equation*}
Recall that the induction hypothesis implies
\begin{equation}\label{eq-G-0} \frac{|G_i(h_2,\lambda_1) - G_i(h_2,\lambda_2)|}{w_i^m} \leq  \Delta_i^m|\lambda_1 - \lambda_2|, \quad 1 \leq i \leq r-1\,.\end{equation}
Using this, we next obtain a bound on $\frac{|G_r(h_2,\lambda_1)-G_r(h_2,\lambda_2)|}{w^m_r}$. To do this, we again follow similar arguments as those leading up to~\eqref{eq-GS_contraction_induct_3} and obtain 
\begin{equation}\nonumber
	\begin{aligned}
		\frac{G_r(h_2,\lambda_1)-G_r(h_2,\lambda_2)}{w_r^m} &\leq  \frac{1}{w_r^m}\bigg(\sum_{j=1}^{r-1}[G_j(h_2,\lambda_1) - G_j(h_2, \lambda_2)]\widehat q(r,j,v^*(r)) + (\lambda_2-\lambda_1)\bigg)\\\nonumber
		&\leq \frac{1}{w_r^m}\bigg(\sum_{j=1}^{r-1}|G_j(h_2,\lambda_1) - G_j(h_2, \lambda_2)|\widehat q(r,j,v^*(r)) + |\lambda_2-\lambda_1|\bigg)\\\nonumber
		&\leq \frac{1}{w_r^m}\bigg(\sum_{j=1}^{r-1}\Delta_i^m|\lambda_1 - \lambda_2|\widehat q(r,j,v^*(r)) + |\lambda_2-\lambda_1|\bigg)\\\nonumber
		&\leq \Delta^m_r|\lambda_2-\lambda_1|\, .
	\end{aligned}
\end{equation}
In the above,  $\widehat q(\cdot,\cdot,\cdot)$ is some  transition probability (unlike earlier, the exact form of $\widehat q(\cdot,\cdot,\cdot)$ is irrelevant for our purpose).  To get the third line, use~\eqref{eq-G-0} and to get the fourth line, we use the definition of $\Delta^m_r$ and the fact that $\sum_{j=1}^{r-1}\widehat q(r,j,v^*(r))\leq 1$.  Interchanging the roles of $\lambda_1$ and $\lambda_2$, we obtain 
\begin{equation}\nonumber
	\frac{G_r(h_2,\lambda_1)-G_r(h_2,\lambda_2)}{w_r^m}  		\leq \Delta^m_r|\lambda_2-\lambda_1|, 
\end{equation}
which gives us 
$$ \frac{|G_r(h_2,\lambda_1)-G_r(h_2,\lambda_2)|}{w_r^m} \leq \Delta^m_r|\lambda_2-\lambda_1|\,.$$
From here, 	it follows that
\begin{equation}\label{eq-G-2}
	\begin{aligned}
		\frac{G_r(h_1,\lambda_1)}{w^m_r}&\leq \frac{G_r(h_2,\lambda_1)}{w^m_r} + \beta_m\|h_1-h_2\|_m\\
		&\leq  \frac{G_r(h_2,\lambda_2)}{w^m_r} + \frac{|G_r(h_2,\lambda_1)-G_r(h_2,\lambda_2)|}{w^m_r} \\
		&\qquad  + \beta_m\|h_1-h_2\|_m\\
		&\leq \frac{G_r(h_2,\lambda_2)}{w_r^m} + \beta_m\|h_1-h_2\|_m + \Delta^m_r|\lambda_1-\lambda_2|\,.
	\end{aligned}
\end{equation}
Again, interchanging the roles of $(h_1,\lambda_1)$ and $(h_2,\lambda_2)$, we obtain a bound similar to the one in~\eqref{eq-G-2}, which, combined with~\eqref{eq-G-2}, gives us~\eqref{eq-G_contract} for $i=r$. This completes the proof of the proposition.
\end{proof}

				\end{document}